\documentclass[12pt, a4paper]{elsarticle}

\usepackage{fullpage}
\usepackage{amsmath}
\usepackage{amssymb}
\usepackage{amsthm}
\usepackage{xypic}
\xyoption{curve}
\xyoption{color}
\usepackage{graphicx}
\usepackage{graphicx}
  \usepackage[pdftex,
  bookmarks = false,
  pdfstartview = FitBH,
  linktocpage = false,
  pagebackref = false,
  pdfhighlight = /O,
  colorlinks=false,
  linkcolor=blue,
  citecolor=blue,
  filecolor =blue,
  urlcolor =blue,
  menucolor =blue,
]{hyperref}

\theoremstyle{definition}
\newtheorem{thm}{Theorem}

\newtheorem{cor}[thm]{Corollary}
\newtheorem{prop}[thm]{Proposition}

\newtheorem{defn}[thm]{Definition}
\newtheorem{Q}[thm]{Question}

\newtheorem{rmk}[thm]{Remark}

\journal{Annals of Pure and Applied Logic}

\begin{document}

\begin{frontmatter}

%% Title, authors and addresses

%% use the tnoteref command within \title for footnotes;
%% use the tnotetext command for the associated footnote;
%% use the fnref command within \author or \address for footnotes;
%% use the fntext command for the associated footnote;
%% use the corref command within \author for corresponding author footnotes;
%% use the cortext command for the associated footnote;
%% use the ead command for the email address,
%% and the form \ead[url] for the home page:
%%

%% \title{Title\tnoteref{label1}}
%% \tnotetext[label1]{}
%% \author{Name\corref{cor1}\fnref{label2}}
% \ead{swalsh108@gmail.com or s.walsh@bbk.ac.uk}
% \ead[url]{http://www.swalsh108.org}
%% \fntext[label2]{}
%% \cortext[cor1]{}
%% \address{Address\fnref{label3}}
%% \fntext[label3]{}

\title{Comparing Peano Arithmetic, Basic Law~V, and Hume's Principle}

%% use optional labels to link authors explicitly to addresses:
%% \author[label1,label2]{<author name>}
%% \address[label1]{<address>}
%% \address[label2]{<address>}

\author{Sean Walsh}

\address{\href{http://www.bbk.ac.uk/philosophy/}{Department of Philosophy}, Birkbeck, University of London \\ Mailing Address: Department of Philosophy, Birkbeck College, Malet Street, London WC1E 7HX, UK \\ Email: \href{mailto:swalsh108@gmail.com}{swalsh108@gmail.com} or \href{mailto:s.walsh@bbk.ac.uk}{s.walsh@bbk.ac.uk} \\ Website: \href{http://www.swalsh108.org}{http://www.swalsh108.org}}

\begin{abstract}
This paper presents new constructions of models of Hume's Principle and Basic Law~V with restricted amounts of comprehension. The techniques used in these constructions are drawn from hyperarithmetic theory and the model theory of fields, and formalizing these techniques within various subsystems of second-order Peano arithmetic allows one to put upper and lower bounds on the interpretability strength of these theories and hence to compare these theories to the canonical subsystems of second-order arithmetic. The main results of this paper are: (i)~there is a consistent extension of the hypearithmetic fragment of Basic Law V which interprets the hyperarithmetic fragment of second-order Peano arithmetic (cf. Corollary~\ref{HPPA:thm:main:cor} and Figure~\ref{HPPA:figure2}), and (ii)~the hyperarithmetic fragment of Hume's Principle does not interpret the hyperarithmetic fragment of second-order Peano arithmetic (cf. Corollary~\ref{HPPA:cor:thebomb} and Figure~\ref{HPPA:figure2}), so that in this specific sense there is no predicative version of Frege's Theorem.
\end{abstract}

\begin{keyword}
Second-order arithmetic \sep Basic Law~V \sep Hume's Principle \sep hyperarithmetic \sep recursively saturated \sep interpretability

%% keywords here, in the form: keyword \sep keyword

\MSC 03F35 \sep 03D65 \sep 03C60 \sep 03F25

%% MSC codes here, in the form: \MSC code \sep code
%% or \MSC[2008] code \sep code (2000 is the default)

\end{keyword}

\end{frontmatter}

%%
%% Start line numbering here if you want
%%
% \linenumbers

\newpage

\tableofcontents

\newpage

\section{Introduction, Definitions, and Overview of Main Results}

% what is currently missing:
% why should we care about HP and Basic~Law~V?
% need to say something about why philosophers of mathematics have been interested in these theories.
% why should we 

\subsection{Introduction}

Second-order Peano arithmetic and its subsystems have been studied for many decades by mathematical logicians  (cf. \cite{Simpson2009aa}), and the resulting theory continues to be the subject of current research and a source of open problems. More recently, philosophers of mathematics have begun to study systems closely related to  second-order Peano arithmetic (cf. \cite{Burgess2005}). One of these systems, namely, Hume's~Principle, constitutes an axiomatization of~cardinality which is similar to the notion of cardinality defined in Zermelo-Frankel set theory. The contemporary philosophical interest in this principle stems from Crispin Wright's suggestion that it can serve as the centerpiece of a revitalized version of Frege's logicism (cf. \cite{Wright1983}, \cite{Wright1999}, \cite{MacBride2003aa}). Frege himself focused his logicism around a principle called Basic Law~V, which in effect codified an alternative conception of set. While Russell's paradox shows that Basic Law~V is inconsistent with the unrestricted comprehension schema (cf. Proposition~\ref{HPPA:form:thebadformtheprop}), this principle has garnered renewed attention due to Ferreira and Wehmeier's recent proof that it is consistent with the hyperarithmetic comprehension schema (\cite{Ferreira2002}, cf. \cite{Wehmeier1999,Wehmeier2004aa} and Remark~\ref{hppa:rmk:fw}).

The goal of this paper is to apply methods from the subsystems of second-order Peano arithmetic to the subsystems of Basic Law~V and Hume's~Principle. In particular, we use methods from hyperarithmetic theory to build models of subsystems of Basic Law~V (\S~\ref{HPPA:secdfadsfsadfkkk}), and we use recursively saturated models and ideas from the model theory of fields to build models of subsystems of Hume's~Principle and Basic~Law~V (\S~\ref{HPPA:bssssssss}). Our primary application of these new constructions is to compare the interpretability strength of the subsystems of second-order Peano arithmetic to the subsystems of Basic~Law~V and Hume's~Principle. For, one of the few known ways to show that one theory is of {\it strictly} greater interpretability strength than another theory  is to show that the first proves the consistency of the second (cf. Proposition~\ref{HPPA:conpropa}).  Hence, by formalizing our constructions, we can compare the interpretability strength of subsystems of Hume's~Principle and Basic~Law~V to subsystems of Peano arithmetic. Our main results about interpretability are summarized in \S~\ref{HPPA:mainresultsasas} and on Figure~\ref{HPPA:figure2}. Prior to summarizing these results, we first present formal definitions of the theories and subsystems of Hume's Principle and Basic Law~V (\S\S~\ref{HPPA:sec:342314213}-\ref{HPPA:sec:342314213ddd}) and then describe what is and is not known about the provability relation among these subsystems (\S~\ref{HPPA:sec:342314213ddd} and Figure~\ref{HPPA:figure1}).

\subsection{Definition of the Signatures and Theories of ${\tt PA}^{2}$, ${\tt BL}^{2}$ and ${\tt HP}^{2}$}\label{HPPA:sec:342314213}

The signature of ${\tt PA}^{2}$ is a many-sorted signature, with sorts for numbers as well as a sort for sets of numbers. The theory  ${\tt PA}^{2}$ is a natural set of axioms for the following many-sorted structure in this signature:
\begin{equation}\label{HPPA:eqn:standardmodel}
(\omega, 0,s,+,\times, \leq, P(\omega))
\end{equation}
This structure satisfies the eight-axioms of Robinson's {\tt Q}
\begin{tabbing}
(Q1) $sx\neq 0$ \;\;\;\;\;\;\;\;\;\;\;\;\;\;\;\;\;\;\;\;   \=  (Q2) $sx=sy\rightarrow x=y$  \;\;\;\;\;\;\;\;\;\;\;\;\;\;\;\;\;\;\;\;    \= (Q3) $x\neq 0\rightarrow \exists \; w \; x=sw$   \\
(Q4) $x+0=x$    \>  (Q5) $x+sy=s(x+y)$  \>   (Q6) $x\cdot 0 =0$ \\
  (Q7) $x\cdot sy=x\cdot y +x$  \> (Q8) $x\leq y \leftrightarrow \exists \; z \; x+z=y$  \>  
\end{tabbing}
and the mathematical induction axiom 
\begin{equation}\label{HPPA:eqn:MIaxiom}
\forall \; F \; [F(0) \; \& \; \forall \; n \; F(n)\rightarrow F(s(n)) ] \rightarrow [\forall \; n \; F(n)]
\end{equation}
and each instance of the comprehension schema  (where $F$ does not occur free in $\varphi$)
\begin{equation}\label{HPPA:eqn:pacompsche}
\exists \; F \; \forall \; n \; [F(n) \leftrightarrow \varphi(n)]
\end{equation}
Here, the formula $\varphi$ is allowed to contain free object~variables (in addition to~$n$) and free set variables (with the exception of $F$). Hence, what an instance of this comprehension schema says is that if $\varphi(n)$ is a formula with parameters, then there is a set $F$ corresponding to it. This all in place, we are now in a position to define:

\begin{defn}\label{HPPA:defn:the:PA2}
The theory ${\tt PA}^{2}$ or ${\tt CA}^{2}$ or {\it second-order Peano arithmetic} consists of Q1-Q8, the mathematical induction axiom~(\ref{HPPA:eqn:MIaxiom}), and each instance of the comprehension schema~(\ref{HPPA:eqn:pacompsche}) (cf. \cite{Simpson2009aa} p.~4). \index{Second-order Peano Arithmetic,  ${\tt PA}^{2}$ or ${\tt CA}^{2}$, Definition~\ref{HPPA:defn:the:PA2}}
\end{defn}

\noindent The name ${\tt CA}^{2}$ is also given to ${\tt PA}^{2}$ because it reminds us of comprehension.

The signature of  ${\tt HP}^{2}$ and ${\tt BL}^{2}$ is likewise a many-sorted signature, with sorts for objects as well as sorts for $n$-ary relations on objects, and with an additional function symbol from the unary relation sort to the object sort. The unary relations are written as $A,B,C,F,G,H,X,Y,Z$ and will be called ${\it sets}$, and the $n$-ary relation symbols for $n>1$ are written as $f,g,h,P,Q,R,S$ and will be called {\it relations}. Occasionally when we want to say something about both sets and relations, we will talk about all $n$-ary relations for $n\geq 1$. The additional function symbol is denoted by $\#$ in the case of ${\tt HP}^{2}$ and by  $\partial$ in the case of ${\tt BL}^{2}$. So the signatures of ${\tt HP}^{2}$ and ${\tt BL}^{2}$ are exactly the same: it is merely for the sake of convenience and clarity that we use  $\#$ in the context of ${\tt HP}^{2}$ and $\partial$ in the context of ${\tt BL}^{2}$. Hence, structures in this signature have the form
\begin{equation}\label{HPPA:eqn:standardmodel2}
(M, S_{1}, S_{2}, \ldots, \#)
\end{equation}
where $M$ is a set, $S_{n}\subseteq P(M^{n})$ and $\#:S_{1}\rightarrow M$. Note that the function $\#$ only goes from $S_{1}$ to $M$, so that the relations from $S_{n}$ for $n>1$ are not in the domain of this function.

It is worth pausing for a moment to dwell on a technical point. Formally, the signature of ${\tt PA}^{2}$ also contains a binary relation symbol~$E$~which holds between an object and a set and which, in the standard model from (\ref{HPPA:eqn:standardmodel}), is interpreted by the~$\in$~relation from the ambient set-theory. In structures where this holds, let us say that the symbol~$E$ is interpreted {\it absolutely}. It is easy to see that every structure in the signature of ${\tt PA}^{2}$ is isomorphic to a structure that interprets this symbol absolutely, and it is for this reason that this symbol is typically suppressed when describing structures. Likewise, formally the signature of ${\tt HP}^{2}$ and ${\tt BL}^{2}$ contains $(n+1)$-ary relation symbols $E_{n}$, which hold between $n$-tuples of objects and $n$-ary relations. Further, there is an obvious generalization of the notion of absoluteness for structures in this signature, such that~the structure from~(\ref{HPPA:eqn:standardmodel2}) interprets~$E_{n}$ absolutely, and such that every structure in this signature is isomorphic to a structure which interprets~$E_{n}$ absolutely. Hence, as in the case of second-order Peano arithmetic, in what follows, these symbols will be suppressed when describing structures, and it will be assumed that every structure in this signature has the form of~(\ref{HPPA:eqn:standardmodel2}).

Hume's~Principle and Basic Law~V can now be defined. {\it Hume's~Principle} is the following axiom in the signature of structure~(\ref{HPPA:eqn:standardmodel2}):\index{Hume's Principle, the sentence, equation~(\ref{HPPA:eqn:introooooo:hp})}
\begin{equation}\label{HPPA:eqn:introooooo:hp}
\#X =\#Y \Longleftrightarrow \exists \; \mbox{ bijection } f:X\rightarrow Y
\end{equation}
Here, the notion of bijectivity is defined in terms of functionality, injectivity, and surjectivity in the obvious manner. The axiom {\it Basic~Law~V} is the following sentence in this signature:\index{Basic Law~V, the sentence, equation~(\ref{HPPA:eqn:introooooo:bl5})}
\begin{equation}\label{HPPA:eqn:introooooo:bl5}
\partial X =\partial Y \Longleftrightarrow X=Y
\end{equation}
Here, two sets are said to be equal if they are coextensive; formally, the equality of coextensive sets can be taken to be an axiom of all the theories considered in this paper. The important thing to note here is that $(M, S_{1}, S_{2}, \ldots, \partial)$ is a model of Basic Law~V if and only if the function $\partial:S_{1}\rightarrow M$ is an injection. That is, Basic Law~V mandates that a very simple relation holds between $S_{1}$ and $M$. There is no analogue of this in the case of Hume's~Principle, since the right-hand side of~(\ref{HPPA:eqn:introooooo:hp}) contains a higher-order quantifier.

Nevertheless, there are many natural models of Hume's~Principle, and examining these models is the easiest way to define the theories  ${\tt HP}^{2}$ and ${\tt BL}^{2}$. In particular, if $\alpha$ is an ordinal which is not a cardinal, and if  $\#$ is interpreted as cardinality, then the following structure is a model of Hume's~Principle:
\begin{equation}\label{HPPA:eqn:sm3}
(\alpha, P(\alpha), P(\alpha^{2}), \ldots, \#)
\end{equation}
Restricting attention to ordinals $\alpha$ that are not cardinals serves the purpose of ensuring that $\#(\alpha)<\alpha$, so that $\mathrm{dom}(\#_{\alpha})$ is $P(\alpha)$ and so that $\mathrm{rng}(\#_{\alpha})$ is a subset of $\alpha$. For all $n$-ary relation variables $R$ and all $n\geq 1$, this structure also satisfies each instance of the following comprehension schema (where $R$ does not occur free in $\varphi(\overline{z})$) 
\begin{equation}\label{HPPA:introoo:comp333}
\exists \; R \; \forall \; \overline{n} \; [\overline{n}\in R \leftrightarrow \varphi(\overline{n})]
\end{equation}
This comprehension schema is simply the generalization of the comprehension schema from ${\tt PA}^{2}$, namely (\ref{HPPA:eqn:pacompsche}), to the $n$-ary relations for all $n\geq 1$. Here, as with  (\ref{HPPA:eqn:pacompsche}), the formula $\varphi$ is allowed to include free object variables (in addition to $\overline{n}$) and free relation variables of any arity $m\geq 1$ (with the exception of $R$). Hence, we can now define the following theories:
\begin{defn}\label{HPPA:defn:the:HP2}
The theory ${\tt HP}^{2}$ is the theory that is given by Hume's~Principle (\ref{HPPA:eqn:introooooo:hp}) and the comprehension schema (\ref{HPPA:introoo:comp333}). \index{Hume's Principle, the theory, ${\tt HP}^{2}$, Definition~\ref{HPPA:defn:the:HP2}}
\end{defn}

\begin{defn}\label{HPPA:defn:the:BL2}
The theory ${\tt BL}^{2}$ is the theory which is given by Basic Law~V (\ref{HPPA:eqn:introooooo:bl5}) and the comprehension schema (\ref{HPPA:introoo:comp333}).\index{Basic Law~V, the theory, ${\tt BL}^{2}$, Definition~\ref{HPPA:defn:the:BL2}}
\end{defn}

The primary focus of this paper is on subsystems of ${\tt HP}^{2}$ and ${\tt BL}^{2}$ that are generated by restrictions on the complexity of the formulas appearing in the comprehension schema~(\ref{HPPA:introoo:comp333}). This is due to the fact that we seek to compare the interpretability strength of these subsystems to those of second-order Peano arithmetic. However, unlike in the case of ${\tt PA}^{2}$ and ${\tt HP}^{2}$, attention {\it must} be restricted to these subsystems in the case of ${\tt BL}^{2}$. For, it is not difficult to see that Russell's paradox shows that ${\tt BL}^{2}$ is inconsistent:

\begin{prop}\label{HPPA:form:thebadformtheprop}
${\tt BL}^{2}$ is inconsistent.
\end{prop}
\begin{proof}
By applying the comprehension schema~(\ref{HPPA:introoo:comp333}) to the formula
\begin{equation}\label{HPPA:form:thebadform}
\varphi(x)\equiv \exists \; Y \; \partial(Y)=x \; \& \; x\notin Y
\end{equation}
it follows that ${\tt BL}^{2}$ proves that there is set~$X$ that satisfies
\begin{equation}\label{HPPA:form:thebadfor3}
\forall \; x \;[x\in X\Longleftrightarrow (\exists \; Y \; \partial(Y)=x \; \& \; x\notin Y)]
\end{equation}
There are then two cases: either $\partial(X)\in X$ or $\partial(X)\notin X$. Case one: suppose that $\partial(X)\in X$. Then by the left-to-right direction of equation~(\ref{HPPA:form:thebadfor3}), it follows that there is $Y$ such that $\partial(Y)=\partial(X)$ and $\partial(X)\notin Y$. But $\partial(Y)=\partial(X)$ and Basic Law~V imply that $Y=X$, so that $\partial(X)\notin X$, which contradicts our case assumption. Case two: suppose that $\partial(X)\notin X$. Then by the right-to-left direction of equation~(\ref{HPPA:form:thebadfor3}), it follows that for any $Y$ we have that $\partial(Y) =\partial(X)$ implies $\partial(X)\notin Y$. But then $\partial(X)=\partial(X)$ implies $\partial(X)\notin X$, which contradicts our case assumption.
\end{proof}

\noindent Hence ${\tt BL}^{2}$ is inconsistent and does not have any models, unlike the theories ${\tt PA}^{2}$ and ${\tt HP}^{2}$, which respectively have the canonical models (\ref{HPPA:eqn:standardmodel}) and (\ref{HPPA:eqn:sm3}).

\subsection{Definition of the Subsystems of ${\tt PA}^{2}$, ${\tt BL}^{2}$ and ${\tt HP}^{2}$}

So if one wants to study Basic~Law~V, one needs to pass to subsystems of Basic~Law~V that do not allow instances of the comprehension schema~(\ref{HPPA:introoo:comp333}) applied to formulas like the one in (\ref{HPPA:form:thebadform}). To this end, let us introduce the following natural hierarchy of formulas in the signature of ${\tt BL}^{2}$ and ${\tt HP}^{2}$. A formula $\varphi$, perhaps with free object~variables $\overline{z}$ and free  relation~variables $\overline{R}$ of different arities $m\geq 1$, is called {\it arithmetical or $\Pi^{1}_{0}$ or $\Sigma^{1}_{0}$} if it does not contain any bound $m$-ary relation variables for any $m\geq 1$. Further, if $m\geq 1$ and $R$ is an $m$-ary relation variable and $\varphi(R)$ is a $\Sigma^{1}_{n}$-formula, then $\exists \; R \; \varphi(R)$ is a $\Sigma^{1}_{n}$-formula and $\forall \; R \; \varphi(R)$ is a $\Pi^{1}_{n+1}$-formula. Likewise, if $m\geq 1$ and $R$ is an $m$-ary relation variable and $\varphi(R)$ is $\Pi^{1}_{n}$-formula, then $\exists \; R \; \varphi(R)$ is a $\Sigma^{1}_{n+1}$-formula and $\forall \; R \; \varphi(R)$ is a $\Pi^{1}_{n}$-formula. 

That is, in this hierarchy of formulas, one is allowed to accumulate arbitrarily many existential relation quantifiers of different arities $m\geq 1$ in front of a~$\Sigma^{1}_{n}$-formula and still remain~$\Sigma^{1}_{n}$, and likewise one is allowed to accumulate arbitrarily many universal relation quantifiers of different arities $m\geq 1$ in front of a~$\Pi^{1}_{n}$-formula and still remain~$\Pi^{1}_{n}$. It is only the change from a universal relation quantifier of some arity $m\geq 1$ to an existential relation quantifier of some arity $m\geq 1$ (or vice-versa) which increases the complexity of the sentence in this hierarchy. For instance, if $X$ is set variable and $R$ and $S$ are binary relation variables, then the following formulas are respectively $\Sigma^{1}_{1}, \Pi^{1}_{1}, \Sigma^{1}_{2}, \Pi^{1}_{2}$:
\begin{align}
& \exists \; X \; \forall \; x \; R(x, \#X) \\
& \forall \; R \; \forall \; X \; \exists \; y \; [\forall \; x \; R(x,y)\rightarrow y=\partial X] \\
& \exists \; X \; \forall \; R \; [\exists \; x \; R(x,x) \rightarrow R(\#X, \#X)] \\
& \forall \; R \; \exists \; X \; \exists \; S \; \forall \; y \; [(\forall \; x \; x\in X \leftrightarrow \neg Sxy) \rightarrow R(\partial X, y)]
\end{align}
\noindent Finally, it is worth explicitly noting that not all formulas are included in our hierarchy of formulas. For instance, we have said nothing about the complexity of formulas which include alternations of object~quantifiers and set~quantifiers, such as the following formula:
\begin{equation}
\forall \; X \; \exists \; y \; \forall \; Z \; [R(\#X,\#Z)\rightarrow R(y,\#Z)]
\end{equation}
However, this is not a serious omission, since so long as one includes enough of the comprehension schema~(\ref{HPPA:introoo:comp333}) to guarantee the existence of the singleton set~$\{n\}$ for each element $n$, the above formula is equivalent to the following $\Pi^{1}_{3}$-formula
\begin{equation}
\forall \; X \; \exists \; Y \; \forall \; Z \; [\exists \; y\in Y \; \&\; \forall \; z\in Y \; z=y] \;\& \;[R(\#X,\#Z)\rightarrow R(y,\#Z)]
\end{equation}
That is, we can correct for this omission by treating object~quantifiers as set~quantifiers over singleton~sets when they occur in alternation of object~quantifiers and set~quantifiers.

Using this hierarchy of formulas, one can define the subsystems of ${\tt BL}^{2}$ and ${\tt HP}^{2}$ by restricting the complexity of formulas which appear in the comprehension schema~(\ref{HPPA:introoo:comp333}). For the following definition, let us recall that ${\tt CA}^{2}$ is another name for ${\tt PA}^{2}$ (cf. Definition~\ref{HPPA:defn:the:PA2}). The idea behind the following definition is then that ${\tt AC}$ reminds us of the axiom of choice and is the result of inverting the letters in ${\tt CA}$, which reminds us of comprehension. So with the exception of the choice schema, each of the schemas which figure in the below definition asserts the existence of a certain class of definable sets and relations:

\begin{defn}\label{hppa:index:subsystems}
Suppose that ${\tt XY}^{2}$ is one of ${\tt CA}^{2}$, ${\tt BL}^{2}$, or ${\tt HP}^{2}$. Then we can define the following four subsystems of ${\tt XY}^{2}$:

\noindent (i)~The subsystem ${\tt AXY_{0}}$ is ${\tt XY}^{2}$ but with the comprehension scheme~(\ref{HPPA:introoo:comp333}) restricted to arithmetical formulas.\index{Arithmetical subsystem ${\tt AXY_{0}}$ of ${\tt XY}^{2}$, Definition~\ref{hppa:index:subsystems}~(i)}

\noindent  (ii)~The subsystem  ${\tt \Delta^{1}_{1}-XY_{0}}$ is ${\tt XY}^{2}$ but with the comprehension scheme~(\ref{HPPA:introoo:comp333}) replaced by the following schema, which is called the {\it ${\tt \Delta^{1}_{1}}$-comprehension schema} or {\it hyperarithmetic comprehension schema}, wherein $\varphi$ is a $\Sigma^{1}_{1}$-formula and $\psi$ is a $\Pi^{1}_{1}$-formula:\index{Hyperarithmetical subsystem ${\tt \Delta^{1}_{1}-XY_{0}}$ of ${\tt XY}^{2}$, Definition~\ref{hppa:index:subsystems}~(ii)}
\begin{equation}\label{HPPA:eqn:hypcompscheme}
[\forall \; \overline{n}\; \varphi(\overline{n})\leftrightarrow \psi(\overline{n})] \rightarrow \; [\exists \; R \; \forall \; \overline{n} \; \overline{n} \in R \leftrightarrow \varphi(\overline{n})]
\end{equation}

\noindent  (iii)~The subsystem  ${\tt \Sigma^{1}_{1}- YX_{0}}$ is ${\tt AXY_{0}}$ and the following schema, which is called the {\it ${\tt \Sigma^{1}_{1}}$-choice schema}, wherein $\varphi$ is a $\Sigma^{1}_{1}$-formula:\index{$\Sigma^{1}_{1}$-choice subsystem ${\tt \Sigma^{1}_{1}- YX_{0}}$  of ${\tt XY}^{2}$, Definition~\ref{hppa:index:subsystems}~(iii)}
\begin{equation}\label{HPPA:eqn:thechoiceschema}
[\forall \; \overline{n} \; \exists \; P \; \varphi(\overline{n},P)] \rightarrow [\exists \; R\; \forall \; \overline{n} \; \forall \; P \; (\forall \; \overline{m}\; (\overline{m}\in P \leftrightarrow \overline{n}\overline{m}\in R)) \rightarrow\varphi(\overline{n}, P)]
\end{equation}

\noindent  (iv)~The subsystem  ${\tt \Pi^{1}_{n}- XY_{0}}$ is ${\tt XY}^{2}$ but with the comprehension schema~(\ref{HPPA:introoo:comp333}) restricted to ${\tt \Pi^{1}_{n}}$-formulas.\index{$\Pi^{1}_{n}$-choice subsystem ${\tt \Pi^{1}_{n}- XY_{0}}$   of ${\tt XY}^{2}$, Definition~\ref{hppa:index:subsystems}~(iv)}

Further, in all these schemata, $\varphi$ and $\psi$ are allowed to contain free object~variables (in addition to $\overline{n}$) and free~relation variables of any arity $m\geq 1$ (with the exception of $R$).
\end{defn}

The intuition behind the choice schema~(\ref{HPPA:eqn:thechoiceschema}) can be made clearer as follows. Suppose that a structure $(M, S_{1}, S_{2}, \ldots, \#)$ is a model of ${\tt \Sigma^{1}_{1}-PH}_{0}$ and that the antecedent of a given instance of the ${\tt \Sigma^{1}_{1}}$-choice schema~(\ref{HPPA:eqn:thechoiceschema}) holds. Then ~${\tt \Sigma^{1}_{1}-PH}_{0}$ asserts the existence of a relation~$R$, which for the sake of simplicity we can assume to be a binary relation. For each object $n$ in $M$,  the following set is then guaranteed to exist~in~$S_{1}$ by the  arithmetic comprehension schema~(which is included in ${\tt \Sigma^{1}_{1}-PH}_{0}$):
\begin{equation}
R_{n} = \{m: Rnm\}
\end{equation}
So it follows that $(M, S_{1}, S_{2}, \ldots, \#) \models \varphi(n, R_{n})$ for every $n$ in $M$. Hence, in the situation where for every $n$ there is a choice of $P$ such that $\varphi(n,P)$, the ${\tt \Sigma^{1}_{1}}$-choice schema  asserts that there is a uniform way to make these choices, in that there is an $R$ such that its columns $R_{n}$ satisfy $\varphi(n,R_{n})$ for each $n$.

Note, however, that the map $(R,n)\mapsto \#(R_{n})$ is {\it not} a function symbol in the signature of ${\tt HP}^{2}$ or ${\tt BL}^{2}$. For instance, given a binary relation $R$, the comprehension schema~(\ref{HPPA:introoo:comp333}) restricted to arithmetical formulas does {\it not} in general guarantee the existence of the binary relation 
\begin{align}
\{ (n,m): \#(R_{n})=m\} & = \{(n,m): \exists \; X \; (\forall \; x \;  x\in X\leftrightarrow Rnx) \; \& \; \#X=m\} \notag \\  &=   \{(n,m): \forall \; X \; (\forall \; x \;  x\in X\leftrightarrow Rnx) \; \rightarrow \; \#X=m\} \label{HPPA:eqn:counterthecounter}
\end{align}
For, as these definitions make evident, one will in general need the hyperarithmetic comprehension schema~(\ref{HPPA:eqn:hypcompscheme}) in order to show that this relation exists (cf. Propositions~\ref{hppa:binaryiswierd1}-\ref{hppa:binaryiswierd2}). This example underscores an important fact: intuitively simple relations expressible via the maps $\#$ or $\partial$ may be quite complex when explicitly written out in terms of the primitives of the signature. Since our interest in this paper is on restrictions of the comprehension schema, this fact will be particularly important to keep in mind throughout this paper. (In \S~\ref{HPPA:hppa:furtherquestions}, we raise the question of what happens when one {\it does} include function symbols $(R,n)\mapsto \#(R_{n})$ in the signature, so that relations like the one defined in equation~(\ref{HPPA:eqn:counterthecounter}) would count as arithmetical.)

\subsection{Summary of Results about the Provability Relation}\label{HPPA:sec:342314213ddd}

Our primary concern in this paper is with the interpretability relation between subsystems of  ${\tt PA}^{2}$, ${\tt HP}^{2}$, and ${\tt BL}^{2}$, and we summarize our results in the next section (\S~\ref{HPPA:mainresultsasas}). However, since provability implies interpretability, and since the provability relation is intrinsically interesting, in this section we record what is known about this relation among the subsystems of ${\tt PA}^{2}$, ${\tt HP}^{2}$, and ${\tt BL}^{2}$. This is summarized in Figure~\ref{HPPA:figure1}, where the double arrows indicate that the provability implication is irreversible, and where the negated arrows indicate that the provability implication fails, and where the arrows with question marks beside them indicate that the provability implication is unknown.

\begin{figure}
\begin{center}
\[
\xymatrix{
 &{\tt \Pi^{1}_{1}-CA}_{0}\ar@{=>}[d] & & & \\
{\tt \Sigma^{1}_{1}-LB}_{0}\ar@/^/[d]  &{\tt \Sigma^{1}_{1}-AC}_{0}\ar@{=>}[d] & {\tt \Pi^{1}_{1}-HP}_{0}\ar@{=>}[dr] \ar@/^/[rr]|{?} & & \ar@/^/[ll]|{|}{\tt \Sigma^{1}_{1}-PH}_{0}\ar@{=>}[dl]\\
\ar@/^/[u]^{?}{\tt \Delta^{1}_{1}-BL}_{0} \ar@{=>}[d] &{\tt \Delta^{1}_{1}-CA}_{0}\ar@{=>}[d] && {\tt \Delta^{1}_{1}-HP}_{0}\ar@{=>}[d] & \\
{\tt ABL}_{0} & {\tt ACA}_{0} & &{\tt AHP}_{0} & 
}
\]
\caption{Provability Relation in Subsystems of ${\tt BL}^{2}$, ${\tt PA}^{2}$,  and ${\tt HP}^{2}$}
\label{HPPA:figure1}
\end{center}
\end{figure}
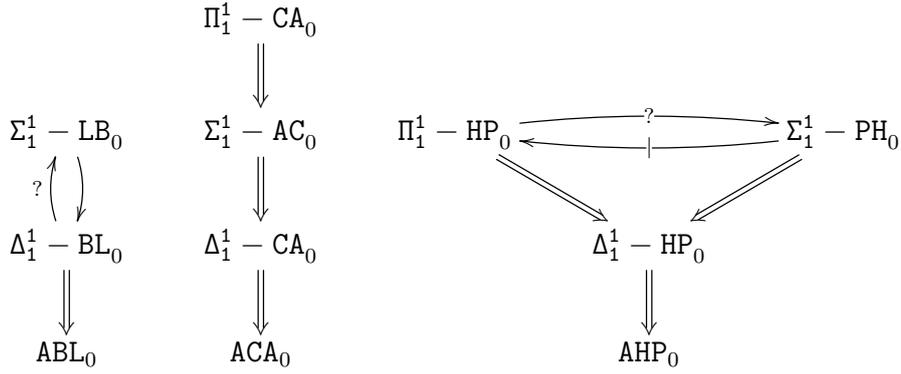

Each of the positive provability relations in in Figure~\ref{HPPA:figure1} follows immediately from the definitions, except for the fact that ${\tt \Pi^{1}_{1}-CA}_{0}$ proves ${\tt \Sigma^{1}_{1}-AC}_{0}$ and the fact that $\Sigma^{1}_{1}$-choice implies $\Delta^{1}_{1}$-comprehension. For the former, see Simpson~\cite{Simpson2009aa} Theorem~V.8.3 pp.~205-206. For the latter, the proof from Simpson~\cite{Simpson2009aa} Theorem~VII.6.6~(i) p.~295 carries over to the setting of ${\tt HP}^{2}$ and ${\tt BL}^{2}$, as we verify now:
\begin{prop}\label{HPPA:prop:choicedefn}
${\tt \Sigma^{1}_{1}-AC}_{0} \rightarrow {\tt \Delta^{1}_{1}-CA}_{0}$, and ${\tt \Sigma^{1}_{1}-PH}_{0} \rightarrow {\tt \Delta^{1}_{1}-HP}_{0}$, and ${\tt \Sigma^{1}_{1}-LB}_{0} \rightarrow {\tt \Delta^{1}_{1}-BL}_{0}$
\end{prop}
\begin{proof}
Let $\mathcal{M}=(M, S, \ldots)$ be a model of ${\tt \Sigma^{1}_{1}-AC}_{0}$ (resp. ${\tt \Sigma^{1}_{1}-PH}_{0}$, ${\tt \Sigma^{1}_{1}-LB}_{0}$). By standard conventions, $\mathcal{M}$ is non-empty. However, nothing in these standard conventions requires that $M$ be non-empty as opposed to say $S$. But, in the case of ${\tt \Sigma^{1}_{1}-AC}_{0}$ we have that $0\in M$, and in the case of ${\tt \Sigma^{1}_{1}-PH}_{0}$ we have that $\#\emptyset\in M$, and likewise in the case of ${\tt \Sigma^{1}_{1}-LB}_{0}$ we have that $\partial\emptyset\in M$. Hence, for the remainder of the proof, fix parameter $a\in M$. Suppose that $\mathcal{M}\models \forall \; \overline{z}\; \varphi(\overline{z})\leftrightarrow \psi(\overline{z})$, where $\varphi$ is $ \Sigma^{1}_{1}$ and $\psi$ is $ \Pi^{1}_{1}$. Then $\mathcal{M}\models \forall \; \overline{z}\; \varphi(\overline{z}) \vee \neg \psi(\overline{z})$. Then by the arithmetical comprehension schema, $\mathcal{M}\models \forall \; \overline{z}\; \exists \; Z \; (\varphi(\overline{z}) \; \wedge \; a\in Z) \vee (\neg \psi(\overline{z}) \wedge a\notin Z)$. By the $ \Sigma^{1}_{1}$-Choice Schema, there is $R$ such that 
\begin{equation}\label{HPPA:eqn:57343214314}
\mathcal{M}\models \forall \; \overline{z} \; \forall \; Z \; (\forall x\; x\in Z \leftrightarrow R\overline{z} x) \rightarrow [(\varphi(\overline{z}) \; \wedge \; a\in Z) \vee (\neg \psi(\overline{z}) \wedge a\notin Z)]
\end{equation}
By the arithmetical comprehension schema, there is $W$ such that $\overline{z}\in W$ if and only if $R\overline{z}a$. Then we claim that $\overline{z}\in W$ if and only if $\varphi(\overline{z})$. For, suppose that $\overline{z}\in W$, so that $R\overline{z}a$. Then $Z=\{x: R\overline{z}x\}$ exists by the arithmetical comprehension schema, and we have $a\in Z$. Then by (\ref{HPPA:eqn:57343214314}), it follows that $\varphi(\overline{z})$. Conversely, suppose that  $\overline{z}\notin W$, so that $\neg R\overline{z}a$.  Then $Z=\{x: R\overline{z}x\}$ exists by the arithmetical comprehension schema, and we have $a\notin Z$.  Then by (\ref{HPPA:eqn:57343214314}), it follows that $\neg\psi(\overline{z})$ and hence $\neg\varphi(\overline{z})$. Hence, in fact we have established that $\overline{z}\in W$ if and only if $\varphi(\overline{z})$. So $\mathcal{M}$ models ${\tt \Delta^{1}_{1}-CA}_{0}$ (resp. ${\tt \Delta^{1}_{1}-HP}_{0}$, ${\tt \Delta^{1}_{1}-BL}_{0}$).
\end{proof}

The known non-provability relations in Figure~\ref{HPPA:figure1} are not difficult to verify. In the case of the subsystems of ${\tt HP}^{2}$, we can read these results off of the results for the subsystems of ${\tt PA}^{2}$, as the proof of Proposition~\ref{HPPA:prop:28} indicates. In the case of the subsystems of ${\tt BL}^{2}$, the only known result we have is that ${\tt ABL}_{0}$ does not prove ${\tt \Delta^{1}_{1}-BL}_{0}$, and this is shown in Proposition~\ref{HPPA:prop:280}. In \S~\ref{HPPA:hppa:furtherquestions}, we list the remaining unknown questions about the provability relation, namely, the question of whether ${\tt \Delta^{1}_{1}-BL}_{0}$ implies ${\tt \Sigma^{1}_{1}-LB}_{0}$ and whether ${\tt \Pi^{1}_{1}-HP}_{0}$ implies ${\tt \Sigma^{1}_{1}-PH}_{0}$.

\subsection{Summary of Results about the Interpretability Relation}\label{HPPA:mainresultsasas}

Most of the formal work done on the the subsystems of ${\tt PA}^{2}$, ${\tt HP}^{2}$, ${\tt BL}^{2}$ has concerned the interpretability strength of these theories. A theory $T_{0}$ {\it is interpretable in} a theory $T_{1}$ ($T_{0}\leq_{\mathrm{I}} T_{1}$) if every model $M_{1}$ of $T_{1}$ uniformly defines without parameters some model $M_{0}$ of $T_{0}$, where ``uniform'' has the sense that e.g. a binary relation symbol $R$ in the signature of $T_{0}$ is defined by one and the same formula $\varphi(x,y)$ in each model $M_{1}$ of $T_{1}$.
(For a more syntactic definition, see Lindstr\"om~\cite{Lindstrom2003aa}~p.~96 or H\'ajek~and~Pudl\'ak~\cite{Hajek1998} pp.~148-149). Since this relation is reflexive and transitive, one can define the associated notions
\begin{align}
\mbox{\;\;\;\;\;}\mbox{\;\;\;\;\;}\mbox{\;\;\;\;\;}\mbox{\;\;\;\;\;}T_{0}\equiv_{\mathrm{I}} T_{1} \Longleftrightarrow T_{0}\leq_{\mathrm{I}} T_{1} \mbox{\; \& \;} T_{1}\leq_{\mathrm{I}} T_{0} \label{HPPA:eqn:theequivalence} & & \\
\mbox{\;\;\;\;\;}\mbox{\;\;\;\;\;}\mbox{\;\;\;\;\;}\mbox{\;\;\;\;\;} T_{0}<_{\mathrm{I}} T_{1} \Longleftrightarrow T_{0}\leq_{\mathrm{I}} T_{1} \mbox{\; \& \;}T_{1}\nleq_{\mathrm{I}} T_{0}\label{HPPA:eqn:theequivalence2} & & 
\end{align}
The relation $\leq_{\mathrm{I}}$ is then a partial order on the set of equivalence classes of theories under the equivalence relation~$\equiv_{\mathrm{I}}$. Since this partial order is in fact a linear order in many natural cases, it can be intuitively conceived as a measure of the strength of the theory. This order is also connected to the formal notion of consistency strength by the following proposition:

\begin{prop}\label{HPPA:conpropa}
Suppose $T_{1}$ is a finitely axiomatizable theory such that ${\tt ACA}_{0}\subseteq T_{1}\subseteq {\tt PA}^{2}$, and suppose that $T_{0}$ is a computable theory in a computable signature. Then
\begin{align}
& T_{1}\vdash \mathrm{Con}(T_{0}) \; \Longrightarrow \;  T_{1}\nleq_{\mathrm{I}} T_{0} \label{HPPA:eqn:theonlyway2}\\
& [T_{0}\leq_{\mathrm{I}} T_{1} \; \& \; T_{1}\vdash \mathrm{Con}(T_{0})] \; \Longrightarrow \;  T_{0}<_{\mathrm{I}} T_{1} \label{HPPA:eqn:theonlyway}
\end{align}
\end{prop}
\begin{proof}
(Sketch) For (\ref{HPPA:eqn:theonlyway2}), note that if $T_{1}\vdash \mathrm{Con}(T_{0})$, then $T_{1}$ proves that there is a model $M_{0}$ of $T_{0}$ (cf. Simpson~\cite{Simpson2009aa} Theorem~IV.3.3 p.~140). But if $T_{1}\leq_{\mathrm{I}} T_{0}$ and $T_{1}$ is finitely axiomatizable, then this interpretation is due to a finite number of the axioms of $T_{0}$. Further, since $T_{0}$ is computable, this can be accurately represented in $T_{1}$, so that inside $T_{1}$ the model $M_{0}$ of $T_{0}$ defines a model $M_{1}$ of $T_{1}$, which likewise exists since the theory inside which we are working (namely $T_{1}$ itself) includes arithmetical comprehension. But then $T_{1}$ would prove $\mathrm{Con}(T_{1})$, which contradicts G\"odel's Second Incompleteness Theorem. (For a formal proof, see Lindstr\"om~\cite{Lindstrom2003aa} Chapter~7 Corollary~1 p.~97). Note that (\ref{HPPA:eqn:theonlyway}) follows immediately from (\ref{HPPA:eqn:theonlyway2}) and definition~(\ref{HPPA:eqn:theequivalence2}).
\end{proof}

\noindent In what follows, we will apply this proposition to $T_{1}={\tt ACA}_{0}$ itself or $T_{1}={\tt \Pi^{1}_{1}-CA}_{0}$, both of which are known to be finitely axiomatizable (cf. Simpson~\cite{Simpson2009aa} Lemma~VIII.1.5 pp.~311-312 and Lemma~VI.1.1 pp.~217-218).

The major previous results on the interpretability strength of the subsystems of ${\tt PA}^{2}$, ${\tt HP}^{2}$, ${\tt BL}^{2}$ can be described as follows. In the 19th Century, Frege in essence showed that ${\tt PA}^{2}\leq_{\mathrm{I}} {\tt HP}^{2}$ (cf. Frege~\cite{Frege1884aa}, \cite{Boolos1995aa}, Boolos~and~Heck~\cite{Boolos1998aa}), and recently Heck (\cite{Heck2000ab} p.~192) and Linnebo (\cite{Linnebo2004} p.~161) noted that Frege's proofs in fact show that ${\tt \Pi^{1}_{1}-CA_{0}}\leq_{\mathrm{I}} {\tt \Pi^{1}_{1}-HP_{0}}$ (cf. \S~\ref{hppa:subsec:frege19}, Corollary~\ref{HPPA:thm:howweroll}). Further, Boolos (\cite{Boolos1987})  showed that the converse holds (cf. Corollary~\ref{hppa:thm:boolos}), so that one has ${\tt \Pi^{1}_{1}-CA_{0}}\equiv_{\mathrm{I}} {\tt \Pi^{1}_{1}-HP_{0}}$ (cf. Corollary~\ref{HPPA:thm:howweroll57}). Heck (\cite{Heck1996}) then showed that ${\tt ABL}_{0}$ interprets Robinson's {\tt Q}, and Ganea and Visser (\cite{Ganea2007}, \cite{Visser2009ab}) independently showed that the converse holds, so that ${\tt ABL}_{0}\equiv_{\mathrm{I}}{\tt Q}$. Likewise, Burgess (\cite{Burgess2005}) showed that ${\tt AHP}_{0}$ interprets Robinson's~Q. Finally, Ferreira and Wehmeier (\cite{Ferreira2002}) showed that ${\tt \Delta^{1}_{1}-BL}_{0}$ is consistent and a slight modification of their proof shows that ${\tt \Sigma^{1}_{1}-LB}_{0}$ is consistent, and inspection of this proof shows that ${\tt \Sigma^{1}_{1}-LB}_{0}<_{\mathrm{I}} {\tt \Pi^{1}_{1}-CA}_{0}$. These previous results and our new results are summarized in Figure~\ref{HPPA:figure2}, where the double arrows indicate that the provability relation is irreversible, and where the single arrows indicate that the provability relation may or may not be irreversible. That is, in the diagram $T_{1}\Rightarrow T_{0}$ means $T_{0}<_{\mathrm{I}} T_{1}$ and $T_{1}\rightarrow T_{0}$ means $T_{0}\leq_{\mathrm{I}} T_{1}$.

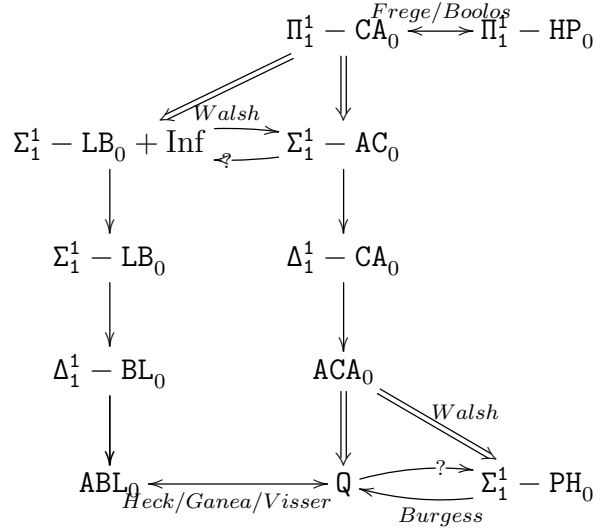
\begin{figure}
\begin{center}
\[
\xymatrix{
 &{\tt \Pi^{1}_{1}-CA}_{0}\ar@{=>}[d] \ar@{=>}[dl]\ar@{<->}[r]^{Frege/Boolos}&{\tt \Pi^{1}_{1}-HP}_{0}  \\
{\tt \Sigma^{1}_{1}-LB}_{0}+\mathrm{Inf}  \ar@/^/[r]^{Walsh} \ar@{->}[d] & \ar@/^/[l]|{?} {\tt \Sigma^{1}_{1}-AC}_{0}\ar@{->}[d] &  \\
{\tt \Sigma^{1}_{1}-LB}_{0} \ar@{->}[d] &{\tt \Delta^{1}_{1}-CA}_{0}\ar@{->}[d] &\\
{\tt \Delta^{1}_{1}-BL}_{0} \ar@{->}[d] \ar@{->}[d] & {\tt ACA}_{0} \ar@{=>}[d] & \\  
{\tt ABL}_{0} & \ar@{<->}[l]^{Heck/Ganea/Visser} {\tt Q} \ar@/^/[r]|{?}  & \ar@/^/[l]^{Burgess} \ar@{<=}[ul]_{Walsh}{\tt \Sigma^{1}_{1}-PH}_{0} 
}
\]
\caption{Interpretability Relation in Subsystems of ${\tt BL}^{2}$, ${\tt PA}^{2}$,  and ${\tt HP}^{2}$}
\label{HPPA:figure2}
\end{center}
\end{figure}

%\begin{figure}
%\begin{center}
%\[
%\xymatrix{
% &{\tt \Pi^{1}_{1}-CA}_{0}\ar@{<=}[d] \ar@{<=}[dl]\ar@{<->}[r]^{Frege/Boolos}&{\tt \Pi^{1}_{1}-HP}_{0}  \\
%{\tt \Sigma^{1}_{1}-LB}_{0}+\mathrm{Inf}  \ar@/^/[r]|{?} \ar@{<-}[d] & \ar@/^/[l]^{Walsh} {\tt \Sigma^{1}_{1}-AC}_{0}\ar@{<-}[d] &  \\
%{\tt \Sigma^{1}_{1}-LB}_{0} \ar@{<-}[d] &{\tt \Delta^{1}_{1}-CA}_{0}\ar@{<-}[d] &\\
%{\tt \Delta^{1}_{1}-BL}_{0} \ar@{<-}[d] \ar@{<-}[d] & {\tt ACA}_{0} \ar@{<=}[d] & \\  
%{\tt ABL}_{0} & \ar@{<->}[l]^{Heck/Ganea/Visser} {\tt Q} \ar@/^/[r]^{Burgess} & \ar@/^/[l]|{?}  \ar@{=>}[ul]_{Walsh}{\tt \Sigma^{1}_{1}-PH}_{0} 
%}
%\]
%\caption{Interpretability Relation in Subsystems of ${\tt BL}^{2}$, ${\tt PA}^{2}$,  and ${\tt HP}^{2}$}
%\label{HPPA:figure2}
%\end{center}
%\end{figure}

Our new results establish upper and lower bounds on consistent subsystems of ${\tt BL}^{2}$ and ${\tt HP}^{2}$ by (i)~finding new constructions of models of these theories, (ii)~noting that the constructions can be formalized in theories such as ${\tt ACA}_{0}$ and ${\tt \Pi^{1}_{1}-CA}_{0}$, and (iii)~applying Proposition~\ref{HPPA:conpropa}. Our first main new result, Theorem~\ref{HPPA:thm:main}, is a construction of a model $\mathcal{M}$ of ${\tt \Sigma^{1}_{1}-LB}_{0}$ using ideas from higher recursion theory (cf. Sacks~\cite{Sacks1990} Part~A). This structure $\mathcal{M}$ models a finite extension of ${\tt \Sigma^{1}_{1}-LB}_{0}$ called ${\tt \Sigma^{1}_{1}-LB}_{0}+\mathrm{Inf}$ which interprets ${\tt \Sigma^{1}_{1}-AC}_{0}$. Moreover, since this construction is formalizable in ${\tt \Pi^{1}_{1}-CA}_{0}$, we have that Proposition~\ref{HPPA:conpropa} implies that ${\tt \Sigma^{1}_{1}-LB}_{0}+\mathrm{Inf}<_{\mathrm{I}} {\tt \Pi^{1}_{1}-CA}_{0}$.

Our second set of results concerns new constructions of models of ${\tt \Delta^{1}_{1}-BL}_{0}$ and ${\tt \Sigma^{1}_{1}-PH}_{0}$ and ${\tt \Delta^{1}_{1}-HP}_{0}+\neg {\tt \Sigma^{1}_{1}-PH}_{0}$. These results are all based on a generalization of a theorem of Barwise-Schlipf and Ferreira-Wehmeier which allows us to built models of these theories on~top of various recursively saturated structures (cf. Theorem~\ref{HPPA:thm:metatheorem}). In particular, we show that if $k$ is a countable recursively saturated o-minimal expansion of a real-closed field, then then there is a function $\#: D(k)\rightarrow k$, where $D(k^{n})$ denotes the definable subsets of $k^{n}$, such that the structure
\begin{equation}
(k, D(k), D(k^{2}), \ldots, \#)
\end{equation}
is a model of ${\tt \Sigma^{1}_{1}-PH}_{0}$. Moreover, we note that this construction can be formalized in ${\tt ACA}_{0}$ for fields with ${\tt ACA}_{0}$-provable quantifier~elimination, so that by Proposition~\ref{HPPA:conpropa}, we have ${\tt \Sigma^{1}_{1}-PH}_{0}<_{\mathrm{I}}{\tt ACA}_{0}$ (cf. Corollary~\ref{HPPA:cor:thebomb}). Further, we show that if $k$ is a saturated algebraically closed field, then there is a there is a function $\#: D(k)\rightarrow k$, where $D(k^{n})$ denotes the definable subsets of $k^{n}$, such that the structure
\begin{equation}
(k, D(k), D(k^{2}), \ldots, \#)
\end{equation}
is a model of ${\tt \Delta^{1}_{1}-HP}_{0}+\neg {\tt \Sigma^{1}_{1}-PH}_{0}$. Further, we can use this construction to answer an open question of Linnebo (cf. Remark~\ref{hppa:rmk:linnebo} and Proposition~\ref{HPPA:lineeeddafdsfa}). However, we do not presently know whether this construction can be formalized in ${\tt ACA}_{0}$, although we have reduced it to the question of whether Ax's Theorem can be formalized in ${\tt ACA}_{0}$ (cf. Remark~\ref{HPPA:axrmk} and Question~\ref{HPPA:axq}). Finally, we show that if $k$ is a countable recursively saturated separably closed field of finite imperfection degree, then there is a function $\partial: D(k)\rightarrow k$, where $D(k^{n})$ denotes the definable subsets of $k^{n}$, such that the structure
\begin{equation}
(k, D(k), D(k^{2}), \ldots, \partial)
\end{equation}
is a model of ${\tt \Delta^{1}_{1}-BL}_{0}$ (cf. Theorem~\ref{HPPA:thmasdfadsf}). However, we do not presently know whether this construction can be formalized in ${\tt ACA}_{0}$, although we have reduced this question to the question of whether the proof of the elimination of imaginaries for separably closed fields can be formalized in ${\tt ACA}_{0}$ (cf. Remark~\ref{HPPA:seprmk} and Question~\ref{HPPA:sepq}).

\section{Standard Models of ${\tt HP}^{2}$ and Associated Results}

Prior to turning to the primary results of this paper in \S\S~\ref{HPPA:secdfadsfsadfkkk}-\ref{HPPA:bssssssss}, the relationship between ${\tt PA}^{2}$ and ${\tt HP}^{2}$ is briefly explored in this section. On the one hand, in \S~\ref{hppa:subsec:frege19}, a brief self-contained proof of Frege and Boolos's result that ${\tt PA}^{2}$ and ${\tt HP}^{2}$ are mutually interpretable is presented (cf. Corollary~\ref{HPPA:thm:howweroll57}). On the other hand, in \S~\ref{HPPA:sec:21431234231423434}, some of the ways in which the standard models of ${\tt HP}^{2}$ are similar to and different from the standard models of ${\tt PA}^{2}$ are examined. The standard model of ${\tt PA}^{2}$ is the structure from~equation~(\ref{HPPA:eqn:standardmodel}), namely,  $(\omega, 0,s,+,\times, \leq, P(\omega))$, while the standard models of ${\tt HP}^{2}$ are the structures from equation~(\ref{HPPA:eqn:sm3}), namely, structures of the form $(\alpha, P(\alpha), P(\alpha^{2}), \ldots, \#_{\alpha})$, where $\alpha$ is an ordinal which is not a cardinal and where $\#_{\alpha}:P(\alpha)\rightarrow \alpha$ denotes cardinality. In \S~\ref{HPPA:sec:21431234231423434}, it is shown that these standard models of ${\tt HP}^{2}$ depend only on the cardinality of $\alpha$ for $\alpha\geq \omega+\omega$ (Proposition~\ref{HPPA:prop:newiso2}~(i)), and further that they can have many automorphisms, unlike the standard model of ${\tt PA}^{2}$ (cf. Proposition~\ref{HPPA:prop:newautomorphic}~(iv)). Finally, it is shown that there is an analogue of the relative categoricity of ${\tt PA}^{2}$ in the setting of ${\tt HP}^{2}$ (cf. Proposition~\ref{hppa:youcantchangeme} and Remark~\ref{rmk:hppa:youcantchangeme}).

\subsection{Models of ${\tt HP}^{2}$ from Infinite Cardinals}\label{HPPA:sec:21431234231423434}

\begin{prop}\label{HPPA:prop:newiso} Suppose  $\alpha,\beta$ are ordinals that are not cardinals, and consider the structures $(\alpha, P(\alpha), P(\alpha^{2}), \ldots, \#_{\alpha})$ and $(\beta, P(\beta), P(\beta^{2}), \ldots, \#_{\beta})$, where $\#_{\alpha}: P(\alpha)\rightarrow \alpha$ and $\#_{\beta}: P(\beta)\rightarrow \beta$ denote cardinality. 
\begin{enumerate}
\item[(i)] The structures $(\alpha, P(\alpha), P(\alpha^{2}), \ldots, \#_{\alpha})$ and $(\beta, P(\beta), P(\beta^{2}), \ldots, \#_{\beta})$ model ${\tt HP}^{2}$.
\item[(ii)] If $\alpha=\omega+k+1$ where $k\geq 0$, then $\left|\alpha-\mathrm{rng}(\#_{\alpha})\right|=k$
\item[(iii)] If $\alpha\geq \omega+\omega$, then $\left|\alpha-\mathrm{rng}(\#_{\alpha})\right|=\left|\alpha\right|$.
\item[(iv)] The structures $(\alpha, P(\alpha), P(\alpha^{2}), \ldots, \#_{\alpha})$ and $(\beta, P(\beta), P(\beta^{2}), \ldots, \#_{\beta})$ are isomorphic if and only if $\alpha=\beta$ or $\alpha, \beta\geq \omega+\omega$ and $\left|\alpha\right|=\left|\beta\right|$.
\end{enumerate}
\end{prop}

\begin{proof}
For~(i), note that restricting attention to ordinals $\alpha$ which are not cardinals serves the purpose of ensuring that $\#(\alpha)<\alpha$, so that $\mathrm{dom}(\#_{\alpha})$ is $P(\alpha)$ and so that $\mathrm{rng}(\#_{\alpha})$ is a subset of $\alpha$. Further, note that $ (\alpha, P(\alpha), P(\alpha^{2}), \ldots, \#_{\alpha})$ satisfies Hume's~Principle by the definition of cardinality. Further, note that by the Power Set Axiom and the Separation Axiom, the structure $(\alpha, P(\alpha), P(\alpha^{2}), \ldots, \#_{\alpha})$ satisfies the full comprehension schema. Hence, in fact $(\alpha, P(\alpha), P(\alpha^{2}),\ldots, \#_{\alpha})$ is a model of ${\tt HP}^{2}$. 

For~(ii), note that $\alpha-\mathrm{rng}(\#_{\alpha})=\{\omega+1, \ldots, \omega+k\}$, which has cardinality $k$.

For~(iii), note that since $\alpha\geq \omega+\omega$, we have that $\alpha-\omega$ is infinite, and hence $\left|\alpha\right|=\left|\alpha-\omega\right|$. Case One: $\alpha$ is a limit ordinal. Then the mapping from $\alpha-\omega$ to $\alpha-\mathrm{rng}(\#_{\alpha})$ given by $\beta\mapsto\beta+1$ is an injection. Case Two: $\alpha$ is a successor ordinal. Then $\alpha=\gamma+n$ where $n>0$ and $\gamma$ is a limit ordinal. Then $\left|\alpha\right|=\left|\alpha-\omega\right|=\left| \gamma-\omega\right|$. Then the mapping from $\gamma-\omega$ to $\alpha-\mathrm{rng}(\#_{\alpha})$ given by $\beta\mapsto \beta+1$ is an injection. Hence in both cases we have $\left|\alpha-\mathrm{rng}(\#_{\alpha})\right|=\left|\alpha\right|$.

For~(iv), suppose that the two structures are isomorphic. Then this isomorphism induces a bijection from $\alpha$ onto $\beta$, and hence $\alpha$ and $\beta$ have the same cardinality. Further, suppose for the sake of contradiction that $\alpha\neq \beta$ and it is not the case that  $\alpha, \beta\geq \omega+\omega$. If $\alpha<\beta<\omega+\omega$, then by part~(ii) we have that $\left|\alpha-\mathrm{rng}(\#_{\alpha})\right|<\left|\beta-\mathrm{rng}(\#_{\beta})\right|<\omega$, and so the two structures are not elementarily equivalent and hence not isomorphic, which is a contradiction. If $\alpha<\omega+\omega\leq \beta$, then by parts~(ii) and~(iii) we have that $\left|\alpha-\mathrm{rng}(\#_{\alpha})\right|<\omega\leq \left|\beta-\mathrm{rng}(\#_{\beta})\right|$, and so the two structures are not elementarily equivalent and hence not isomorphic, which is a contradiction. Hence, in fact, we must have that $\alpha=\beta$ or $\alpha, \beta\geq \omega+\omega$ and $\left|\alpha\right|=\left|\beta\right|$.

Conversely, suppose that $\alpha, \beta\geq \omega+\omega$ have the same cardinality, so that $\mathrm{rng}(\#_{\alpha})=\mathrm{rng}(\#_{\beta})$ by definition, and hence that $\left| \alpha-\mathrm{rng}(\#_{\alpha})\right|=\left|\alpha\right|=\left|\beta\right|=\left|\beta-\mathrm{rng}(\#_{\beta})\right|$ by part~(iii). Hence choose a bijection $f:\alpha\rightarrow \beta$ such that $f(\gamma)=\gamma$ on $\mathrm{rng}(\#_{\alpha})$.  Extend $f$ to a bijection $\overline{f}:  P(\alpha)\rightarrow  P(\beta)$ by setting $\overline{f}(X)=\{f(x): x\in X\}$. Since $f(\gamma)=\gamma$ on $\mathrm{rng}(\#_{\alpha})$ and since $f$ is a bijection, we have that 
\begin{equation}
\overline{f}(\#_{\alpha} (X)) = f(\left|X\right|)= \left|X\right| = \left| \{f(x): x\in X\}\right|=\left|\overline{f}(X)\right|=\#_{\beta}  (\overline{f}(X))
\end{equation}
Hence, $\overline{f}$ is an isomorphism.
\end{proof}

%switch the letters so as not to confuse with HOD

\begin{defn}\label{HPPA:eqn:canon}\index{$H_{\kappa}$, canonical models of ${\tt HP}^{2}$, Definition~\ref{HPPA:eqn:canon}}
If $\kappa$ is a cardinal, then define the ordinal 
\begin{equation}
H_{\kappa}=\begin{cases}
\omega+\kappa+1      & \text{if $\kappa<\omega$}, \\
\omega+\omega      & \text{if $\kappa=\omega$} \\
\kappa+1      & \text{if $\kappa>\omega$}.
\end{cases}
\end{equation}
and define the structure
\begin{equation}
\mathcal{H}_{\kappa}=(H_{\kappa}, P(H_{\kappa}), P(H_{\kappa}^{2}), \ldots, \#_{\kappa})
\end{equation}
where $\#_{\kappa}: P(H_{\kappa})\rightarrow H_{\kappa}$ denotes cardinality.
\end{defn}

\begin{prop}\label{HPPA:prop:newiso2}

\begin{enumerate}

\item[]

\item[(i)] For every ordinal $\alpha$ that is not a cardinal, there is exactly one cardinal $\kappa$ such that the structure  $\mathcal{H}_{\kappa}$ is isomorphic to the structure $(\alpha, P(\alpha), P(\alpha^{2}), \ldots, \#_{\alpha})$, where $\#_{\alpha}: P(\alpha)\rightarrow \alpha$  denotes cardinality. 

\item[(ii)] If $\kappa$ is a cardinal then $\left| H_{\kappa}-\mathrm{rng}(\#_{\kappa})\right|=\kappa$.

\item[(iii)] If $\kappa, \lambda$ are cardinals, then $\mathcal{H}_{\kappa}$ and $\mathcal{H}_{\lambda}$ are isomorphic if and only if $\kappa=\lambda$.

\end{enumerate}
\end{prop}
\begin{proof}
For~(ii), there are three cases. First, suppose that $\kappa=k<\omega$. Then $H_{\kappa}-\mathrm{rng}(\#_{\kappa})=\{\omega+1,\ldots, \omega+k\}$. Second, suppose that $\kappa=\omega$. Then $H_{\kappa}-\mathrm{rng}(\#_{\kappa})=\{\omega+n: 0<n<\omega\}$. Third, suppose that $\kappa>\omega$. Then by Proposition~\ref{HPPA:prop:newiso}~(iii), $\left|H_{\kappa}-\mathrm{rng}(\#_{\kappa})\right|=\left|\kappa+1-\mathrm{rng}(\#)\right| = \left|\kappa+1\right|=\kappa$.

For~(iii), note that the right-to-left direction is trivial. For the left-to-right direction, suppose for the sake of contradiction that $\mathcal{H}_{\kappa}$ and $\mathcal{H}_{\lambda}$ are isomorphic and that $\kappa\neq \lambda$. Then without loss of generality, $\kappa<\lambda$. First suppose that $\kappa<\lambda<\omega$. Then part~(ii) implies that $\mathcal{H}_{\kappa}$ and $\mathcal{H}_{\lambda}$ are not elementarily equivalent, since $\mathcal{H}_{\kappa}$ models that there are exactly $\kappa$ elements not in the range of $\#$, whereas $\mathcal{H}_{\kappa}$ models that there are exactly $\lambda$ elements not in the range of $\#$. Second suppose that $\kappa<\omega\leq \lambda$. Then likewise the structures  $\mathcal{H}_{\kappa}$ and $\mathcal{H}_{\lambda}$ are not elementarily equivalent, since $\mathcal{H}_{\kappa}$ models that there are exactly $\kappa$ many elements not in the range of $\#$, whereas $\mathcal{H}_{\lambda}$ models that there are at least $\kappa+1$ many elements not in the range of $\#$. Third, suppose that $\kappa=\omega<\lambda$. But this cannot happen, since the isomorphism from   $\mathcal{H}_{\kappa}$ and $\mathcal{H}_{\lambda}$  would induce a bijection between the first-order parts of these structures, which, respectively, have cardinality $\omega$ and $\lambda>\omega$. Fourth, suppose that $\omega<\kappa<\lambda$. Again this cannot happen, since the isomorphism from   $\mathcal{H}_{\kappa}$ and $\mathcal{H}_{\lambda}$  would induce a bijection between the first-order parts of these structures, which respectively, have cardinality $\kappa$ and $\lambda>\kappa$.

For~(i), note that uniqueness follows from part~(iii). For existence, there are two cases. If $\alpha<\omega+\omega$, then $\alpha=\omega+k+1$ where $k\geq 0$. Then of course the structure $(\alpha, P(\alpha), P(\alpha^{2}), \ldots, \#_{\alpha})$ is identical with the structure $\mathcal{H}_{k}$. If $\alpha\geq \omega+\omega$, then by Proposition~\ref{HPPA:prop:newiso}~(iv), we have that $(\alpha, P(\alpha), P(\alpha^{2}), \ldots, \#_{\alpha})$ is isomorphic to $H_{\left|\alpha\right|}$.
\end{proof}

\begin{prop}\label{HPPA:prop:newautomorphic}
Suppose that $\kappa$ is a cardinal.
\begin{enumerate}
\item[(i)] If $\beta, \gamma\in (H_{\kappa}-\mathrm{rng}(\#_{\kappa}))$ then there is $f\in \mathrm{Aut}(\mathcal{H}_{\kappa})$ such that $f(\beta)=\gamma$.
\item[(ii)] If $X\subseteq H_{\kappa}$ is $\emptyset$-definable in $\mathcal{H}_{\kappa}$ then $X\subseteq \mathrm{rng}(\#_{\kappa})$ or $(H_{\kappa}-\mathrm{rng}(\#_{\kappa}))\subseteq X$.
\item[(iii)] If $\beta\in \mathrm{rng}(\#_{\kappa})$ and $f\in \mathrm{Aut}(\mathcal{H}_{\kappa})$ then $f(\beta)=\beta$.
\item[(iv)] $\mathrm{Aut}(\mathcal{H}_{\kappa})$ and $\mathrm{Aut}(\kappa)$ are isomorphic, where we view $\kappa$ as a structure in the empty signature.
\end{enumerate}
\end{prop}
\begin{proof}
(i) Let $f:H_{\kappa}\rightarrow H_{\kappa}$ by setting $f(\gamma)=\beta$, $f(\beta)=\gamma$, and let $f$ be the identity otherwise, so that $f$ is a bijection of $H_{\kappa}$. Extend $f$ to a mapping $\overline{f}:\mathcal{H}_{\kappa}\rightarrow \mathcal{H}_{\kappa}$ by setting $\overline{f}(X)=\{f(x): x\in X\}$. Then $\overline{f}$ is clearly a bijection since $f$ is a bijection. To show that it is an automorphism of the structure $\mathcal{H}_{\kappa}$, it suffices to show that $\overline{f}(\#_{\kappa}X)=\#_{\kappa}\overline{f}(X)$. But, since $f$ is the identity on $\mathrm{rng}(\#_{\kappa})$, we have that $\overline{f}(\#_{\kappa}X)=f(\#_{\kappa}X)=\#_{\kappa}X$, and since $f$ is a bijection, we have that $f\upharpoonright X: X\rightarrow \overline{f}(X)$ is a bijection, and so $\#_{\kappa}X=\#_{\kappa}\overline{f}(X)$. Hence, in fact $\overline{f}$ is an automorphism of $\mathcal{H}_{\kappa}$ which sends $\beta$ to $\gamma$.

(ii) Suppose that $X\subseteq H_{k}$ is $\emptyset$-definable in $\mathcal{H}_{\kappa}$, but it is not the case that $X\subseteq \mathrm{rng}(\#_{\kappa})$ or $(H_{\kappa}-\mathrm{rng}(\#_{\kappa}))\subseteq X$. Then there is $\beta\in X\cap (H_{\kappa}-\mathrm{rng}(\#_{\kappa}))$ and $\gamma\in (H_{\kappa}-\mathrm{rng}(\#_{\kappa}))\cap (H_{\kappa}-X)$. By part~(i), there is $f\in \mathrm{Aut}(\mathcal{H}_{\kappa})$ such that $f(\beta)=\gamma$. But since $X$ is $\emptyset$-definable, we have that $\beta\in X$ if and only if $\gamma=f(\beta)\in X$, which is a contradiction.

(iii) Suppose that $\beta\in \mathrm{rng}(\#_{\kappa})$ and $f\in \mathrm{Aut}(\mathcal{H}_{\kappa})$ and $f(\beta)\neq \beta$. Since $\mathrm{rng}(\#_{\kappa})$ is $\emptyset$-definable and $\beta\in \mathrm{rng}(\#_{\kappa})$, we have that $f(\beta)\in \mathrm{rng}(\#_{\kappa})$. Case One: $f(\beta)<\beta$. Note that the relation $<$ on $\mathrm{rng}(\#_{\kappa})$ is $\emptyset$-definable, since on $\mathrm{rng}(\#_{\kappa})$ we have
\begin{equation}\label{eqndsafdafdasf3233}
\lambda\leq \lambda^{\prime} \Longleftrightarrow \mathcal{H}_{\kappa}\models \exists \; X \; \exists \; Y \; \#_{\kappa}(X)=\lambda \; \& \; \#_{\kappa}(Y)=\lambda^{\prime} \; \& \; \exists \; \mbox{ injective } f: X\rightarrow Y
\end{equation}
Then our case assumption $f(\beta)<\beta$ implies $f(f(\beta))<f(\beta)<\beta$ and so we obtain an infinite decreasing sequence of ordinals, which is a contradiction. Case Two: $\beta<f(\beta)$. Since $f\in \mathrm{Aut}(\mathcal{H}_{\kappa})$ we have that $f^{-1}\in \mathrm{Aut}(\mathcal{H}_{\kappa})$, and since $\beta<f(\beta)$ we have $f^{-1}(\beta)<\beta$, since again the relation $<$ on  $\mathrm{rng}(\#_{\kappa})$ is $\emptyset$-definable. Hence, by iterating $f^{-1}(f^{-1}(\beta))<f^{-1}(\beta)<\beta$ as before, we again obtain an infinite decreasing sequence of ordinals, which is a contradiction.

(iv) If $X$ is a set viewed as a structure in the empty signature, then $\mathrm{Aut}(X)$ is just the set of permutations of $X$, and hence if $X$ and $Y$ have the same cardinality, then $\mathrm{Aut}(X)$ and $\mathrm{Aut}(Y)$ are isomorphic as groups. Hence by Proposition~\ref{HPPA:prop:newiso2}~(ii), we have that $\mathrm{Aut}(\kappa)$ and $\mathrm{Aut}(H_{\kappa}-\mathrm{rng}(\#))$ are isomorphic as groups. So it suffices to find a group isomorphism $F: \mathrm{Aut}(H_{\kappa}-\mathrm{rng}(\#))\rightarrow \mathrm{Aut}(\mathcal{H}_{\kappa})$.

To this end, given a bijection $f:H_{\kappa}\rightarrow H_{\kappa}$, extend $f$ to a mapping $\overline{f}:\mathcal{H}_{\kappa}\rightarrow \mathcal{H}_{\kappa}$ by setting $\overline{f}(X)=\{f(x): x\in X\}$, so that $\overline{f}:\mathcal{H}_{\kappa}\rightarrow \mathcal{H}_{\kappa}$ is a bijection. Then we claim that
\begin{equation}\label{HPPA:eqn:equal}
\overline{f}\in \mathrm{Aut}(\mathcal{H}_{\kappa}) \Longleftrightarrow f\upharpoonright (\mathrm{rng}(\#_{\kappa}))= \mathrm{id}_{\mathrm{rng}(\#_{\kappa})}
\end{equation}
The left-to-right direction follows directly from part~(iii). For the right-to-left direction, it suffices to show that $\overline{f}(\#_{\kappa}X)=\#_{\kappa}\overline{f}(X)$. Since $f$ is the identity on $\mathrm{rng}(\#_{\kappa})$, we have that $\overline{f}(\#_{\kappa}X)=f(\#_{\kappa}X)=\#_{\kappa}X$, and since $f$ is a bijection, we have that $f\upharpoonright X: X\rightarrow \overline{f}(X)$ is a bijection, and so $\#_{\kappa}X=\#_{\kappa}\overline{f}(X)$. Hence, equation~(\ref{HPPA:eqn:equal}) does hold, and so we can define $F:  \mathrm{Aut}(H_{\kappa}-\mathrm{rng}(\#_{\kappa})) \rightarrow \mathrm{Aut}(\mathcal{H}_{\kappa})$ by setting $F(g)=\overline{f}$, where $f$ is $g$ on $H_{\kappa}-\mathrm{rng}(\#_{\kappa})$ and where $f$ is the identity on $\mathrm{rng}(\#_{\kappa})$. Since $F(g_{1}\circ g_{2})=F(g_{1})\circ F(g_{2})$, we have that $F$ witnesses the group isomorphism between $\mathrm{Aut}(H_{\kappa}-\mathrm{rng}(\#_{\kappa}))$ and $\mathrm{Aut}(\mathcal{H}_{\kappa})$.
\end{proof}

\begin{rmk}\label{indepndendasfd8788}
The proof of the theorem above shows one how to construct many natural examples of sentences that are independent of ${\tt HP}^{2}$. For instance, in equation~(\ref{eqndsafdafdasf3233}), it was shown how to define the ordering in $\mathcal{H}_{\kappa}$. Using this, one can form a sentence $\varphi$ such that $\mathcal{H}_{\kappa}\models \varphi$ if and only if $\kappa$ is an infinite  successor cardinal, so that $\mathcal{H}_{\omega_{2}}\models {\tt HP}^{2}+\varphi$ and $\mathcal{H}_{\omega_{\omega}}\models {\tt HP}^{2}+\neg \varphi$. This contrasts starkly with the case of ${\tt PA}^{2}$, where there are comparatively few known examples of natural independent sentences.
\end{rmk}

\begin{rmk}
The structures $\mathcal{H}_{\kappa}$ for $\kappa<\omega$ from Definition~\ref{HPPA:eqn:canon} are on one level very different: for, they are not elementarily equivalent since $\mathcal{H}_{\kappa}$ models that there are exactly $\kappa$-many elements that are not in the range of the~$\#$-function. However, on another level, these structures are very similar to each other: for, when $\kappa<\omega$, it is easy to see that $\mathcal{H}_{\kappa}$ is isomorphic to the structure $(\omega, P(\omega), P(\omega^{2}), \ldots, \#_{\kappa}^{\ast})$, where $\#_{\kappa}^{\ast}(X)=0$ if $X$ is infinite and where $\#_{\kappa}^{\ast}(X)=\kappa+1+\left|X\right|$ if $X$ is finite. Further, when one restricts to the ranges of the $\#_{\kappa}^{\ast}$-functions, the induced structures $(\mathrm{rng}(\#_{\kappa}^{\ast}), P(\omega)\cap P(\mathrm{rng}(\#_{\kappa}^{\ast})), P(\omega)\cap P(\mathrm{rng}(\#_{\kappa}^{\ast})^{2}), \ldots, \#_{\kappa}^{\ast})$ are all isomorphic to the structure $(\omega, P(\omega), P(\omega^{2}), \ldots, \#^{\ast})$ where $\#^{\ast}(X)=0$ if $X$ is infinite and where $\#^{\ast}(X)=1+\left|X\right|$ if $X$ is finite. As the next theorem indicates, this is a very general phenomenon among models of ${\tt HP}^{2}$: namely, so long as different $\#$-functions on one and the same underlying set can in some sense see each other, they yield isomorphic structures when one restricts attention to their ranges.
\end{rmk}

\begin{prop}\label{hppa:youcantchangeme}
Suppose that $(M, S_{1}, S_{2}, \ldots, \#_{1}, \#_{2})$ is a structure where $S_{n}\subseteq P(M^{n})$ and where $\#_{i}:S_{1}\rightarrow M$. Suppose further that the structures $(M, S_{1}, S_{2}, \ldots, \#_{i})$ are models of ${\tt HP}^{2}$ for $i\in \{1,2\}$, and further that the structure  $(M, S_{1}, S_{2}, \ldots, \#_{1}, \#_{2})$ satisfies every instance of the comprehension schema~(\ref{HPPA:introoo:comp333}), in the signature that includes both of the function symbols $\#_{1}, \#_{2}$. Finally, for $i\in \{1,2\}$, define the following induced structure:
\begin{equation}
\mathcal{N}_{i}=(\mathrm{rng}(\#_{i}), S_{1}\cap P(\mathrm{rng}(\#_{i})), S_{2}\cap P(\mathrm{rng}(\#_{i})^{2}), \ldots, \#_{i})
\end{equation}
Then $\mathcal{N}_{1}$ and $\mathcal{N}_{2}$ are isomorphic models of ${\tt HP}^{2}$.
\end{prop}
\begin{proof}
First we define a bijection $\Gamma: \mathrm{rng}\#_{1}\rightarrow \mathrm{rng}\#_{2}$. If $\#_{1}X\in \mathrm{rng}\#_{1}$ where $X\in S_1$, then we define $\Gamma(\#_{1}X)=\#_{2}X$. Note that $\Gamma: \mathrm{rng}\#_{1}\rightarrow \mathrm{rng}\#_{2}$ is well-defined: if $\#_{1}X=\#_{1}Y$ then we need to show that $\#_{2} X=\#_{2} Y$. This follows, since
\begin{equation}
\#_{1}X=\#_{1}Y \Longrightarrow [\exists \; \mbox{ bijection } f:X\rightarrow Y] \Longrightarrow \#_{2}X=\#_{2}Y
\end{equation}
Next, note that  $\Gamma: \mathrm{rng}\#_{1}\rightarrow \mathrm{rng}\#_{2}$ is injective:
\begin{equation}
\Gamma(\#_{1}X)=\Gamma(\#_{1}Y)\Longrightarrow \#_{2}X=\#_{2}Y \Longrightarrow [\exists \; \mbox{ bijection } f:X\rightarrow Y] \Longrightarrow \#_{1}X=\#_{1}Y
\end{equation}
Finally, note that $\Gamma: \mathrm{rng}\#_{1}\rightarrow \mathrm{rng}\#_{2}$ is surjective: if $\#_{2}X\in \mathrm{rng}\#_{2}$ then by definition $\Gamma(\#_{1} X)=\#_{2} X$. Hence, in fact  $\Gamma: \mathrm{rng}\#_{1}\rightarrow \mathrm{rng}\#_{2}$ is a bijection. Further, note that the graph of $\Gamma$ is in $S_{2}$ since one has the equality  
\begin{equation}
\mathrm{graph}(\Gamma)=\{(x,y)\in M^{2}: \exists \; Z \; \#_{1}(Z)=x \; \& \; \#_{2}(Z)=y\}
\end{equation}
and since it was assumed that the structure  $(M, S_{1}, S_{2}, \ldots, \#_{1}, \#_{2})$ satisfies every instance of the comprehension schema~(\ref{HPPA:introoo:comp333}) in the signature that includes both of the function symbols $\#_{1}, \#_{2}$. Now, extend to $\overline{\Gamma}: \mathcal{N}_{1}\rightarrow \mathcal{N}_{2}$ by setting $\overline{\Gamma}(X)=\{\Gamma(x): x\in X\}$, which exists in $S_{1}$ since the graph of $\Gamma$ is in $S_{2}$. Then $\overline{\Gamma}: \mathcal{N}_{1}\rightarrow \mathcal{N}_{2}$ is an isomorphism, because 
\begin{equation}
\overline{\Gamma}(\#_{1}X)=\Gamma(\#_{1}X)=\#_{2} X = \#_{2} \{\Gamma(x): x\in X\} = \#_{2} \overline{\Gamma}(X),
\end{equation}
where the first and second equalities follow respectively from the definitions of $\overline{\Gamma}$ and $\Gamma$, and where the third equality follows from the fact that $\Gamma: X\rightarrow \{\Gamma(x): x\in X\}$ is a bijection whose graph is in $S_{2}$, and where the last equality follows from the definition of~$\overline{\Gamma}$.
\end{proof}

\begin{rmk}\label{rmk:hppa:youcantchangeme}
The previous proposition can be thought of as an analogue of the relative categoricity results for models of ${\tt PA}^{2}$. In the 19th Century, Dedekind showed that any two models $(M,+,\times, P(M), P(M^{2}),\ldots)$ and $(N,\oplus, \otimes, P(N), P(N^{2}),\ldots)$ of ${\tt PA}^{2}$ are isomorphic (\cite{Dedekind1888} \S~132, cf. Shapiro~\cite{Shapiro1991} Theorem 4.8 p.~82). However, it is not difficult to see that Dedekind's result can be relativized, in the following way: if $(M,+,\times, \oplus, \otimes, S_{1}, S_{2},\ldots)$ is a structure where $S_{n}\subseteq P(M^{n})$ such that $(M,+,\times, S_{1}, S_{2}, \ldots)$ and $(M,\oplus,\otimes, S_{1}, S_{2}, \ldots)$ are models of ${\tt PA}^{2}$ and such that $(M,+,\times, \oplus, \otimes, S_{1}, S_{2},\ldots)$ satisfies every instance of the comprehension schema~(\ref{HPPA:introoo:comp333}) in the signature of $+,\times,\oplus,\otimes$, then $(M,+,\times, S_{1}, S_{2}, \ldots)$ and $(M,\oplus,\otimes, S_{1}, S_{2}, \ldots)$ are isomorphic (cf. Parsons~\cite{Parsons2008} \S~49 pp.~279~ff). The previous proposition is simply the analogue of this phenomenon in the setting of ${\tt HP}^{2}$.
\end{rmk}

\subsection{The Mutual Interpretability of ${\tt PA}^{2}$ and ${\tt HP}^{2}$}\label{hppa:subsec:frege19}

The goal of this section is to present a brief and self-contained proof of the result that ${\tt PA}^{2}$ is mutually interpretable with ${\tt HP}^{2}$ (Corollary~\ref{HPPA:thm:howweroll57}). One half of this result, namely, the interpretability of ${\tt HP}^{2}$ in ${\tt PA}^{2}$ is due to Boolos (Corollary~\ref{hppa:thm:boolos}). The other half of the result, namely, the interpretability of ${\tt PA}^{2}$ in ${\tt HP}^{2}$ is now called Frege's Theorem, namely (Corollary~\ref{HPPA:thm:howweroll}). The proof of Frege's Theorem can be broken down into two steps: first, the proof that ${\tt PA}^{2}$ is interpretable in the theory consisting of (Q1)-(Q2) and the comprehension schema~(\ref{HPPA:eqn:pacompsche}) (cf. Theorem~\ref{HPPA:thm:howweroll3}), and second the argument that this latter theory is interpretable in ${\tt HP}^{2}$ (cf. Theorem~\ref{HPPA:thm:howweroll2}). Elements of the first step can be found in Dedekind (cf. \cite{Dedekind1888} \S~72), and elements of this second step can be traced back to Frege (cf. Boolos and Heck \cite{Boolos1998aa}). 

However, the modern presentation stems from Wright~\cite{Wright1983} pp.~154-169 (cf. also Boolos~\cite{Boolos1996aa}). The warrant for including a proof of this result here is two-fold: (i)~the proof presented here is slightly briefer than other published presentations, and (ii)~the proof presented here is slightly different from other published presentations in that it is centered around the notion of Dedekind-finiteness, defined in terms of the lack of injective non-surjective functions, as opposed to Frege's ancestral notion (cf. the relation $X\nprec X$ in Proposition~\ref{HPPA:prop24123423141243} and Theorem~\ref{HPPA:thm:howweroll2}).

The observations recorded in this section about the $\Pi^{1}_{n}$-comprehension schema are due to Heck (\cite{Heck2000ab} p.~192) and Linnebo (\cite{Linnebo2004} p.~161). The trick of defining the graph of addition and multiplication in terms of its initial segments in the proof of Theorem~\ref{HPPA:thm:howweroll3} is adapted from Burgess~and~Hazen~\cite{Burgess1998}~pp.~6-10, although their concern there was not with Frege's Theorem.

\begin{thm}\label{HPPA:thm:howweroll3}
${\tt PA}^{2}$ is interpretable in the theory consisting of (Q1)-(Q2) and the comprehension schema~(\ref{HPPA:eqn:pacompsche}). More generally, ${\tt \Pi^{1}_{n}- CA}_{0}$ is interpretable in the theory consisting of (Q1)-(Q2) and the comprehension schema~(\ref{HPPA:eqn:pacompsche}) restricted to $\Pi^{1}_{n}$-formulas for $n>0$.
\end{thm}
\begin{proof}
Suppose that we are working with structure $\mathcal{M}=(M,S_{1}, S_{2}, \ldots, 0,s)$ that satisfies (Q1)-(Q2) and the comprehension schema~(\ref{HPPA:eqn:pacompsche}) restricted to $\Pi^{1}_{n}$-formulas for $n>0$. In what follows, we will refer respectively to the element $0$ and the function $s$ as ``zero'' and ``successor.'' It must be shown how to uniformly define a model of ${\tt \Pi^{1}_{n}- CA}_{0}$  within this structure. We say that $X$ in $S_{1}$ is {\it inductive} if it contains zero and is closed under successor. Let $N$ be the intersection of all the inductive sets $X$ in $S_{1}$, which exists in $S_{1}$ by $\Pi^{1}_{1}$-comprehension. Note that zero is in $N$ by construction, and note that $N$ is closed under successor: for, if $a$ is in $N$ then $a$ is contained in every inductive set $X$, and by definition of inductive sets, it follows that the successor of $a$ is contained in every inductive set $X$, which is to say that the successor of $a$ is in $N$.

Hence, we can define the structure $\mathcal{N}=(N,S_{1}\cap P(N), S_{2}\cap P(N^{2}), \ldots, 0,s)$ uniformly within $\mathcal{M}$. This structure then satisfies (Q1)-(Q2) since $\mathcal{M}$ satisfies (Q1)-(Q2). Further, $\mathcal{N}$ satisfies the Mathematical Induction Axiom~(\ref{HPPA:eqn:MIaxiom}), since if $F\in S_{1}\cap P(N)$ contains zero and is closed under successor, then $F\in S_{1}$ contains zero and is closed under successor, and so by definition of $N$, it follows that $N\subseteq F\subseteq N$. For (Q3), let $X$ be the subset of $N$ for which the conclusion holds, i.e., $X=\{a\in N: a\neq 0\rightarrow \exists \; w \in N \; x=sw\}$. Clearly zero is in $X$, and suppose that $a\in X\subseteq N$: then of course $sa=sw$ for some $w\in N$, namely $w=a$, and hence $sa\in X$. Hence, by the Mathematical Induction Axiom~(\ref{HPPA:eqn:MIaxiom}), it follows that $X=N$. Finally, before turning to the remainder of the axioms of Robinson's~Q, note that since $\mathcal{M}$ satisfies $\Pi^{1}_{n}$-comprehension, we have that $\mathcal{N}$ satisfies $\Pi^{1}_{n}$-comprehension as well, since the second-order parts of $\mathcal{N}$ are just the second-order parts of $\mathcal{M}$ restricted to subsets of $N$.

To verify axioms Q4-Q5 of Robinson's~Q, we must first define addition. Let $x+y=z$ if and only if there is a graph of a partial function $G\subseteq N^{3}$ such that $(x,y,z)\in G\subseteq N^{3}$ and
\begin{equation}\label{HPPA:eqn:dfadsfasdfadsf33}
(x,0,x)\in G \; \& \;  [(x,sy,z)\in G \rightarrow \; \exists \; w \; sw=z \; \&\; (x,y,w)\in G]
\end{equation}
That is, we define the graph of addition as the union of its initial segments. Note that this graph of addition exists by the $\Pi^{1}_{1}$-Comprehension Schema. Further, note that addition is well-defined on its domain. Suppose that $G_{0}$ and $G_{1}$ are partial functions which satisfy equation~(\ref{HPPA:eqn:dfadsfasdfadsf33}) and fix an arbitrary $x$ and let $Y = \{y\in N : \forall\; \; z_{0}, z_{1} \; (x,y,z_{0})\in G_{0} \; \&\;  (x,y,z_{1})\in G_{1} \rightarrow z_{0}=z_{1}\}$. Clearly, $0\in Y$ and if $y\in Y$ and $(x, sy, z_{0})\in G_{0}$ and $(x, sy, z_{1})\in G_{1}$ then there is $w_{0}, w_{1}$ such that $sw_{0}=z_{0}$ and $sw_{1}=z_{1}$ and $(x,y,w_{0})\in G_{0}$ and $(x,y,w_{1})\in G_{1}$. Then since $y\in Y$ we have $w_{0}=w_{1}$ and hence $z_{0}=sw_{0}=sw_{1}=z_{1}$. Hence, in fact, addition is a well-defined function on its domain. To show that it is a total function, fix an arbitrary $x$ and let $Y = \{y\in N: \exists \; z\; x+y=z\}$. Clearly, $0\in Y$, since we can choose $G=\{(x,0,x)\}$. Suppose that $y\in Y$, say, with $(x,y,z)\in G$. To see that $sy\in Y$, set $G^{\prime} = G\cup \{(x,sy,sz)\}$. Then clearly $G^{\prime}$ also satisfies equation~(\ref{HPPA:eqn:dfadsfasdfadsf33}). Hence, in fact, addition is a total function. Finally, the verification of Q4 and Q5 follows directly from our construction in equation~(\ref{HPPA:eqn:dfadsfasdfadsf33}). To verify Q6-Q7, just define multiplication analogously.
\end{proof}

\begin{rmk}
Hence, it remains to show that the theory consisting of (Q1)-(Q2) and the comprehension schema~(\ref{HPPA:eqn:pacompsche}) is interpretable in ${\tt HP}^{2}$. In preparation for this result (Theorem~\ref{HPPA:thm:howweroll2}), we first record some elementary considerations in the following proposition.
\end{rmk}

\begin{prop}\label{HPPA:prop24123423141243}
Suppose that $(M, S_{1}, S_{2}, \ldots, \#)$ models ${\tt AHP}_{0}$. For $X,Y$ in $S_{1}$, define $X\prec Y$ if and only if there is injective non-surjective function $f:X\rightarrow Y$ such that $\mathrm{graph}(f)$ is in $S_{2}$. Then for $a,b\in M$ and $X,U,A,B$ in $S_{1}$, it follows that

\begin{enumerate}

\item[(i)] If $a\notin X$ and $X\cup \{a\} \prec X\cup \{a\}$ then $X\prec X$.

\item[(ii)] If $a\notin X$ and $U \prec X\cup \{a\}$ then $U\prec X$ or $\#U=\#X$.

\item[(iii)] If $a\in A, b\in B$, then $\#A=\#B$ if and only if $\#(A-\{a\})=\#(B-\{b\})$

\item[(iv)] If $X\neq \emptyset$ then $\emptyset \prec X$

\item[(v)] $X\nprec \emptyset$

\end{enumerate}

\end{prop}

\begin{proof}
For~(i), suppose that $f: X\cup \{a\} \rightarrow X\cup \{a\}$ is an injection that is not a surjection. If $f(X)\subseteq X$ and $f:X\rightarrow X$ is surjective, then $f(a)=a$ and hence $f: X\cup \{a\} \rightarrow X\cup \{a\}$ would be surjective, contrary to hypothesis; hence when $f(X)\subseteq X$, it must be the case that $f:X\rightarrow X$ is injective but not surjective. On the other hand, when $f(X)\nsubseteq X$ then say $f(y)=a$ where $y\in X$ and $f(a)=z\in X$, and hence define $g:X\rightarrow X$ by $g(y)=z$ and $g=f$ otherwise. Then $g$ is injective and misses the same point that $f$ does. Further, the graph of $g$ exists by the arithmetical comprehension schema.

For~(ii), suppose that $f: U \rightarrow X\cup \{a\}$ is an injection which is not a surjection. If $f(U)\subseteq X$ then $\#U=\#X$ when $f: U\rightarrow X$ is a bijection and $U\prec X$ otherwise. If $f(U)\nsubseteq X$ then say $f(y)=a$ and $f$ misses $b\in X$, in which case we define an injective function $g: U\rightarrow X$ by $g(y)=b$ and $g=f$ otherwise. The graph of $g$ exists by the arithmetical comprehension schema. If $g$ is a bijection, then $\#U=\#X$ and $U\prec X$ otherwise.

For~(iii), suppose that $a\in A$ and $b\in B$ and let us first establish the left-to-right direction. So suppose that $f:A \rightarrow B$ is a bijection. If $f(a)=b$ then $f\upharpoonright (A-\{a\})$ is the desired bijection. If $f(a)=d$ for $d\neq b$ and $f(c)=b$ for $c\neq a$, then define a bijection $g:(A-\{a\})\rightarrow (B-\{b\})$ by $g(c)=d$ and $g=f$ otherwise. The graph of this function~$g$ then exists by the arithmetical comprehension schema. Now let us establish the right-to-left direction. Suppose that $g:(A-\{a\})\rightarrow (B-\{b\})$ is a bijection. Then define $f:A\rightarrow B$ by $f(a)=b$ and $f=g$ otherwise. Then the graph of $f$ exists by the arithmetical comprehension schema and $f$ is a bijection since $g$ was a bijection.

For~(iv), note that the ``empty'' binary relation witnesses that there is an injective non-surjective function from $\emptyset$ to $X$.

For~(v), note that if $X\prec\emptyset$, then there would be an injective non-surjective function $f:X\rightarrow \emptyset$, which would imply that there was an element in $\emptyset\setminus \mathrm{rng}(f)$, which would imply that there was some element in $\emptyset$. 

\end{proof}

\begin{rmk}
It is well-known that the chief difficulty in the proof of the following theorem is establishing the totality of the successor function (cf. remarks to this effect in Wright~\cite{Wright1983} p.~161). Prior to looking at the proof, it is helpful to think about what happens on the standard models $(\alpha, P(\alpha), P(\alpha^{2}), \ldots, \#)$ from \S~\ref{HPPA:sec:21431234231423434}, where $\alpha$ is an ordinal which is not a cardinal and where $\#:P(\alpha)\rightarrow \alpha$ is cardinality. It is easy to see that~$\omega$ is uniformly definable in each of these structures. Further, it is easy to see that for each $n\in \omega$, it follows that 
\begin{equation}
\{\#W: W\prec \{0, \ldots, n\}\} = \{0, \ldots, n\}
\end{equation}
where as in the previous proposition, $X\prec Y$ if and only if there is injective non-surjective function $f:X\rightarrow Y$. From this we see that
\begin{equation}\label{HPPA:eqn:ddsafdsaf67869dsfa476}
\{0, \ldots, n\}\nprec \{0, \ldots, n\} \; \& \; \#\{0, \ldots, n\}=\#\{\#W: W\prec \{0, \ldots, n\}\}
\end{equation}
as well as
\begin{eqnarray}
s(\#\{0, \ldots, n\})& =& s(n+1)=n+2 =\#(\{0, \ldots,n\} \cup\{n+1\})\nonumber \\
&=& \#(\{\#W: W\prec \{0, \ldots, n\}\} \cup \{\#(\{0, \ldots, n\})\})
\end{eqnarray}
The entire idea of the below proof is to show that we can replicate these considerations in arbitrary models of ${\tt HP}^{2}$. So in such an arbitrary model, we will define an analogue $N$ of $\omega$, and for analogues $X$ of $\{0, \ldots, n\}$, we will find that
\begin{equation}\label{HPPA:eqn:341234832132143aadfasdf}
s(\#X)= \#(\{\#W: W\prec X\} \cup \{\#X\})
\end{equation}
This, in any case, is the heuristic explanation of the proof of the totality of the successor function in the following theorem.
\end{rmk}

\begin{thm}\label{HPPA:thm:howweroll2} 
The theory consisting of (Q1)-(Q2) and the comprehension schema~(\ref{HPPA:eqn:pacompsche}) is interpretable in ${\tt HP}^{2}$. More generally, the theory consisting of (Q1)-(Q2) and the comprehension schema~(\ref{HPPA:eqn:pacompsche})  restricted to $\Pi^{1}_{n}$-formulas is interpretable in ${\tt \Pi^{1}_{n}-HP}_{0}$ for $n>0$.
\end{thm}
\begin{proof}
Suppose that we are working with structure $\mathcal{M}=(M,S_{1}, S_{2}, \ldots, \#)$ that satisfies ${\tt \Pi^{1}_{n}-HP}_{0}$. It must be shown how to uniformly define a model of (Q1)-(Q2) and the comprehension schema~(\ref{HPPA:eqn:pacompsche})  restricted to $\Pi^{1}_{n}$-formulas. Define $0=\#\emptyset$ and define $s(x,y)$ if and only if there is $X,Y$ in $S_{1}$ such that $\#X=x, \#Y=y$, and there is $b\in Y$ such that $\# X=\# (Y-\{b\})$. That is, $s(x,y)$ says that $x,y$ are respectively cardinalities of sets $X,Y$ and the cardinality of $X$ is equal to the cardinality of $Y$ minus one point. Note that the relation $s$ exists in $S_{2}$ by the $\Pi^{1}_{1}$-comprehension schema.  In what follows, we will respectively refer to the element $0$ and the relation $s$ as ``zero'' and ``successor,'' keeping in mind that formally $s$ is a binary relation. Then say that $X$ in $S_{1}$ is ${\it inductive}$ if it contains zero and is closed under successors, that is, if $x\in X$ and $s(x,y)$ then $y\in X$. Then define $N$ to be the intersection of all the inductive sets, so that $N$ is in $S_{1}$ by the $\Pi^{1}_{1}$-comprehension schema. Now we show that (i)~$s$ is a well-defined function on its domain, and that (ii)~$s$ is a total function on $N$, and that (iii)~$s$~maps elements of $N$ to elements of $N$, and that (iv)~$s$~satisfies axioms Q1-Q2 on $N$.

For~(i), to see that $s$ is well-defined, suppose that $s(x,y)$ and $s(x,z)$. Then $x=\#X$, $y=\#Y$, $z=\#Z$ and there exists $b\in Y, c\in Z$ such that $\# X=\# (Y-\{b\})=\#(Z-\{c\})$. Then by the right-to-left direction of Proposition~\ref{HPPA:prop24123423141243}~(iii), it follows that $\#Y=\#Z$ and hence $y=\#Y=\#Z=z$. Hence, $s$ is a well-defined function on its domain.

For~(ii), recall from Proposition~\ref{HPPA:prop24123423141243} that for $X,Y$ in  $S_{1}$, we say $X\prec Y$ if and only if there is an injective non-surjective function $f:X\rightarrow Y$ such that $\mathrm{graph}(f)$ is in $S_{2}$. Then by iterated applications of $\Pi^{1}_{1}$-comprehension, the following exist in $S_{2}$ and $S_{1}$ respectively
\begin{eqnarray}
R & = & \{(\# W, \# X): W\prec X\} \\
Z & = & \{\# X: X\nprec X \; \&\;  \exists \; Y \; (\forall \; w \; w\in Y \leftrightarrow (w, \# X)\in R ) \; \& \; \#X=\#Y\}
\end{eqnarray}
Note that
\begin{equation}\label{HPPA:eqn:dafdsfasdf08080808}
Z = \{ \# X : X\nprec X \; \& \; \#X=\#(\{\# W : W\prec X\})\}
\end{equation}
(It may be heuristically helpful to compare this with equation~(\ref{HPPA:eqn:ddsafdsaf67869dsfa476})). Suppose that $\# X$ is in $Z$. Then $X\nprec X$ and  $\#X=\#(\{\# W : W\prec X\})$. Then 
\begin{equation}
s(\# X,\# (\{\# W: W\prec X\}\cup \{\# X\}))
\end{equation}
(Likewise, it may be helpful to compare this with equation~(\ref{HPPA:eqn:341234832132143aadfasdf})). Hence, we have the inclusion $Z\subseteq \{x: \exists \; y \; s(x,y)\}$, and so it suffices to show that $Z$ is inductive.

Clearly, $0\in Z$. Suppose that $\# X$ is in $Z$, so that $X\nprec X$ and $\#X=\#(\{\# W : W\prec X\})$. Then $s(\# X,\# (\{\# W: W\prec X\}\cup \{\# X\}))$. Since successor is well-defined on its domain by part~(i), it suffices to show that $\# (\{\# W: W\prec X\}\cup \{\# X\})$ is in $Z$. We have $\{\# W : W\prec X\}\nprec \{\# W : W\prec X\}$. Since $\# X\notin \{\# W : W\prec X\}$, it follows from Proposition~\ref{HPPA:prop24123423141243}~(i)  that $\{\# W : W\prec X\}\cup \{\# X\} \nprec \{\# W : W\prec X\}\cup \{\# X\}$. Hence, $\# (\{\# W: W\prec X\}\cup \{\# X\})$ satisfies the first conjunct of $Z$ in equation~(\ref{HPPA:eqn:dafdsfasdf08080808}). To see that  $\# (\{\# W: W\prec X\}\cup \{\# X\})$ satisfies the second conjunct of $Z$ in equation~(\ref{HPPA:eqn:dafdsfasdf08080808}), it suffices to show that
\begin{equation}\label{HPPA:eqn:9999777989088}
\{\# W : W\prec X\}\cup \{\# X\} = \{\# U : U\prec \{\# W : W\prec X\}\cup \{\# X\}\}
\end{equation}
For the left-to-right direction, suppose first that $W\prec X$. Since $X$ is bijective with $\{\# W : W\prec X\}$, we have that $W\prec \{\# W : W\prec X\}\cup \{\# X\}$. Continuing with the left-to-right direction, suppose now that $\#U=\#X$. Since $X$ is bijective with $\{\# W : W\prec X\}$, we have that $\#U=\#( \{\# W : W\prec X\})$ and hence $U\prec \{\# W : W\prec X\}\cup \{\# X\}$. For the right-to-left direction, suppose that $U\prec \{\# W : W\prec X\}\cup \{\# X\}$. Since $\# X\notin \{\# W : W\prec X\}$, we have by Proposition~\ref{HPPA:prop24123423141243}~(ii) that $\#U=\#(\{\# W : W\prec X\})=\# X$ or $U\prec \{\# W : W\prec X\}$. Hence, in fact equation~(\ref{HPPA:eqn:9999777989088}) holds. It follows that $\# (\{\# W: W\prec X\}\cup \{\# X\})$ is in $Z$. Hence, $Z$ is an inductive set, and as mentioned at the close of the above paragraph, it thus follows that successor is a total function on $N$.

(iii) Now we show that successor maps elements of $N$ to elements of $N$. Suppose that $a$ is in $N$. Then by definition, $a$ is contained in every inductive set, and by parts~(i)-(ii), it follows that there is unique~$b$ such that $s(a,b)$, from which it follows that $b$ is contained in every inductive set, so that $b$ is contained in $N$ as well. Hence, successor maps elements of $N$ to elements of $N$.

(iv) Finally, we note that the successor function~$s$ satisfies axioms (Q1)-(Q2). To see that it satisfies (Q1), note that if $s\#X=0=\#\emptyset$, then $\emptyset$ would be bijective with a non-empty set, which is a contradiction. To see that it satisfies (Q2), suppose that $s\# X = s\# Y$. Then $s\# X=\# A$ where  $\# X=\# (A-\{a\})$ for some $a\in A$ and $s\# Y=\# B$ where $\# Y=\# (B-\{b\})$ for some $b\in B$. Then the left-to-right direction of Proposition~\ref{HPPA:prop24123423141243}~(iii) implies that $\#X=\#(A-\{a\})=\#(B-\{b\})=\# Y$.

Putting this all together, we can uniformly define the structure $\mathcal{N}=(N,S_{1}\cap P(N), S_{2}\cap P(N^{2}), \ldots, 0, s)$ which satisfies (Q1)-(Q2).  Finally, note that since $\mathcal{M}$ satisfies $\Pi^{1}_{n}$-comprehension, we have that $\mathcal{N}$ satisfies $\Pi^{1}_{n}$-comprehension as well, since the second-order parts of $\mathcal{N}$ are just the second-order parts of $\mathcal{M}$ restricted to subsets of $N$.
\end{proof}

\begin{cor}\label{HPPA:thm:howweroll} 
${\tt PA}^{2}$ is interpretable in ${\tt HP}^{2}$. More generally, ${\tt \Pi^{1}_{n}-CA}_{0}$ is interpretable in ${\tt \Pi^{1}_{n}-HP}_{0}$ for $n>0$.
\end{cor}
\begin{proof}
This follows immediately from Theorem~\ref{HPPA:thm:howweroll2} and Theorem~\ref{HPPA:thm:howweroll3}.\index{Frege's Theorem, Corollary~\ref{HPPA:thm:howweroll}}
\end{proof}

\begin{rmk}
The following theorem was first noted by Boolos (\cite{Boolos1987}).  We include here for the sake of having a relatively self-contained presentation of the main results in this area, and because we will use Boolos' construction to transfer facts about the provability relation from subsystems of ${\tt PA}^{2}$ to subsystems of ${\tt HP}^{2}$ (cf. the proofs of Proposition~\ref{HPPA:prop:28} and Proposition~\ref{hppa:binaryiswierd1}).
\end{rmk}

\begin{thm}\label{hppa:thm:boolos}
${\tt HP}^{2}$ is interpretable in ${\tt PA}^{2}$. More generally, ${\tt \Pi^{1}_{n}-HP}_{0}$ is interpretable in ${\tt \Pi^{1}_{n}-CA}_{0}$ for $n>0$, and ${\tt \Sigma^{1}_{1}-PH}_{0}$ is interpretable in ${\tt \Sigma^{1}_{1}-AC}_{0}$ and ${\tt AHP}_{0}$ is interpretable in ${\tt ACA}_{0}$.
\end{thm}
\begin{proof}
We begin with the proof of the interpretability of ${\tt AHP}_{0}$ in ${\tt ACA}_{0}$. We will note how this proof yields all the other results as well. Let us work in a model $\mathcal{M}=(M,S_{1}, S_{2}, \ldots, \oplus,\otimes)$ of ${\tt ACA}_{0}$, where $S_{n}\subseteq P(M^{n})$. We must show how to uniformly define a model of ${\tt AHP}_{0}$. Consider the model $\mathcal{N}=(M, S_{1}, S_{2}, \ldots, \#)$ where $\#(X)=n+1$ if $\left|X\right|=n$, and where $\#(X)=0$ if $X$ is infinite. Then $\mathcal{N}$ is clearly definable in $\mathcal{M}$ since the graph of $X$ is arithmetically definable. Further, since this graph is arithmetically definable, it follows that $\mathcal{N}$ satisfies the arithmetical comprehension schema. Further, by Simpson~\cite{Simpson2009aa} Lemma~II.3.6 p.~70, ${\tt ACA}_{0}$ proves that any two infinite sets are bijective, so that $\mathcal{N}$ is a model of ${\tt AHP}_{0}$. Hence, in fact we have that  ${\tt AHP}_{0}$ is interpretable in ${\tt ACA}_{0}$. Further, it is obvious from this construction that $\mathcal{N}$ will satisfy whatever comprehension schemas $\mathcal{M}$ satisfies.
\end{proof}

\begin{cor}\label{HPPA:thm:howweroll57} 
${\tt PA}^{2}$ is mutually interpretable with ${\tt HP}^{2}$. More generally, ${\tt \Pi^{1}_{n}-CA}_{0}$ is mutually interpretable with ${\tt \Pi^{1}_{n}-HP}_{0}$ for $n>0$.
\end{cor}
\begin{proof}
This follows immediately from Corollary~\ref{HPPA:thm:howweroll} and Theorem~\ref{hppa:thm:boolos}.
\end{proof}

\section{Standard Models of Subsystems of ${\tt BL}^{2}$ and Associated Results}\label{HPPA:secdfadsfsadfkkk}
 
The primary goal of this section is to study models of subsystems of ${\tt BL}^{2}$ that are standard in the sense that they have the form $(\omega, S_{1}, S_{2}, \ldots, \partial)$, where the sets $S_{n}\subseteq P(\omega^{n})$ all come from some antecedently fixed computational class (e.g. the recursive sets, the arithmetical sets, the hyperarithmetical sets, etc.). The main result of this section is Theorem~\ref{HPPA:thm:main} which gives a construction of a standard model of the hyperarithmetic subsystem of ${\tt BL}_{0}$ in terms of the hyperarithmetic subsets of natural numbers. Further, this construction isolates a certain sentence $\mathrm{Inf}$ (cf. Definition~\ref{HPPA:eqn:defn:inf}) such that  ${\tt \Sigma^{1}_{1}-AC}_{0}\leq_{\mathrm{I}} {\tt \Sigma^{1}_{1}-LB}_{0}+\mathrm{Inf}<_{\mathrm{I}} {\tt \Pi^{1}_{1}-CA}_{0}$ (cf. Corollary~\ref{HPPA:thm:main:cor} and Figure~\ref{HPPA:figure2}). 

In the preliminary section \S~\ref{hppa:sec:injectsurj}, we record some elementary facts about arbitrary models of subsystems of~${\tt BL}^{2}$, focusing in particular on the fact that arbitrary models of the hyperarithmetic subsystems of~${\tt BL}^{2}$ require the existence of injective non-surjective functions (cf. Proposition~\ref{HPPA:prop:whatitporves}). Such functions are important both because they are used to define the sentence~$\mathrm{Inf}$ (cf. Definition~\ref{HPPA:eqn:defn:inf}) and because such functions are not required to exist by the hyperarithmetic subsystems of~${\tt HP}^{2}$ (cf. Remark~\ref{hppa:rmk:thecomparisonsadfasd}). Further, in the preliminary section \S~\ref{HPPA:sbusec}, we review some elementary facts about hyperarithmetic theory, which we will employ in \S~\ref{HPPA:ehdfasdfasdfsd}. We also use these facts to fill in some parts of the provability relation (cf. Propositions~\ref{HPPA:appdfadsfad}-\ref{HPPA:prop:28} and Figure~\ref{HPPA:figure1}). Finally, in \S~\ref{HPPA:ehdfasdfasdfsd}, we turn to the main results of this section, namely the aforementioned Theorem~\ref{HPPA:thm:main} and Corollary~\ref{HPPA:thm:main:cor}.

\subsection{Generalities on Models of Subsystems of ${\tt BL}^{2}$}\label{hppa:sec:injectsurj}

\begin{prop}\label{hppa:propaboutarithm}
Suppose that $Y\subseteq M$ is definable with parameters by an arithmetical formula in the structure $(M, S_{1}, S_{2}, \ldots, \partial)$ (resp. in the structure $(M, S_{1}, S_{2}, \ldots, \#)$). Then $Y$ is definable with parameters by an arithmetical formula that does not contain any instances of $\partial$ (resp. does not contain any instances of $\#$).
\end{prop}
\begin{proof}
If $Y\subseteq M$ is definable in $(M, S_{1}, S_{2}, \ldots, \partial)$ by an arithmetical formula $\varphi$, and if $\partial(P)$ appears in $\varphi$, then $P$ is not free in $\varphi$ but rather is a parameter from $S_{1}$ and hence $a=\partial(P)$ is a parameter from~$M$. So, replacing parameters from $S_{1}$ with parameters from $M$, it follows that the set $Y$ is also definable by an arithmetical formula that does {\it not} contain any instances of $\partial$.
\end{proof}

\begin{prop}\label{HPPA:prop:truethat}
Suppose that $M$ is a structure and $\partial: D(M)\rightarrow M$ is an injection, where $D(M^{n})$ is the definable subsets of $M^{n}$. Then $(M, D(M), D(M^{2}), \ldots, \partial)$ is a model of ${\tt ABL}_{0}$.
\end{prop}
\begin{proof}
It is a model of Basic~Law~V since $\partial$ is an injection (cf. discussion subsequent to (\ref{HPPA:eqn:introooooo:bl5})). Further, it satisfies the arithmetical comprehension schema, since if $X\subseteq M$ is defined by an arithmetical formula, then by Proposition~\ref{hppa:propaboutarithm} it is defined by an arithmetical formula which does not include any instances of $\partial$. Hence, since $D(M)$  is closed under arithmetical comprehension, it follows that $X$ is in $D(M)$, so that the structure $(M, D(M), D(M^{2}), \ldots, \partial)$ satisfies the arithmetical comprehension schema.
\end{proof}

\begin{prop}\label{HPPA:prop:theinjection} Suppose that $(M, S_{1}, S_{2}, \ldots, \partial)$ is a model of ${\tt \Delta^{1}_{1}-BL}_{0}$. (a) Then there is a injective function $s:M\rightarrow M$ such that $s(x)=\partial(\{x\})$ and such that $\mathrm{graph}(s)$ is in $S_{2}$. (b) Further, there is a function  $s:M^{n}\rightarrow M$ such that $s(x_{1}, \ldots, x_{n})=\partial(\{x_{1}, \ldots, x_{n}\})$ and such that $\mathrm{graph}(s)$ is in $S_{n+1}$.
\end{prop}
\begin{proof}
The proof of (b) is identical to the proof of (a), so we present only the proof of (a). It suffices to show three things: first, that the graph of this function is $\Delta^{1}_{1}$, second that this function is well-defined and total, and third that the function is injective. Note that the following $\Sigma^{1}_{1}$ and $\Pi^{1}_{1}$-definitions of $s(x)=y$ agree:
\begin{equation}
[\exists \; X \;  (\forall \; z \; z\in X \leftrightarrow z=x) \; \& \; \partial X=y]\Longleftrightarrow [\forall \; Y \;  (\forall \; z \; z\in Y \leftrightarrow z=x) \rightarrow \partial Y=y]
\end{equation}
Suppose that the left-hand-side of this equation holds and that $Y=\{x\}$. Then $Y=X$ and hence $\partial(Y)=\partial(X)=y$. Conversely, suppose that the right-hand-side of this equation holds. By arithmetical comprehension, form the set $X=\{x\}$. Then by the right-hand-side it is the case that $\partial(X)=y$. Hence, by $\Delta^{1}_{1}$-comprehension, there is an $s$ such that $s(x,y)$ if and only if both the left-hand-side and the right-hand-side of the above equation holds with respect to $x$ and $y$. To see that the function is well-defined, suppose that the left-hand-side holds both of $x$ and $y$ and of~$x$ and~$z$. By arithmetical comprehension, form the set $Y=\{x\}$. Then the right-hand-side implies that $y=\partial(Y)=z$. Hence, the function is well-defined. Further, it is everywhere defined because given $x$ one can use arithmetical comprehension to form $X=\{x\}$, and hence $x$ and $y=\partial(X)$ will satisfy the right-hand-side. Finally, to see that the function $X$ is injective, suppose that $s(x)=s(y)$. Then $\partial(\{x\})=\partial(\{y\})$. By Basic~Law~V, it follows that $\{x\}=\{y\}$ and hence that $x=y$.
\end{proof}

\begin{rmk}
The following proposition generalizes the construction in the Russell Paradox (cf.~Proposition~(\ref{HPPA:form:thebadformtheprop})). Note that in the following proposition, the term $\mathrm{rng}\partial$ is employed to designate the range of the function $\partial$. However, this set need not exist in the second-order parts of any of the models under consideration, even though it is is defined by a $\Sigma^{1}_{1}$-formula in these models.
\end{rmk}

\begin{prop}\label{HPPA:prop:genrussel}
Suppose that $(M, S_{1}, S_{2}, \ldots, \partial)$ is a model of ${\tt \Delta^{1}_{1}-BL}_{0}$. For every $A$ in $S_{1}$ such that $A \subseteq \mathrm{rng}\partial$, there is $B$ in $S_{1}$ such that $B\subseteq A$ and $\partial B\in \mathrm{rng}\partial-A$.
\end{prop}
\begin{proof}
First we claim that for all $x$ it is the case that
\begin{equation}\label{HPPA:eqn:1324132412342134}
[\exists \; X \; x\in A \; \& \; \partial X=x \;\&\; x\notin X] \Longleftrightarrow [\forall \; Y \; x\in A\; \& \; (\partial Y=x \rightarrow x\notin Y)]
\end{equation}
Suppose that the left-hand-side holds, i.e., suppose that $x\in A \; \& \; \partial X=x \;\&\; x\notin X$, and further suppose that $Y$ is such that $\partial Y=x$. Then $\partial X=x=\partial Y$ and Basic~Law~V implies that $X=Y$. Conversely, suppose that the right-hand-side holds, i.e., suppose it is the case that $\forall \; Y \; x\in A\; \& \; (\partial Y=x \rightarrow x\notin Y)$. Since $x\in A\subseteq \mathrm{rng}\partial$, there is $X$ such that $\partial X=x$, and hence $x\notin X$. The claim is proved, and, hence, by the $\Delta^{1}_{1}$-Comprehension Schema, there exists $B$ such that $x\in B$ if and only if both the left-hand-side and right-hand-side of (\ref{HPPA:eqn:1324132412342134}) hold with respect to $x$. Note that it follows automatically from the left-hand-side that $B\subseteq A$. So it remains to show that $\partial B\in \mathrm{rng}\partial-A$. Suppose not. Then $\partial B\in \mathrm{rng}\partial \cap A$. Then either $\partial B\in B$ or $\partial B \notin B$. If $\partial B\in B$ then by right-hand-side we have $\partial B\notin B$, which is a contradiction. If $\partial B \notin B$, then by the left-hand-side we have that $\forall \; X \; \partial B\notin A \; \vee \; \partial X \neq \partial B \;\vee\; \partial B\in X$. Applying this to $X=B$ we have that   $\partial B\notin A \; \vee \; \partial B\neq \partial B \;\vee\; \partial B\in B$. Since by hypothesis we have that $\partial B \in \mathrm{rng}\partial \cap A$, we must conclude that $\partial B\in B$, which again contradicts our supposition. Hence, in fact, $\partial B\in \mathrm{rng}\partial-A$.
\end{proof}

%%% Need to be clearer about whether the injection actually exists

\begin{rmk}\label{hppa:rmk:thecomparisonsadfasd}
The following corollary is important because it shows that satisfying ${\tt \Delta^{1}_{1}-BL}_{0}$ requires the existence of injective non-surjective functions. As we note in Proposition~\ref{HPPA:prop:whatitporves222} and later in Corollary~\ref{hppa:corcorcor3434123431}, this is not the case with ${\tt ABL}_{0}$ and ${\tt \Delta^{1}_{1}-HP}_{0}$.
\end{rmk}

\begin{cor}\label{HPPA:prop:whatitporves}
Suppose that $(M, S_{1}, S_{2}, \ldots, \partial)$ is a model of ${\tt \Delta^{1}_{1}-BL}_{0}$. Then there is a injective non-surjective function $s:M\rightarrow M$ such that $\mathrm{graph}(s)$ is in $S_{2}$ and such that $s(x)=\partial(\{x\})$.
\end{cor}
\begin{proof}
By Proposition~\ref{HPPA:prop:theinjection} there is an injective function $s:M\rightarrow M$ such that $\mathrm{rng}(s)\subseteq \mathrm{rng}\partial$ and such that $\mathrm{graph}(s)$ is in $S_{2}$ and such that $s(x)=\partial(\{x\})$. By Proposition~\ref{HPPA:prop:genrussel}, there is $B$ in $S_{1}$ such that $B\subseteq \mathrm{rng}(s)$ and $\partial B\in \mathrm{rng}\partial-\mathrm{rng}(s)$. Hence, $s:M\rightarrow M$ is not surjective.
\end{proof}

\begin{prop}\label{HPPA:prop:whatitporves222}

There is a structure $(M, S_{1}, S_{2}, \ldots)$ such that 
\begin{itemize}

\item[(i)] For any injection $\partial:S_{1}\rightarrow M$ it is the case that $(M, S_{1}, S_{2}, \ldots, \partial)$ models both the theory~${\tt ABL}_{0}$ as well as the sentence that expresses that there are no injective non-surjective functions $f:M\rightarrow M$.

\item[(ii)] There is no injection $\partial:S_{1}\rightarrow M$ such that $(M, S_{1}, S_{2}, \ldots, \partial)$ models ${\tt \Delta^{1}_{1}-BL}_{0}$.

\end{itemize}

\end{prop}

\begin{proof}
Let $M$ be an algebraically closed field (cf. Marker~\cite{Marker2002} Example~4.3.10 p.~140) and let $S_{n}=D(M^{n})$, i.e. the definable subsets of $M^{n}$. Suppose that $s:M\rightarrow M$ was an injective surjective function whose graph was in $S_{2}=D(M^{2})$. Then this implies that there is a definable injective non-surjective function $s: M\rightarrow M$, which contradicts Ax's Theorem (cf. Theorem~\ref{HPPA:prop:axtheorem}). For~(i), note that by Proposition~\ref{HPPA:prop:truethat}, the structure $(M, S_{1}, S_{2}, \ldots, \partial)$ is a model of ${\tt ABL}_{0}$ for any injection $\partial:D(k)\rightarrow k$. For~(ii), note that if there was such an injection $\partial:S_{1}\rightarrow M$, then by Corollary~\ref{HPPA:prop:whatitporves}, there would be an injective non-surjective $s: M\rightarrow M$ such that $\mathrm{graph}(s)$ is in $S_{2}$, which is a contradiction.
\end{proof}

\subsection{Hyperarithmetic Theory and Some Related Elementary Results}\label{HPPA:sbusec}

\begin{defn}
Suppose that $X,Y\in 2^{\omega}$. Then $X\leq_{T} Y$ if $X$ is Turing computable from $Y$ or if $X$ is $\Delta^{0,Y}_{1}$. Further, $X\leq_{a} Y$ if $X$ is arithmetical in $Y$ or if there is $n>0$ such that $X$ is $\Delta^{0,Y}_{n}$. Finally, $X\leq_{h} Y$ if $X$ is hyperarithmetic in $Y$ or if $X$ is $\Delta^{1,Y}_{1}$ (For computational definitions of these reducibilities and proofs that they correspond with the relevant definability notion, see respectively Soare~\cite{Soare1987} p.~64, Odifreddi~\cite{Odifreddi1989aa} p.~375, Sacks~\cite{Sacks1990} p.~44).
\end{defn}

\begin{defn}\label{hppa:defn:standard3333}
Suppose that $Y\in 2^{\omega}$. Then define
\begin{align}
& \mathrm{REC}(Y)=\{X\in 2^{\omega}: X\leq_{T} Y\} \\
& \mathrm{ARITH}(Y)=\{X\in 2^{\omega}: X\leq_{a} Y\} \\
& \mathrm{HYP}(Y)=\{X\in 2^{\omega}: X\leq_{h} Y\}
\end{align}
Further, let $\mathrm{REC}=\mathrm{REC}(\emptyset)$ and $\mathrm{ARITH}=\mathrm{ARITH}(\emptyset)$ and $\mathrm{HYP}=\mathrm{HYP}(\emptyset)$ (cf. Simpson~\cite{Simpson2009aa} Remark~I.7.5. p.~25, Example~I.11.2 p.~39).
\end{defn}

\begin{rmk}
Recall that structures in the language of ${\tt HP}^{2}$ and ${\tt BL}^{2}$ have the form $(M, S_{1}, S_{2}, \ldots, \#)$, where $S_{n}\subseteq P(M^{n})$ and $\#:S_{1}\rightarrow M$ (cf. equation~(\ref{HPPA:eqn:standardmodel2})). If $\#:\mathrm{HYP}(Y)\rightarrow \omega$, then $(\omega, \mathrm{HYP}(Y), \#)$ will be used as an abbreviation for the structure $(\omega, S_{1}, S_{2}, \ldots, \#)$, where $S_{n}\subseteq P(\omega^{n})$ is the set of $n$-ary relations whose graph is in $\mathrm{HYP}(Y)$ under any standard computable pairing function. Similarly, in what follows, we will sometimes use the abbreviations  $(\omega, \mathrm{REC}(Y), \#)$ and  $(\omega, \mathrm{ARITH}(Y), \#)$.
\end{rmk}

%%% Need to find the corresponding result in Simpons!

\begin{prop}
The relation $X\leq_{h} Y$ is $\Pi^{1}_{1}$.
\end{prop}
\begin{proof}
See Sacks~\cite{Sacks1990} p.~45.
\end{proof}

\begin{thm}\label{HPPA:thm:kleenerestricted} (Kleene's Theorem on Restricted Quantification) \index{Kleene's Theorem on Restricted Quantification, Theorem~HPPA:thm:kleenerestricted}
Suppose that $\varphi(X,Y)$ is a $\Pi^{1}_{1}$ predicate. Then $\exists \; X\leq_{h} Y \; \varphi(X,Y)$ is a $\Pi^{1}_{1}$-predicate. Moreover, this is provable in ${\tt \Pi^{1}_{1}-CA_{0}}$.
\end{thm}
\begin{proof}
See Kleene~\cite{Kleene1959} and Moschovakis~\cite{Moschovakis1980} Theorem~4D.3~p.~220. That this theorem is provable in ${\tt \Pi^{1}_{1}-CA}_{0}$ was noted by Simpson~\cite{Simpson2009aa} VIII.3.20 p.~330.
\end{proof}

%%% Need to find the corresponding result in Simpons!

\begin{thm}\label{HPPA:thm:specgandy} (Spector-Gandy Theorem)\index{Spector-Gandy Theorem, Theorem~\ref{HPPA:thm:specgandy}}
Suppose that $\varphi(Y)$ is a $\Pi^{1}_{1}$-predicate. Then there is an arithmetic predicate $\psi(X,Y)$ such that $\varphi(Y)\leftrightarrow \exists \; X\leq_{h} Y \; \psi(X,Y)$.
\end{thm}
\begin{proof}
See Spector and Gandy (\cite{Spector1959}, \cite{Gandy1960}), Sacks~\cite{Sacks1990} Theorem~III.3.5 p.~61 and Exercise~III.3.13 p.~62.
\end{proof}

\begin{rmk}
The following proposition is non-trivial only because the second-order quantifiers must be evaluated with respect to the second-order part~$S_{1}\subseteq P(\omega)$ of the structure $(\omega, S_{1})$ and not with respect to $P(\omega)$ itself. For instance, one cannot infer that $(\omega, \mathrm{HYP}(Y))\models \neg  {\tt \Pi^{1}_{1}-CA}_{0}$ simply from the fact that $\mathcal{O}^{Y}$ is $\Pi^{1}_{1}$ but not $\Sigma^{1}_{1}$, since to say this is merely to say that $\mathcal{O}^{Y}$ is $\Pi^{1}_{1}$-definable but not $\Sigma^{1}_{1}$-definable on the structure $(\omega, P(\omega))$.
\end{rmk}

\begin{prop}\label{HPPA:appdfadsfad}
Suppose that $Y\in 2^{\omega}$. Then $(\omega, \mathrm{ARITH}(Y))\models {\tt ACA}_{0}+\neg {\tt \Delta^{1}_{1}-CA}_{0}$ and $(\omega, \mathrm{HYP}(Y))\models {\tt \Sigma^{1}_{1}-AC}_{0}+\neg {\tt \Pi^{1}_{1}-CA}_{0}$.
\end{prop}
\begin{proof}
For the fact that $(\omega, \mathrm{ARITH}(Y))\models  {\tt ACA}_{0}$, see Simpson~\cite{Simpson2009aa} Theorem~VIII.1.13 p.~313. Suppose that $(\omega, \mathrm{ARITH}(Y))\models {\tt \Delta^{1}_{1}-CA}_{0}$. But note that
\begin{align}\label{hppa:align:delta11defntruth}
(n,m)\in Y^{(\omega)} & \Longleftrightarrow \exists \; X \in \mathrm{ARITH}(Y) \; X =\oplus_{i=1}^{n} Y^{(i)} \; \& \; m\in X \notag \\ & \Longleftrightarrow \forall \; X \in \mathrm{ARITH}(Y) \; X =\oplus_{i=1}^{n} Y^{(i)} \rightarrow m\in X
\end{align}
and hence $Y^{(\omega)}\in \mathrm{ARITH}(Y)$, which would contradict Tarski's Theorem on Truth. Hence, in fact $(\omega, \mathrm{ARITH}(Y))\models \neg {\tt \Sigma^{1}_{1}-AC}_{0}$. For the fact that 
 $(\omega, \mathrm{HYP}(Y))\models {\tt \Sigma^{1}_{1}-AC}_{0}$, see Simpson~\cite{Simpson2009aa} Theorem~VIII.4.5 p.~334 and Theorem~VIII.4.8 p.~335. This proof uses Kleene's Theorem on Restricted Quantification~\ref{HPPA:thm:kleenerestricted}, and below in Theorem~\ref{HPPA:thm:main} we will emulate this proof in the setting of ${\tt BL}^{2}$. Suppose for the sake of contradiction that $(\omega, \mathrm{HYP}(Y))\models {\tt \Pi^{1}_{1}-CA}_{0}$. Since $\mathcal{O}^{Y}$ is $\Pi^{1,Y}_{1}$, by the Spector-Gandy Theorem~\ref{HPPA:thm:specgandy}, there is an arithmetic predicate $\psi(n,X,Y)$ such that $n\in \mathcal{O}^{Y}\Longleftrightarrow \exists \; X\leq_{h} Y \; \psi(n,X,Y)$. Then $\mathcal{O}^{Y}$ is $\Sigma^{1}_{1}$-definable on $(\omega, \mathrm{HYP}(Y))$ and hence exists in $\mathrm{HYP}(Y)$ by ${\tt \Pi^{1}_{1}-CA}_{0}$, which contradicts that $\mathcal{O}^{Y}$ is not in $\mathrm{HYP}(Y)$.
\end{proof}

\begin{cor}\label{HPPA:cor:specgandy}
Suppose that there is a $\Pi^{1}_{1}$-formula $\theta(X,Y,Z)$ such that for all $Z\in 2^{\omega}$ the set $G_{Z}=\{(X,Y)\in 2^{\omega}\times 2^{\omega}: \theta(X,Y,Z)\}$ is the graph of a function $g_{Z}:\mathrm{HYP}(Z)\rightarrow \mathrm{HYP}(Z)$. Then the graph $G_{Z}$ of $g_{Z}$ is $\Sigma^{1}_{1}$-definable in the structure $(\omega, \mathrm{HYP}(Z))$ uniformly in $Z$.
\end{cor}
\begin{proof}
Note that since $g_{Z}:\mathrm{HYP}(Z)\rightarrow \mathrm{HYP}(Z)$, we have that for all $X,Y,Z\in 2^{\omega}$
\begin{equation}
\theta(X,Y,Z) \Longrightarrow X\oplus Y\leq_{h} Z
\end{equation}
By the Spector-Gandy Theorem~\ref{HPPA:thm:specgandy}, there is an arithmetical predicate $\psi(X,Y,Z,W)$ such that for all $X,Y,Z\in 2^{\omega}$
\begin{equation}
\theta(X,Y,Z) \Longleftrightarrow \exists \; W\leq_{h} X\oplus Y\oplus Z  \; \psi(X,Y,Z,W)
\end{equation}
Putting the two previous equations together, we have that for all $X,Y,Z\in 2^{\omega}$
\begin{equation}
\theta(X,Y,Z) \Longleftrightarrow \exists \; W\leq_{h} Z  \; \psi(X,Y,Z,W)
\end{equation}
Then for all $X,Y,Z\in 2^{\omega}$
\begin{equation}
g_{Z}(X)=Y \Longleftrightarrow (\omega, \mathrm{HYP}(Z))\models \exists \; W\; \psi(X,Y,Z,W)
\end{equation}
Hence, in fact the graph $G_{Z}$ of $g_{Z}$ is $\Sigma^{1}_{1}$-definable in the structure $(\omega, \mathrm{HYP}(Z))$ uniformly in~$Z$.
\end{proof}

\begin{thm}\label{hppa:thm:kondo} (Kondo's Uniformization Theorem) \index{Kondo's UniformizationTheorem, Theorem~\ref{hppa:thm:kondo}}Suppose that $\varphi(X,Y)$ is a~$\Pi^{1}_{1}$ predicate. Then there is a $\Pi^{1}_{1}$-predicate $\varphi^{\prime}(X,Y)$ such that
\begin{align}
& \forall \; X, Y \; [\varphi^{\prime}(X,Y)\rightarrow \varphi(X,Y)] \label{hppa:eqn:compareasdfasdf3} \\
& \forall \; X \; [\exists \; Y \; \varphi(X,Y)]\rightarrow [\exists ! Y\; \varphi^{\prime}(X,Y)] \label{hppa:eqn:compareasdfasdf4}
\end{align}
Moreover, this is provable in ${\tt \Pi^{1}_{1}-CA}_{0}$.
\end{thm}
\begin{proof}
See Moschovakis~\cite{Moschovakis1980} pp.~235-236. Simpson notes that Kondo's theorem is provable in ${\tt \Pi^{1}_{1}-CA}_{0}$ (cf. \cite{Simpson2009aa}~Theorem~VI.2.6~p.~225).
\end{proof}

\begin{rmk}
The following two propositions use some of the preceding material to fill in some information about the provability relation (cf. Figure~\ref{HPPA:figure1}).
\end{rmk}

\begin{prop}\label{HPPA:prop:280} There are models of ${\tt ABL}_{0}+\neg {\tt \Delta^{1}_{1}-BL}_{0}$.
\end{prop}
\begin{proof}
Choose any injection $\partial: \mathrm{ARITH}\rightarrow \omega$. Then by Proposition~\ref{HPPA:prop:truethat} the structure $(\omega, \mathrm{ARITH}, \partial)$ is a model of ${\tt ABL}_{0}$. Further, since the graphs of addition and multiplication are in $\mathrm{ARITH}$, if  $(\omega, \mathrm{ARITH}, \partial)\models {\tt \Delta^{1}_{1}-BL}_{0}$, then one would have that $\emptyset^{(\omega)}\in \mathrm{ARITH}$ (cf. equation~(\ref{hppa:align:delta11defntruth})), which would contradict Tarski's theorem on truth.
\end{proof}

\begin{rmk}
The construction in the following proposition is the same construction as Boolos used to prove the interpretability of ${\tt HP}^{2}$ in ${\tt PA}^{2}$ (cf. the proof of Theorem~\ref{hppa:thm:boolos}).
\end{rmk}

\begin{prop}\label{HPPA:prop:28} There are models of ${\tt AHP}_{0}+\neg {\tt \Delta^{1}_{1}-HP}_{0}$ and ${\tt \Sigma^{1}_{1}-PH}_{0}+\neg {\tt \Pi^{1}_{1}-HP}_{0}$ and ${\tt \Delta^{1}_{1}-HP}_{0}+\neg {\tt \Sigma^{1}_{1}-PH}_{0}$
\end{prop}
\begin{proof}
Define a function $\#:\mathrm{ARITH}\rightarrow \omega$ by $\#X=0$ if $X$ is infinite and $\#X=\left|X\right|+1$ if $X$ is finite. By Simpson~\cite{Simpson2009aa} Lemma~II.3.6 p.~70, ${\tt ACA}_{0}$ proves that any two infinite sets are bijective, and hence $(\omega, \mathrm{ARITH}, \#)$ is a model of Hume's~Principle.  Further, it satisfies the arithmetical comprehension schema, since if $X\subseteq \omega$ is defined by an arithmetical formula, then by Proposition~\ref{hppa:propaboutarithm} it is defined by an arithmetical formula that does not include any instances of $\#$. Hence, since $\mathrm{ARITH}$  is closed under arithmetical comprehension, it follows that $X$ is in $\mathrm{ARITH}$, so that the structure $(\omega, \mathrm{ARITH}, \#)$ satisfies the arithmetical comprehension schema. Since $\emptyset^{(\omega)}\notin \mathrm{ARITH}$ but $\emptyset^{(\omega)}$ is $\Delta^{1}_{1}$-definable over $\mathrm{ARITH}$ using the graphs of addition and multiplication as parameters (cf. equation~(\ref{hppa:align:delta11defntruth})), we have that  $(\omega, \mathrm{ARITH}, \#)$ is a model of ${\tt AHP}_{0}+\neg {\tt \Delta^{1}_{1}-HP}_{0}$. Similarly, using the fact that the graph of $\#$ is arithmetical, we can argue that  $(\omega, \mathrm{HYP}, \#)$ is a model of ${\tt \Sigma^{1}_{1}-PH}_{0}+\neg {\tt \Pi^{1}_{1}-HP}_{0}$. Likewise, Steel constructs a sequence of reals $G_{n}$ such that $(\omega, \bigcup_{n=1}^{\infty }\mathrm{HYP}^{G_{1}\oplus \cdots \oplus G_{n}})$ is a model of ${\tt \Delta^{1}_{1}-CA}_{0}+\neg {\tt \Sigma^{1}_{1}-AC}_{0}$ (\cite{Steel1978}~Theorem~4 pp.~68~ff), and we can argue as before that $(\omega, \bigcup_{n=1}^{\infty }\mathrm{HYP}^{G_{1}\oplus \cdots \oplus G_{n}}, \#)$ is a model of  ${\tt \Delta^{1}_{1}-HP}_{0}+\neg {\tt \Sigma^{1}_{1}-PH}_{0}$.
\end{proof}

\begin{rmk}
The following two propositions use elementary considerations about arithmetical sets (cf. Definition~\ref{hppa:defn:standard3333}) to record some observations about natural functions whose existence cannot be proven in ${\tt ABL}_{0}$ or ${\tt AHP}_{0}$. For the motivation for these propositions, see  \S~2.2, and in particular around equation~(\ref{HPPA:eqn:counterthecounter}). The only reason for including these propositions here (as opposed to earlier) is that it seemed prudent to delay their proof until the arithmetical sets had been introduced, which we did earlier in this section (cf. Definition~\ref{hppa:defn:standard3333}). Note that the construction in the following proposition is analogous to the construction used by Boolos to prove the interpretability of ${\tt HP}^{2}$ in ${\tt PA}^{2}$ (cf. the proof of Theorem~\ref{hppa:thm:boolos}).
\end{rmk}

\begin{prop}\label{hppa:binaryiswierd1}
There is a structure $M$ and a function $\#: D(M)\rightarrow M$, where $D(M^{n})$ is the definable subsets of $M^{n}$, such that $(M, D(M), D(M^{2}), \ldots, \#)$ is a model of ${\tt AHP}_{0}$, and further there is binary relation $R$ in $D(M^{2})$ such that the set $\{(n,m): \#(R_{n})=m\}$ does not exist in $D(M^{2})$, where $R_{n}=\{x: Rnx\}$.
\end{prop}
\begin{proof}
Let $M$ be the standard model of first-order arithmetic $(\omega,+,\times)$ so that $D(M)$ are the arithmetical sets $\mathrm{ARITH}$. Choose a real $Z\notin \mathrm{ARITH}$, such as $\emptyset^{(\omega)}$, and enumerate $Z$ as $z_{0}, z_{1}, z_{2}, \ldots$. Define the function $\#:\mathrm{ARITH}\rightarrow \omega$ by $\#(X)=z_{n}$ if $X$ is finite and $\left|X\right|=n$ and define $\#(X)=z_{\infty}$ for some fixed $z_{\infty}\notin Z$ if $X$ is infinite. This structure satisfies arithmetical comprehension, since if $X\subseteq M$ is defined by an arithmetical formula, then by Proposition~\ref{hppa:propaboutarithm} it is defined by an arithmetical formula which does not include any instances of $\#$. Hence, since $D(M)$  is closed under arithmetical comprehension, it follows that $X$ is in $D(M)$, so that the structure $(M, D(M), D(M^{2}), \ldots, \#)$ satisfies the arithmetical comprehension schema.  Further, by Simpson~\cite{Simpson2009aa} Lemma~II.3.6 p.~70, ${\tt ACA}_{0}$ proves that any two infinite sets are bijective, and hence the structure $(M, D(M), D(M^{2}), \ldots, \#)$ is a model of Hume's~Principle. Hence, $(M, D(M), D(M^{2}), \ldots, \#)$ is a model of ${\tt AHP}_{0}$. Consider now the set $R=\{(n,m): m< n\}$, which is clearly arithmetical and so exists in $D(M^{2})$. Then $R_{n}=\{x: Rnx\}=\{0, \ldots, n-1\}$ and $\#(R_{n})=z_{n}$. Then the set 
\begin{equation}
\{(n,m): \#(R_{n})=m\}=\{(n,m):z_{n}=m\}
\end{equation}
is equal to the graph of $n\mapsto z_{n}$, which is not arithmetical: for, if it were arithmetical, then its range $Z$ would be be arithmetical, which contradicts the hypothesis on $Z$.
\end{proof}

\begin{prop}\label{hppa:binaryiswierd2}
There is a structure $M$ and an injection $\partial: D(M)\rightarrow M$, where $D(M^{n})$ is the definable subsets of $M^{n}$, such that $(M, D(M), D(M^{2}), \ldots, \partial)$ is a model of ${\tt ABL}_{0}$, and further there is binary relation $R$ in $D(M^{2})$ such that the set $\{(n,m): \partial(R_{n})=m\}$ does not exist in $D(M^{2})$, where $R_{n}=\{x: Rnx\}$.
\end{prop}
\begin{proof}
Let $M$ be the standard model of first-order arithmetic $(\omega,+,\times)$ so that $D(M)$ are the arithmetical sets $\mathrm{ARITH}$. Choose a real $Z\notin \mathrm{ARITH}$, such as $\emptyset^{(\omega)}$, and enumerate $Z$ as $z_{0}, z_{1}, z_{2}, \ldots$. Choose an injection $\partial:\mathrm{ARITH}\rightarrow \omega$ such that $\partial(\{n\})=z_{n}$, which we can do since $Z$ is coinfinite (since it is not arithmetical). Then by Proposition~\ref{HPPA:prop:truethat}, the structure $(M, D(M), D(M^{2}), \ldots, \partial)$ is a model of ${\tt ABL}_{0}$. Consider now the diagonal $R=\{(n,m): n=m\}$ which is clearly arithmetical and so exists in $D(M^{2})$. Then $R_{n}=\{x:Rnx\}=\{n\}$ and $\partial(R_{n})=\partial(\{n\})=z_{n}$. Then the set 
\begin{equation}
\{(n,m): \partial(R_{n})=m\}=\{(n,m):z_{n}=m\}
\end{equation}
is equal to the graph of $n\mapsto z_{n}$, which is not arithmetical: for, if it were arithmetical, then its range $Z$ would be be arithmetical, which contradicts the hypothesis on $Z$.
\end{proof}

\subsection{Standard Models of the Hyperarithmetic Subsystems of ${\tt BL}^{2}$}\label{HPPA:ehdfasdfasdfsd}

%%% Need to say something about how this compares to FW construction.

\begin{rmk}
Recall from Proposition~\ref{HPPA:prop:theinjection} that ${\tt \Delta^{1}_{1}-BL}_{0}$ proves the existence of the graph of an injective function $s:M\rightarrow M$ such that $s(x)=\partial(\{x\})$. This function is is mentioned in the following axiom.
\end{rmk}

\begin{defn}\label{HPPA:eqn:defn:inf}
The following sentence $\mathrm{Inf}$ is a sentence in the signature of ${\tt BL}^{2}$:\index{$\mathrm{Inf}$, the sentence in signature of ${\tt BL}^{2}$, Definition~\ref{HPPA:eqn:defn:inf}}

\begin{align}
\mathrm{Inf} \equiv & \; \exists \; s: M\rightarrow M \; [\forall \; x \; s(x)=\partial(\{x\})] 
 \; \& \; \exists \; N \; [\partial(\emptyset)\in N \; \& \;\forall \; x \; x\in N \rightarrow sx\in N] \notag \\
&\;  \& \;\forall \; N^{\prime} \; [\partial(\emptyset)\in N^{\prime} \; \&\; \forall \; x \; x\in N^{\prime} \rightarrow sx\in N^{\prime}] \rightarrow N\subseteq N^{\prime} \notag \\
&\; \& \;\exists \; \oplus: N^{2}\rightarrow N \; \exists \; \otimes: N^{2}\rightarrow N \; \exists \; \preceq \; \subseteq\; N^{2} \; [(N, \partial(\emptyset), s, \oplus, \otimes, \preceq)\models (\mathrm{Q1})-(\mathrm{Q8})]
\end{align}
Intuitively, $\mathrm{Inf}$ says that there is a smallest set $N$ which contains the zero element $\partial(\emptyset)$ and which is closed under the successor function $s(x)=\partial(\{x\})$ and which has addition and multiplication functions $\oplus$ and $\otimes$ and an ordering relation $\preceq$ which satisfy the eight axioms of Robinson's~Q.
\end{defn}

\begin{rmk}\label{hppa:rmk:fw}
The following theorem and its corollary is the main result of \S~\ref{HPPA:secdfadsfsadfkkk}. Recall that the Russell paradox showed that ${\tt BL}_{0}$ and ${\tt \Pi^{1}_{1}-BL}_{0}$ is inconsistent (cf. Proposition~(\ref{HPPA:form:thebadformtheprop})). Recently Ferreira and Wehmeier (\cite{Ferreira2002}) showed that ${\tt \Delta^{1}_{1}-BL}_{0}$ is consistent, using Barwise and Schlipf's recursively-saturated model construction. In \S~\ref{HPPA:gendafdsfad}, we present a generalization of this construction (cf. Theorem~\ref{HPPA:thm:metatheorem}), which we apply to ${\tt \Delta^{1}_{1}-BL}_{0}$ and ${\tt \Delta^{1}_{1}-HP}_{0}$ (cf. Proposition~\ref{HPPA:lineeeddafdsfa}, Corollary~\ref{HPPA:cor:thebomb}, and Theorem~\ref{HPPA:thmasdfadsf}). However, the recursively-saturated model construction does not provide one with natural models, simply because most natural structures are not recursively saturated (unless of course they are saturated {\it tout court}). Hence, this raises the question of whether there are natural models of ${\tt \Delta^{1}_{1}-BL}_{0}$. The following theorem constructs a model of ${\tt \Delta^{1}_{1}-BL}_{0}$ which is mutually interpretable with the minimal $\omega$-model of ${\tt \Delta^{1}_{1}-CA}_{0}$, namely, the model whose second-order part consists of the hyperarithmetic sets.
\end{rmk}

\begin{thm}\label{HPPA:thm:main} For any real $Y\in 2^{\omega}$, there is a map $\partial_{Y}:\mathrm{HYP}(Y)\rightarrow \omega$ with $\Pi^{1,Y}_{1}$-graph such that~(i) the structure $M_{Y}=(\omega, \mathrm{HYP}(Y), \partial_{Y})$ is a model of (a) ${\tt \Sigma^{1}_{1}-LB}_{0}$ and (b) the sentence~$\mathrm{Inf}$, and such that~(ii) the two structures 
\begin{equation}
M_{Y}=(\omega, \mathrm{HYP}(Y), \partial_{Y}), \hspace{5mm}(\omega, 0,s,+,\times, \leq, \mathrm{HYP}(Y))
\end{equation}
are mutually interpretable uniformly in~$Y$, in the following sense: (a) the map $\partial_{Y}:\mathrm{HYP}(Y)\rightarrow \omega$ is definable in $(\omega, \mathrm{HYP}(Y), 0,s,+,\times, \leq)$ uniformly in~$Y$, and (b) an isomorphic copy $H_{Y}$ of the structure $(\omega, \mathrm{HYP}(Y), 0,s,+,\times, \leq)$ is definable in the structure $M_{Y}=(\omega, \mathrm{HYP}(Y), \partial_{Y})$ uniformly in $Y$. Moreover, all these facts are provable in ${\tt \Pi^{1}_{1}-CA}_{0}$.
\end{thm}
\begin{proof}
Define $P(Y\oplus X,n)$ iff $X\in \mathrm{HYP}(Y)$ and $n=\langle a,e\rangle$ is a hyperarithmetical-in-$Y$ index of $X$:
\begin{equation}
P(Y\oplus X,\langle a, e\rangle) \equiv X\in \mathrm{HYP}(Y) \; \& \; a\in \mathcal{O}^{Y} \; \&\; X = \{e\}^{H_{a}^{Y}}
\end{equation}
Since the relation $X\in \mathrm{HYP}(Y)$ is $\Pi^{1}_{1}$ and membership in $H_{a}^{Y}$ is $\Delta^{1,Y}_{1}$ for $a\in \mathcal{O}^{Y}$
%***
, we have that $P(Y\oplus X,n)$ is a $\Pi^{1}_{1}$-predicate. By Kondo uniformization (Theorem~\ref{hppa:thm:kondo}), there is a $\Pi^{1}_{1}$-uniformization $P^{\prime}$ of $P$. For $Y\in 2^{\omega}$, define $\partial_{Y}(X)=n$ if and only if $P^{\prime}(Y\oplus X,n)$. Since $\partial_{Y}(X)=n$ implies that $n$ is a hyperarithmetical-in-$Y$ index of $X$, we have that $\partial_{Y}: \mathrm{HYP}(Y)\rightarrow \omega$ is an injection and hence $M_{Y}=(\omega, \mathrm{HYP}(Y), \partial_{Y})$ is a model of Basic~Law~V. Note that since $\partial_{Y}:\mathrm{HYP}(Y)\rightarrow \omega$ has a $\Pi^{1,Y}_{1}$-graph, the Corollary to the Spector-Gandy Theorem (cf. Corollary~\ref{HPPA:cor:specgandy}) implies that  $\partial_{Y}:\mathrm{HYP}(Y)\rightarrow \omega$ is definable in the structure $(\omega, \mathrm{HYP}(Y),0,s,+,\times, \leq)$, and this establishes~(ii)(a).

To establish~(i)(a), note that since $\partial_{Y}:\mathrm{HYP}(Y)\rightarrow \omega$ is an injection, it follows that $M_{Y}=(\omega, \mathrm{HYP}(Y), \partial_{Y})$ is a model ${\tt ABL}_{0}$ (as in the proof of Proposition~\ref{HPPA:prop:truethat}). To see that it also models the $\Sigma^{1}_{1}$-choice schema~(\ref{HPPA:eqn:thechoiceschema}), suppose that $M_{Y}\models \forall \; z\; \exists \; X\; \varphi(z,X, \partial_{Y}(X))$, where $\varphi$ is an arithmetical formula. (The proof for the case where $z$ is replaced by a tuple $\overline{z}$, or where there are multiple existential set quantifiers and multiple existential relation quantifiers, or where there are parameters from the model present in $\varphi$ is exactly similar). Then $M_{Y}\models \forall \; z\; \exists \; X\; \exists \; e \; [\partial_{Y}(X)= e \; \wedge\; \varphi(z,X, e)$]. Define a relation $Q(Y \oplus \{z\},X)$ as follows:
\begin{equation}
Q(Y\oplus \{z\},X) \Longleftrightarrow X\in \mathrm{HYP}(Y) \; \&\; \exists \; e \; [\partial_{Y}(X) = e \; \wedge\; \varphi(z,X, e)]
\end{equation}
Then $Q$ is a $\Pi^{1}_{1}$-predicate. By Kondo uniformization, there is a $\Pi^{1}_{1}$-uniformization $Q^{\prime}$ of $Q$. For $Y\in 2^{\omega}$, define $q_{Y}(z)=X$ if and only if $Q^{\prime}(Y\oplus \{z\},X)$ and let
\begin{equation}
R_{Y}= \{(z, x): \exists \; X\in \mathrm{HYP}(Y) \; q_{Y}(z)=X \wedge x\in X\}
\end{equation}
Then by Kleene's Theorem on Restricted Quantification~\ref{HPPA:thm:kleenerestricted}, $R_{Y}$ is $\Pi^{1,Y}_{1}$-definable. Moreover, since $Q^{\prime}$ is a uniformization, we also have 
\begin{equation}
R_{Y}= \{(z, x): \forall \; X\in \mathrm{HYP}(Y) \; q_{Y}(z)=X \rightarrow  x\in X\}
\end{equation}
Again, by Kleene's Theorem on Restricted Quantification~(\ref{HPPA:thm:kleenerestricted}), the set $R_{Y}$ is $\Sigma^{1,Y}_{1}$-definable. Hence $R_{Y}$ is $\Delta^{1,Y}_{1}$ and so $R_{Y}\in \mathrm{HYP}(Y)$. Finally, since $Q^{\prime}$ is a uniformization, we have that $M_{Y}\models \forall \; z \; \varphi(z, (R_{Y})_{z}, \partial_{Y}((R_{Y})_{z}))$, so in fact $M_{Y}$ is a model of ${\tt \Sigma^{1}_{1}-BL}_{0}$ and this establishes~(i)(a).

To show~(i)(b) and~(ii)(b), we first prove~(ii)(b) and then note how our proof of~(ii)(b) in fact establishes~(i)(b). Recall that by Proposition~\ref{HPPA:prop:theinjection}, there is an injective function $s_{Y}:\omega\rightarrow \omega$ whose graph is in $\mathrm{HYP}(Y)$ such that $s_{Y}(n)=\partial_{Y}(\{n\})$ for all $n\in \omega$. Define an $s_{Y}$-recursive function  $f_{Y}:\omega\rightarrow \omega$: 
\begin{equation}\label{HPPA:eqn:dsafdasfdasf}
f_{Y}(0)=\partial_{Y}(\emptyset) \; \;\;\& \;\;\; f_{Y}(n+1)=s_{Y}(f_{Y}(n)) 
\end{equation}
Let $N_{Y}$ be the range of $f_{Y}$, so that both the graph of $f_{Y}$ and its range $N_{Y}$ are in $\mathrm{HYP}(Y)$. Since $N_{Y}=\mathrm{rng}(f_{Y})$ and $\mathrm{dom}(f_{Y})=\omega$, the following induction principle holds:
\begin{equation}\label{HPPA:eqn:hypinduction}
\forall \; P \; [f_{Y}(0)\in P \; \& \; \forall \; n \in \omega \; f_{Y}(n)\in P \rightarrow f_{Y}(n+1)\in P]\rightarrow N_{Y}\subseteq P
\end{equation}
%** potential ambiguity: what is range of quantifiers
Using this form of induction, one can show that $f_{Y}: \omega\rightarrow N_{Y}$ is injective, so that its inverse $f_{Y}^{-1}: N_{Y}\rightarrow \omega$~exists and is likewise in $\mathrm{HYP}(Y)$. Further, one can arithmetically define from $N_{Y}, f_{Y}$ and $f_{Y}^{-1}$ the functions $\oplus_{Y}:N^{2}_{Y}\rightarrow N_{Y}$ and $\otimes_{Y}:N^{2}_{Y}\rightarrow N_{Y}$ as follows:
\begin{equation}\label{HPPA:eqn:hyprecursion}
f_{Y}(x) \oplus_{Y} f_{Y}(y) = f_{Y}(f^{-1}_{Y}(x)+f^{-1}_{Y}(y)) \; \; \; \; \; \; \; \; \; \; f_{Y}(x) \otimes_{Y} f_{Y}(y) = f_{Y}(f^{-1}_{Y}(x)\cdot f^{-1}_{Y}(y))
\end{equation}
and then arithmetically define a relation $\preceq$ on $N^{2}_{Y}$ by
\begin{equation}\label{HPPA:eqn:dfadsafdasfadsfadsfadsfdsa}
x\preceq_{Y} y \Longleftrightarrow \exists \; z\in N_{Y} \; x\oplus_{Y} z = y
\end{equation}
Further one can extend the map to $\overline{f}_{Y}:\mathrm{HYP}(Y)\rightarrow (P(N_{Y})\cap \mathrm{HYP}(Y))$ by setting 
\begin{equation}
\overline{f}_{Y}(X)=\{f_{Y}(n): n\in X\}
\end{equation}
and define the following structure in the signature of $(\omega, 0,s,+,\times, \leq, \mathrm{HYP}(Y))$:
\begin{equation}
H_{Y} = (N_{Y}, \partial_{Y}(\emptyset), s_{Y}, \oplus_{Y}, \otimes_{Y}, \preceq_{Y}, \overline{f}_{Y}(\mathrm{HYP}(Y)))
\end{equation}
Then the functions $f_{Y}$ and $\overline{f}_{Y}$ witness that the two structures $(\omega, 0,s,+,\times, \leq, \mathrm{HYP}(Y))$ and $H_{Y}$ are isomorphic. 

Further, note that $H_{Y}$ is definable within $M_{Y}$: for, by the induction principle~(\ref{HPPA:eqn:hypinduction}) one can show that $N_{Y}$ is the unique smallest set containing $\partial_{Y}(\emptyset)$ and closed under $s_{Y}$, and using equation~(\ref{HPPA:eqn:hyprecursion}) and the induction principle (\ref{HPPA:eqn:hypinduction}) one can show that $\oplus_{Y}$ and $\otimes_{Y}$ are the unique functions on $N_{Y}$ satisfying the following recursion clauses
\begin{align}
&x\oplus_{Y} \partial_{Y}(\emptyset) = x & x\oplus_{Y} (s_{Y}(z))= s_{Y}(x \oplus_{Y} z) \\
& x\otimes_{Y} \partial_{Y}(\emptyset) =\partial_{Y}(\emptyset) & x\otimes_{Y} (s_{Y}(z))= (x\otimes_{Y} z)\oplus_{Y} x
\end{align}
Hence, since $H_{Y}$ and $(\omega, \mathrm{HYP}(Y), 0,s,+,\times, \leq)$ are isomorphic and since $H_{Y}$ is definable in $M_{Y}$, we have established~(ii)(b). Finally, note by construction that the structure $H_{Y}$ witnesses that $M_{Y}$ is a model of the axiom~$\mathrm{Inf}$, so that we have established~(i)(b).
\end{proof}

\begin{cor}\label{HPPA:thm:main:cor}
 ${\tt \Sigma^{1}_{1}-AC}_{0}\leq_{\mathrm{I}} {\tt \Sigma^{1}_{1}-LB}_{0}+\mathrm{Inf}<_{\mathrm{I}} {\tt \Pi^{1}_{1}-CA}_{0}$.
\end{cor}
\begin{proof}
Note that ${\tt \Sigma^{1}_{1}-AC}_{0}\leq_{\mathrm{I}} {\tt \Sigma^{1}_{1}-LB}_{0}+\mathrm{Inf}$ because the sentence $\mathrm{Inf}$ (cf. Definition~\ref{HPPA:eqn:defn:inf}) literally provides an interpretation. To see that $ {\tt \Sigma^{1}_{1}-LB}_{0}+\mathrm{Inf}<_{\mathrm{I}} {\tt \Pi^{1}_{1}-CA}_{0}$, note that since the previous theorem can be proven in ${\tt \Pi^{1}_{1}-CA}_{0}$, it follows that ${\tt \Pi^{1}_{1}-CA}_{0}$ proves the consistency of ${\tt \Sigma^{1}_{1}-LB}_{0}+\mathrm{Inf}$. Thus the result follows from Proposition~\ref{HPPA:conpropa}.
\end{proof}

\section{Barwise-Schlipf Models of the Hyperarithmetic Subsystems of ${\tt BL}^{2}$ and ${\tt HP}^{2}$}\label{HPPA:bssssssss}

%% Go back and check all referenes

In this section, we turn to building models of subsystems of ${\tt BL}^{2}$ and ${\tt HP}^{2}$ on top of various recursively saturated fields. In particular, \S~\ref{HPPA:gendafdsfad} is devoted to the statement and proof of a generalization of a theorem of Barwise-Schlipf and Ferreira-Wehmeier (Theorem~\ref{HPPA:thm:metatheorem}). Then in \S\S~\ref{hppa:lastsec2}-\ref{hppa:lastsec} three applications of this theorem are presented. The major result here is Corollary~\ref{HPPA:cor:thebomb}, which says that ${\tt \Sigma^{1}_{1}-PH}_{0}<_{\mathrm{I}} {\tt ACA}_{0}$, and this fills in a key piece of Figure~\ref{HPPA:figure2} about the interpretability relation.

\subsection{Generalization of the Barwise-Schlipf/Ferreira-Wehmeier Metatheorems}\label{HPPA:gendafdsfad}

The main theorem of this section (Theorem~\ref{HPPA:thm:metatheorem}) is a generalization of the way in which Barwise-Schlipf (\cite{Barwise1975aa}) built models of ${\tt \Delta^{1}_{1}-CA}_{0}$ on top of recursively saturated models of Peano arithmetic, and the way in which Ferreira-Wehmeier (\cite{Ferreira2002}) built models of ${\tt \Delta^{1}_{1}-BL}_{0}$ on top of recursively saturated structures. The new addition is the concept of a uniformly definable function $\partial: D(M)\rightarrow M$ (Definition~\ref{hppa:def:dafdsaf898999}). Subsequent to defining this notion, the definitions of definable skolem functions and recursively saturated structures are recalled, and then Theorem~\ref{HPPA:thm:metatheorem} is stated and proven.

\begin{defn}\label{hppa:def:dafdsaf898999}
Suppose that $M$ is an $L$-structure and let $D(M^{n})$ be the definable subsets of $M^{n}$. Then $\partial: D(M)\rightarrow M$ is {\it uniformly definable} if for all $L$-formula $\theta(x,\overline{y})$ with all free variables displayed and with a non-empty set $\overline{y}$ of parameter variables, there is an $L$-formula $\theta^{\prime}(x,\overline{y})$ with the same free variables, such that $\{\partial(\theta(\cdot,\overline{a} ))\}=\{b: M\models \theta^{\prime}(b,\overline{a})\}$ for all $\overline{a}\in M$,\index{Uniformly definable map $\partial: D(M)\rightarrow M$, Definition~\ref{hppa:def:dafdsaf898999}} i.e.:
\begin{equation}
\overline{a},b\in M \Longrightarrow [\partial(\theta(\cdot, \overline{a}))=b \Longleftrightarrow M\models \theta^{\prime}(b,\overline{a})]
\end{equation}
\end{defn}

\begin{defn}
Suppose that $L$ is countable and that $M$ is an $L$-structure and that $B\in 2^{\omega}$. Then $\partial: D(M)\rightarrow M$ is {\it $B$-computably uniformly definable} if it is uniformly definable and the map $\theta\mapsto \theta^{\prime}$ is $B$-computable.
\end{defn}

\begin{defn}\label{HPPA:defn:parametericskolemfunctions}
Suppose that $M$ is an $L$-structure. Then $M$ has {\it definable skolem functions} if for every definable set $P\subseteq M^{m+n}$ there is a definable set $P^{\prime}\subseteq M^{m+n}$ such that\index{Definable skolem functions, Definition~\ref{HPPA:defn:parametericskolemfunctions}}
\begin{align}
& M\models \forall \overline{x},\overline{y} \; [P^{\prime}\overline{x}\overline{y} \rightarrow P\overline{x}\overline{y}] \label{hppa:eqn:compareasdfasdf1}\\
& M\models \forall \; \overline{x} \; [\exists \; \overline{y} \; P\overline{x}\overline{y}]\rightarrow [\exists !\; \overline{y}\; P^{\prime}\overline{x}\overline{y}] \label{hppa:eqn:compareasdfasdf2}
\end{align}
\end{defn}

\begin{rmk}
Note that in this definition, the parameters used to define $P^{\prime}$ may exceed those used to define $P$. Note also the obvious similarity between  definable skolem functions and the uniformization results, such as Kondo's Uniformization Theorem~\ref{hppa:thm:kondo}, which we employed in Theorem~\ref{HPPA:thm:main}. In particular, equations~(\ref{hppa:eqn:compareasdfasdf1})-(\ref{hppa:eqn:compareasdfasdf2}) are nearly identical to equations~(\ref{hppa:eqn:compareasdfasdf3})-(\ref{hppa:eqn:compareasdfasdf4}).
\end{rmk}

\begin{defn}
Suppose that $M$ is an $L$-structure and $A\subseteq M$. A set of $A$-formulas $p(\overline{v})$ in finitely many variables $\overline{v}$ is {\it realized} in $M$ if there is an $\overline{b}$ in $M$ such that $M\models \theta(\overline{b})$ for every $A$-formula $\theta(\overline{v})$ in $p(\overline{v})$. A set of $A$-formulas $p(\overline{v})$ is {\it finitely realized} in $M$ if every finite subset $p_{0}(\overline{v})$ of $p(\overline{v})$ is realized in $M$. The structure $M$ is {\it saturated} if for every $A\subseteq M$ with $\left|A\right|<\left|M\right|$ and every set of $A$-formulas $p(\overline{v})$, if $p(\overline{v})$ is finitely realized in $M$ then $p(\overline{v})$ is realized in $M$.
\end{defn}

\begin{defn}\label{hppa:index:compsat}
Suppose that $L$ and $M$ are countable and $B\in 2^{\omega}$. Then $M$ is {\it $B$-recursively saturated} if for every finite $A\subseteq M$ and every $B$-computable set of $A$-formulas $p(\overline{v})$, if $p(\overline{v})$ is finitely realized in $M$ then $p(\overline{v})$ is realized in $M$.\index{Recursively saturated structure, Definition~\ref{hppa:index:compsat}}
\end{defn}

\begin{rmk}
The following proposition records the very elementary observation that saturated structures (resp. $B$-recursively saturated structures) have a kind of compactness property, in that each covering of $M^{n}$ by definable sets has a finite sub-covering (resp. each $B$-recursive covering of $M^{n}$ by definable sets has a finite sub-covering).
\end{rmk}

\begin{prop}\label{HPPA:prop:saturatedequalscompact}
Suppose that $M$ is a saturated $L$-structure (resp. $B$-recursively saturated $L$-structure) and that $A\subseteq M$ with $\left|A\right|<\left|M\right|$. Further, suppose that $\{\theta_{i}(\overline{v})\}_{i\in I}$ is a set of $A$-formulas (resp. $B$-computable set of $A$-formulas). Then 
\begin{equation}\label{HPPA:eqn:satequalscompact}
[M\models \forall \; \overline{a} \; \bigvee_{i\in I} \theta_{i}(\overline{v})] \Longrightarrow [\exists \; \mbox{ finite } I_{0}\subseteq I \; M\models \forall \; \overline{a} \; \bigvee_{i\in I_{0}} \theta_{i}(\overline{v})]
\end{equation}
\end{prop}
\begin{proof}
The contrapositive of equation~(\ref{HPPA:eqn:satequalscompact}) says that if the set of $A$-formulas $p(\overline{v})=\{\neg \theta_{i}(\overline{a}): i\in I\}$ is finitely realized, then it is realized.
\end{proof}

\begin{thm}\label{HPPA:thm:metatheorem}
Suppose that $M$ is an $L$-structure and $\partial: D(M)\rightarrow M$ such that the structure $N=(M, D(M), D(M^{2}), \ldots, \partial)$ models ${\tt ABL}_{0}$ (resp. ${\tt AHP}_{0}$). Suppose that $B\in 2^{\omega}$. Then
\begin{enumerate}
\item[(i)] If $\partial: D(M)\rightarrow M$ is uniformly definable and $M$ is saturated, then the structure $N$ models ${\tt \Delta^{1}_{1}-BL}_{0}$ (resp. ${\tt \Delta^{1}_{1}-HP}_{0}$).
\item[(ii)] If $\partial: D(M)\rightarrow M$ is uniformly definable and $M$ is saturated, then the structure $N$ models ${\tt \Sigma^{1}_{1}-LB}_{0}$ (resp. ${\tt \Sigma^{1}_{1}-PH}_{0}$) if and only if $M$ has definable skolem functions.
\item[(iii)] If $\partial: D(M)\rightarrow M$ is $B$-computably uniformly definable and $M$ is $B$-recursively saturated, then the structure $N$ models ${\tt \Delta^{1}_{1}-BL}_{0}$ (resp. ${\tt \Delta^{1}_{1}-HP}_{0}$).
\item[(iv)] If $\partial: D(M)\rightarrow M$ is $B$-computably uniformly definable and $M$ is $B$-recursively saturated, then the structure $N$ models ${\tt \Sigma^{1}_{1}-LB}_{0}$ (resp. ${\tt \Sigma^{1}_{1}-PH}_{0}$) if and only if $M$ has definable skolem functions.
\end{enumerate}
\end{thm}
\begin{proof}
In all four parts of this proof, the proof is identical between Basic~Law~V and Hume's~Principle, and so we only include the proofs for the case of Basic~Law~V. Further, the proof of~(i) and~(iii) are parallel and the proof of~(ii) and~(iv) are parallel, and so we present the proofs of~(i) and~(iii) simultaneously and the proofs of~(ii) and~(iv) simultaneously. For~(i) and~(iii), suppose that  $\partial: D(M)\rightarrow M$ is uniformly definable (resp. $B$-computably uniformly definable) and $M$ is saturated (resp. $B$-recursively saturated). To see that $N$ is a model of ${\tt \Delta^{1}_{1}-BL}_{0}$, suppose that there is a subset $Z$ of $M^{n}$ which is defined on $N$ by a $\Sigma^{1}_{1}$-formula $\varphi(\overline{z})$ and by a $\Pi^{1}_{1}$-formula $\psi(\overline{z})$. Let us suppose that $\varphi(\overline{z})$ and $\psi(\overline{z})$ use exactly one set parameter $A\in D(M)$ where 
\begin{equation}\label{HPPA:defn:Arho}
A=\{w\in M: M\models \rho(w, \overline{a})\}
\end{equation}
and where $\rho(w, \overline{v})$ is an $\emptyset$-formula with $\overline{a}\in M$, since the proof in the case where there are multiple parameters, with some being objects, some sets, and some binary relations etc., is exactly identical. Further, let us suppose that $\varphi(\overline{z})\equiv \exists \; X \; \varphi_{0}(\overline{z}, X, \partial(X),A)$ and that $\psi(\overline{z})\equiv \forall \; X \; \psi_{0}(\overline{z},X, \partial(X), A)$, since the proof in the case where there are multiple existential (resp. universal) set-quantifiers or relation-quantifiers in $\varphi(\overline{z})$ (resp. $\psi(\overline{z}))$ is exactly identical.  Then
\begin{equation}\label{HPPA:eqn:bigsuppp}
\overline{z}\in Z \Longleftrightarrow N\models \exists \; X \; \varphi_{0}(\overline{z}, X, \partial(X), A) \Longleftrightarrow N\models \forall \; X \; \psi_{0}(\overline{z},X, \partial(X), A)
\end{equation}
Then
\begin{equation}\label{HPPA:eqn:5555}
N\models \forall \; \overline{z} \; \exists \; X \; \varphi_{0}(\overline{z}, X, \partial(X),A) \vee \neg \psi_{0}(\overline{z},X, \partial(X),A)
\end{equation}
Let us abbreviate 
\begin{equation}\label{HPPA:eqn:dfdfadfabbb}
\xi_{0}(\overline{z}, X, \partial(X),A) \equiv \varphi_{0}(\overline{z}, X, \partial(X),A) \vee \neg \psi_{0}(\overline{z},X, \partial(X),A)
\end{equation}
so that equation~(\ref{HPPA:eqn:5555}) becomes
\begin{equation}
N\models \forall \; \overline{z} \; \exists \; X \; \xi_{0}(\overline{z}, X, \partial(X),A)
\end{equation}
Then this translates into $M$ as 
\begin{equation}
M\models \forall \; \overline{z} \; \bigvee_{\theta(x, \overline{y}) } \exists \; \overline{b} \; \xi_{0}(\overline{z}, \theta(\cdot, \overline{b}), \partial(\theta(\cdot, \overline{b})), \rho(\cdot, \overline{a}))
\end{equation}
where $\theta(x, \overline{y})$ ranges over $\emptyset$-formulas with non-empty set of parameter variables $\overline{y}$. Since the map $\partial:D(M)\rightarrow M$ is uniformly definable (resp. $B$-computably uniformly definable) via the map $\theta\mapsto \theta^{\prime}$, we have
\begin{equation}\label{HPPA:eqn:helpadsfadsf333}
M\models \forall \; \overline{z} \; \bigvee_{\theta(x, \overline{y})} \exists \; \overline{b} \; \exists \; c \; (\theta^{\prime}(c, \overline{b}) \; \& \; \xi_{0}(\overline{z}, \theta(\cdot, \overline{b}), c, \rho(\cdot, \overline{a}))
\end{equation}
Since $M$ is saturated (resp. $B$-recursively saturated), an application of Proposition~\ref{HPPA:prop:saturatedequalscompact} implies that there is $K>0$ and there are $\emptyset$-formulas $\theta_{1}(x, \overline{y}), \ldots, \theta_{K}(x, \overline{y})$ such that 
 \begin{equation}
M\models \forall \; \overline{z} \; \bigvee_{i=1}^{K} \exists \; \overline{b} \; \exists \; c \; (\theta^{\prime}_{i}(c, \overline{b}) \; \& \; \xi_{0}(\overline{z}, \theta_{i}(\cdot, \overline{b}), c, \rho(\cdot, \overline{a})))
\end{equation}
Then by definition of $\xi_{0}$ (cf. equation~\ref{HPPA:eqn:dfdfadfabbb})), we have:
 \begin{equation}
M\models \forall \; \overline{z} \; \bigvee_{i=1}^{K} \exists \; \overline{b} \; \exists \; c \; (\theta^{\prime}_{i}(c, \overline{b}) \; \& \; (\varphi_{0}(\overline{z}, \theta_{i}(\cdot, \overline{b}), c, \rho(\cdot, \overline{a})) \vee \neg \psi_{0}(\overline{z},\theta_{i}(\cdot, \overline{b}), c, \rho(\cdot, \overline{a}))))
\end{equation}
It follows from equation~(\ref{HPPA:eqn:bigsuppp}) that
\begin{equation}
Z = \{\overline{z}\in M^{n} : M\models \bigvee_{i=1}^{K} \exists \; \overline{b} \; \exists \; c \; (\theta^{\prime}_{i}(c, \overline{b}) \; \& \; (\varphi_{0}(\overline{z}, \theta_{i}(\cdot, \overline{b}), c, \rho(\cdot, \overline{a})))\}
\end{equation}
Hence $Z\in D(M^{n})$ and so $N$ satisfies ${\tt \Delta^{1}_{1}-BL}_{0}$. Hence, this completes the proof of parts~(i) and~(iii).

We turn to the proofs of parts~(ii) and~(iv). First, we handle the proof of the right-to-left direction, which is quite similar to the proof from the above paragraph.  Suppose that $\partial: D(M)\rightarrow M$ is uniformly definable (resp. $B$-computably uniformly definable) and $M$ is saturated (resp. $B$-recursively saturated) and has definable skolem functions. To see that $N$ is a model of ${\tt \Sigma^{1}_{1}-LB}_{0}$, suppose that
\begin{equation}\label{HPPA:eqn:dafdfasdfasdfdsafdsfasdfdjj}
N\models \forall \; \overline{z} \; \exists \; X \; \xi_{0}(\overline{z}, X, \partial(X), A)
\end{equation}
where $\xi_{0}$ is arithmetical and where $A\in D(M)$ is a set parameter with 
\begin{equation}\label{HPPA:defn:Arho}
A=\{w\in M: M\models \rho(w, \overline{a})\}
\end{equation}
and where $\rho(w, \overline{v})$ is an $\emptyset$-formula with $\overline{a}\in M$. (As in the proof in the previous paragraph, the case of multiple parameters or multiple set or relation quantifiers is exactly similar). Then equation~(\ref{HPPA:eqn:dafdfasdfasdfdsafdsfasdfdjj}) translates into $M$ as 
\begin{equation}
M\models \forall \; \overline{z} \; \bigvee_{\theta(x, \overline{y}) } \exists \; \overline{b} \; \xi_{0}(\overline{z}, \theta(\cdot, \overline{b}), \partial(\theta(\cdot, \overline{b})), \rho(\cdot, \overline{a}))
\end{equation}
where $\theta(x, \overline{y})$ ranges over $\emptyset$-formulas with non-empty set of parameter variables $\overline{y}$. Since $\partial:D(M)\rightarrow M$ is uniformly definable (resp. $B$-computably uniformly definable) via the map $\theta\mapsto \theta^{\prime}$, we have
\begin{equation}\label{HPPA:eqn:helpadsfadsf}
M\models \forall \; \overline{z} \; \bigvee_{\theta(x, \overline{y})} \exists \; \overline{b} \; \exists \; c \; (\theta^{\prime}(c, \overline{b}) \; \& \; \xi_{0}(\overline{z}, \theta(\cdot, \overline{b}), c, \rho(\cdot, \overline{a}))
\end{equation}
Since $M$ is saturated (resp. $B$-recursively saturated), an application of Proposition~\ref{HPPA:prop:saturatedequalscompact} implies that there is $K>0$ and there are $\emptyset$-formulas $\theta_{1}(x, \overline{y}), \ldots, \theta_{K}(x, \overline{y})$ such that 
 \begin{equation}
M\models \forall \; \overline{z} \; \bigvee_{i=1}^{K} \exists \; \overline{b} \; \exists \; c \; (\theta^{\prime}_{i}(c, \overline{b}) \; \& \; \xi_{0}(\overline{z}, \theta_{i}(\cdot, \overline{b}), c, \rho(\cdot, \overline{a})))
\end{equation}
Then by adding dummy variables if need be, we can move the disjunction to the right as follows:
 \begin{equation}
M\models \forall \; \overline{z} \; \exists \; \overline{b} \; \exists \; c \; \bigvee_{i=1}^{K} \;(\theta^{\prime}_{i}(c, \overline{b}) \; \& \; \xi_{0}(\overline{z}, \theta_{i}(\cdot, \overline{b}), c, \rho(\cdot, \overline{a})))
\end{equation}
and one can take the first such $i$ as follows:
 \begin{equation}
M\models \forall \; \overline{z} \; \exists \; \overline{b} \; \exists \; c \; \bigvee_{i=1}^{K} \; [(\theta^{\prime}_{i}(c, \overline{b}) \; \& \; \xi_{0}(\overline{z}, \theta_{i}(\cdot, \overline{b}), c, \rho(\cdot, \overline{a}))) \; \& \; \bigwedge_{j<i} \neg (\theta^{\prime}_{j}(c, \overline{b}) \; \& \; \xi_{0}(\overline{z}, \theta_{j}(\cdot, \overline{b}), c, \rho(\cdot, \overline{a})))]
\end{equation}
Then since $M$ has definable skolem functions, there is a possibly larger finite set of parameters $\overline{a}^{\prime}\supseteq \overline{a}$ and $\overline{a}^{\prime}$-definable functions $f,g$ such that
 \begin{align}
M\models \forall \; \overline{z} \; \bigvee_{i=1}^{K}\; & \; [(\theta^{\prime}_{i}(g(\overline{z}), f(\overline{z})) \; \& \; \xi_{0}(\overline{z}, \theta_{i}(\cdot, f(\overline{z})), g(\overline{z}), \rho(\cdot, \overline{a})))\notag  \\ & \; \& \; \bigwedge_{j<i} \neg (\theta^{\prime}_{j}(g(\overline{z}), f(\overline{z})) \; \& \; \xi_{0}(\overline{z}, \theta_{j}(\cdot, f(\overline{z})), g(\overline{z}), \rho(\cdot, \overline{a})))]
\end{align}
Then there is a partition of $M^{n}$ into the $\overline{a}^{\prime}$-definable sets $P_{1}, \ldots, P_{K}$ which are defined as follows:
\begin{align}\label{HPPA:eqn:defPi}
P_{i} =\{\overline{z}\in M^{n}: M\models & [(\theta^{\prime}_{i}(g(\overline{z}), f(\overline{z})) \; \& \; \xi_{0}(\overline{z}, \theta_{i}(\cdot, f(\overline{z})), g(\overline{z}), \rho(\cdot, \overline{a}))) \notag \\ & \; \& \; \bigwedge_{j<i} \neg (\theta^{\prime}_{j}(g(\overline{z}), f(\overline{z})) \; \& \; \xi_{0}(\overline{z}, \theta_{j}(\cdot, f(\overline{z})), g(\overline{z}), \rho(\cdot, \overline{a})))]\}
\end{align}
Then define the $\overline{a}^{\prime}$-definable relation
\begin{equation}
R = \{(\overline{z}, w): \bigvee_{i=1}^{K} [\overline{z}\in P_{i}\rightarrow \theta_{i}(w, f(\overline{z}))]\}
\end{equation}
so that
\begin{align}
& \overline{z} \in P_{i} \Longrightarrow R_{\overline{z}} =\{w\in M: (\overline{z}, w)\in R\}= \{w\in M: M\models \theta_{i}(w, f(\overline{z}))\}= \theta_{i}(\cdot, f(\overline{z})) \\
& \overline{z} \in P_{i} \Longrightarrow \{\partial(R_{\overline{z}})\}=\{\partial( \theta_{i}(\cdot, f(\overline{z}))\} =\{c\in M: M\models \theta^{\prime}_{i}(c, f(\overline{z}))\}=\{g(\overline{z})\} \\
& \overline{z} \in P_{i} \Longrightarrow \partial(R_{\overline{z}}) = g(\overline{z})
\end{align}
Putting these things together and glancing back at the definition of $P_{i}$ in equation~(\ref{HPPA:eqn:defPi}) we have,
\begin{equation}
\overline{z} \in P_{i} \Longrightarrow M\models  (\theta^{\prime}_{i}(g(\overline{z}), f(\overline{z})) \; \& \; \xi_{0}(\overline{z}, \theta_{i}(\cdot, f(\overline{z})), g(\overline{z}), \rho(\cdot, \overline{a}))) \Longrightarrow N\models \xi_{0}(\overline{z}, R_{\overline{z}}, \partial(R_{\overline{z}}), A)
\end{equation}
Since the sets $P_{1}, \ldots, P_{K}$ partition $M^{n}$ we have
\begin{equation}
N\models \forall \; \overline{z} \; \xi_{0}(\overline{z}, R_{\overline{z}}, \partial(R_{\overline{z}}), A)
\end{equation}
and this implies that $N$ models ${\tt \Sigma^{1}_{1}-BL}_{0}$. Hence we have established the right-to-left direction of~(ii) and~(iv).

We want to establish the left-to-right direction of~(ii) and~(iv). Suppose that $\partial: D(M)\rightarrow M$ is uniformly definable (resp. $B$-computably uniformly definable) and $M$ is saturated (resp. $B$-recursively saturated) and that $N$ models ${\tt \Sigma^{1}_{1}-BL}_{0}$. Suppose that $P\subseteq M^{m+n}$ is definable, perhaps with a finite set $\overline{a}$ of parameters from $M$. Note that for every $\overline{x}\in M^{m}$ with a tuple $\overline{y}\in M^{n}$ such that $P\overline{x}\overline{y}$, we can arbitrarily choose one such $\overline{y}\in M^{n}$ and form the $\overline{y}$-definable singleton $\{\overline{y}\}$. This implies that
\begin{equation}
N \models \; \forall \; \overline{x} \; \exists \; R \; [(\exists \; \overline{y} \; P\overline{x}\overline{y}) \rightarrow ((\exists ! \; \overline{y} \; R \overline{y}) \; \& \; (\forall \; \overline{y} \; R\overline{y}\rightarrow P\overline{x}\overline{y}))]
\end{equation}
Since $N\models {\tt \Sigma^{1}_{1}-LB}_{0}$, one then has
\begin{equation}
N \models \exists \; P^{\prime} \; \forall \; \overline{x} \; [(\exists \; \overline{y} \; P\overline{x}\overline{y}) \rightarrow ((\exists ! \; \overline{y} \; P^{\prime}_{\overline{x}} \overline{y}) \; \& \; (\forall \; \overline{y} \; P^{\prime}_{\overline{x}}\overline{y}\rightarrow P\overline{x}\overline{y}))]
\end{equation}
Since $P^{\prime}_{\overline{x}}\overline{y}$ if and only if $P^{\prime}\overline{x} \overline{y}$, this implies that
\begin{equation}
N \models \exists \; P^{\prime} \; \forall \; \overline{x} \; [(\exists \; \overline{y} \; P\overline{x}\overline{y}) \rightarrow ((\exists ! \; \overline{y} \; P^{\prime}\overline{x} \overline{y}) \; \& \; (\forall \; \overline{y} \; P^{\prime} \overline{x}\overline{y}\rightarrow P\overline{x}\overline{y}))]
\end{equation}
Finally, let $P^{\prime\prime} = P^{\prime}\cap P$. Then
\begin{align}
& M\models \forall \overline{x},\overline{y} \; (P^{\prime\prime}\overline{x}\overline{y} \rightarrow P\overline{x}\overline{y}) \\
& M\models \forall \; \overline{x} \; [(\exists \; \overline{y} \; P\overline{x}\overline{y})\rightarrow (\exists !\; \overline{y}\; P^{\prime\prime}\overline{x}\overline{y})] 
\end{align}
Hence, $M$ has definable skolem functions. 
\end{proof}

\subsection{Application to Algebraically Closed Fields}\label{hppa:lastsec2}

\begin{rmk}
In this section, we apply Theorem~\ref{HPPA:thm:metatheorem} to construct models of ${\tt \Delta^{1}_{1}-HP}_{0}$ on top of certain~algebraically closed fields (cf. Theorem~\ref{hppa:coolacfthm}). The primary application of this construction is to answer a question posed by Linnebo (cf. Remark~\ref{hppa:rmk:linnebo} and Theorem~\ref{HPPA:lineeeddafdsfa}). Prior to doing this, we recall Ax's Theorem and note one elementary consequence of this theorem.
\end{rmk}

\begin{thm}\label{HPPA:prop:axtheorem}
(Ax's Theorem) \index{Ax's Theorem, Theorem~\ref{HPPA:prop:axtheorem}}Suppose that $k$ is an algebraically closed field and $f:k\rightarrow k$ is a definable injective function. Then $f$ is surjective.
\end{thm}
\begin{proof}
See Ax~\cite{Ax1968} Theorem~C pp.~241, 270 or Poizat~\cite{Poizat2001aa} Lemma~4.3 pp.~70-71, in which is proved the stronger result wherein $k$ is replaced by a definable subset of $k^{n}$.
\end{proof}

\begin{defn}\label{hppa:index:stomniddd}
A structure $k$ is {\it strongly minimal} if every definable $X\subseteq k$ is finite or cofinite.\index{Strongly minimal structure, Definition~\ref{hppa:index:stomniddd}}
\end{defn}

\begin{prop}
Every algebraically closed field is strongly minimal.
\end{prop}
\begin{proof}
See Marker~\cite{Marker2006ac} p.~5.
\end{proof}

\begin{prop}\label{HPPA:prop:defnhpadsfasd}
 Suppose that $k$ is an algebraically closed field and that $X,Y\subseteq k$ are definable. Then the following are equivalent:
\begin{enumerate}
\item[(i)] There is definable bijection $f:X\rightarrow Y$
\item[(ii)] Either both $X$ and $Y$ are finite and of the same cardinality, or both $X$ and $Y$ are cofinite and $k\setminus X$ and $k\setminus Y$ are of the same cardinality.
\end{enumerate}
\end{prop}
\begin{proof}
Suppose that~(i) holds. Then by strong minimality and the fact that an infinite set cannot be bijective with a finite set, either both $X$ and $Y$ are finite or both $X$ and $Y$ are cofinite. If $X$ and $Y$ are both finite then the fact that there is a definable bijection between them implies that $X$ and $Y$ have the same cardinality. If $X$ and $Y$ are both cofinite but $k\setminus X$ and $k\setminus Y$ are not of the same cardinality, then without loss of generality $k\setminus X =\{a_{1}, \ldots, a_{m}\}$ and $k\setminus Y =\{b_{1}, \ldots, b_{n}\}$ where $m<n$. Then define a function $\overline{f}:k\rightarrow k$ by $\overline{f}\upharpoonright X = f$ and $f(a_{i})=b_{i}$ for $i\leq m$. Then $f: k\rightarrow k$ is an injection that is not a surjection, since $b_{n}$ is not the in the range of $f$. This contradicts Ax's Theorem~\ref{HPPA:prop:axtheorem}. So, in fact, $k\setminus X$ and $k\setminus Y$ are of the same cardinality. Then~(ii) holds.

Conversely, suppose that~(ii) holds. If both $X$ and $Y$ are finite of the same cardinality, then simply enumerate the elements of $X$ and $Y$ and use these elements as parameters to define a bijection $f:X\rightarrow Y$. If $X$ and $Y$ are both cofinite and $k\setminus X$ and $k\setminus Y$ are of the same finite cardinality, then enumerate $k\setminus X=\{y_{1}, \ldots, y_{n}\}$ and $k\setminus Y=\{x_{1}, \ldots, x_{n}\}$. By renumbering, we can assume without loss of generality that $(k\setminus X)\cap (k\setminus Y)=\{x_{1}, \ldots, x_{m}\}=\{y_{1}, \ldots, y_{m}\}$ where $m\leq n$ and $x_{1}=y_{1}, \ldots, x_{m}=y_{m}$. If $m=n$ then this implies that $(k\setminus X)= (k\setminus Y)$ and $X=Y$, and we can choose the definable bijection $f:X\rightarrow Y$ to be the identity map. If $m<n$, then note that $\{x_{m+1}, \ldots, x_{n}\}\subseteq X$ and $\{y_{m+1}, \ldots, y_{n}\}\subseteq Y$ and $X\setminus \{x_{m+1}, \ldots, x_{n}\}\subseteq Y$ and $Y\setminus \{y_{m+1}, \ldots, y_{n}\}\subseteq X$. Then we can choose the definable bijection $f:X\rightarrow Y$ which is given by the identity on $X\setminus \{x_{m+1}, \ldots, x_{n}\}$ and by $f(x_{i})=y_{i}$ on $\{x_{m+1}, \ldots, x_{n}\}$.
\end{proof}

\begin{prop}\label{HPPA:prop:acfdsf}
Algebraically closed fields do not have definable skolem functions.
\end{prop}
\begin{proof}
Let $\varphi(x,y)\equiv x=y^{2}$. Then $k\models \forall \; x\; \exists \; y \; x=y^{2}$. If $k$ has definable skolem functions, then there is a definable function $f:k\rightarrow k$ such that $k\models \forall \;x \; x=(f(x))^{2}$. Then $\mathrm{rng}(f)$ is a definable set which includes exactly one square root for each $x\in k$. Then $\mathrm{rng}(f)$ is infinite and coinfinite, which contradicts strong minimality.
\end{proof}

\begin{thm}\label{hppa:coolacfthm}
Suppose that $k$ is a saturated algebraically closed field of characteristic zero. Then there is a uniformly definable function $\#:D(k)\rightarrow k$ such that $(k, D(k), D(k^{2}), \ldots, \#)$ is a model of ${\tt \Delta^{1}_{1}-HP}_{0}+ \neg {\tt \Sigma^{1}_{1}-PH}_{0}+\neg {\tt \Pi^{1}_{1}-HP}_{0}$. Further, there is {\it no} function $\partial:D(k)\rightarrow k$ such that $(k, D(k), D(k^{2}), \ldots, \partial)$ is a model of ${\tt \Delta^{1}_{1}-BL}_{0}$.
\end{thm}
\begin{proof}
Since $k$ is a field of characteristic zero, the prime field of $k$ is $\mathbb{Q}$ and the integers $\mathbb{Z}$ are hence embedded into $k$ via $\mathbb{Q}$. Using this embedding, define $\#:D(k)\rightarrow k$ by $\#X=\left|X\right|$ if $X$ is finite and $\#X=-(\left|k\setminus X\right|+1)$ if $X$ is cofinite. Then by Proposition~\ref{HPPA:prop:defnhpadsfasd}, the structure $(k, D(k), D(k^{2}), \ldots, \#)$ is a model of Hume's~Principle. To apply Theorem~\ref{HPPA:thm:metatheorem}~(i)-(ii), we need to show that $\#: D(k)\rightarrow k$ is uniformly definable. Suppose that $\theta(x, \overline{y})$ is an $\emptyset$-formula with non-empty set $\overline{y}$ of parameter variables. Then by strong minimality, for any $\overline{a}$ we have that $\theta(\cdot, \overline{a})$ is finite or $\neg\theta(\cdot, \overline{a})$ is finite. Then 
\begin{equation}
k\models \forall \; \overline{a} \; \bigvee_{N\geq 0} \; [\left| \theta(\cdot, \overline{a})\right|\leq N \vee \left| \neg\theta(\cdot, \overline{a})\right|\leq N]
\end{equation}
Since $k$ is saturated, by Proposition~\ref{HPPA:prop:saturatedequalscompact}, there is an integer $N_{\theta}>0$ such that
\begin{equation}
k\models \forall \; \overline{a} \; \bigvee_{i=0}^{N_{\theta}} \; [\left| \theta(\cdot, \overline{a})\right|\leq i \vee \left| \neg\theta(\cdot, \overline{a})\right|\leq i]
\end{equation}
Then for each such formula $\theta(x,\overline{y})$ we define the following $\emptyset$-formula $\theta^{\prime}(x,\overline{y})$ as follows:
\begin{equation}
\theta^{\prime}(x,\overline{y})\equiv \bigvee_{i=0}^{N_{\theta}} [\left| \theta(\cdot, \overline{y})\right| = i \; \& \; x=i] \vee [\left| \neg\theta(\cdot, \overline{y})\right| = i \; \& \; x=-(i+1)]
\end{equation}
Hence, by definition, we have that for any $\overline{a}$
\begin{equation}
\{\#(\theta(\cdot, \overline{a}))\} = \{c: k\models \theta^{\prime}(c, \overline{a})\}
\end{equation}
The map $\#: D(k)\rightarrow k$ is uniformly definable. Hence, by Theorem~\ref{HPPA:thm:metatheorem}~(i)-(ii) and Proposition~\ref{HPPA:prop:acfdsf}, we have that $(k, D(k),D(k^{2}), \ldots, \#)$ is a model of ${\tt \Delta^{1}_{1}-HP}_{0}+\neg {\tt \Sigma^{1}_{1}-PH}_{0}$. Further, since the set $\mathrm{rng}(\#)=\mathbb{Z}$ is definable by a $\Sigma^{1}_{1}$-formula in the structure $(k, D(k), D(k^{2}), \ldots, \#)$ but is {\it not} definable in $k$ since $k$ is strongly minimal, we have that $(k, D(k),D(k^{2}), \ldots, \#)$ is a model of $\neg {\tt \Pi^{1}_{1}-HP}_{0}$.

Now let us note why there is {\it no} function $\partial:D(k)\rightarrow k$ such that $(k, D(k), D(k^{2}), \ldots, \partial)$ is a model of ${\tt \Delta^{1}_{1}-BL}_{0}$. If there was such a function, then by Corollary~\ref{HPPA:prop:whatitporves} it would follow that there was an injective non-surjective function $s:k\rightarrow k$ whose graph is in $D(k^{2})$, which would contradict Ax's Theorem (\ref{HPPA:prop:axtheorem}).
\end{proof}

\begin{rmk}\label{HPPA:axrmk}
If we knew that all the parts of the proof of the above theorem were formalizable in ${\tt ACA}_{0}$, then we could infer from the proof of the above theorem and Proposition~\ref{HPPA:conpropa} that ${\tt \Delta^{1}_{1}-HP}_{0}<_{\mathrm{I}} {\tt ACA}_{0}$. It is clear from the proof that this comes down to determining whether or not Ax's Theorem~\ref{HPPA:prop:axtheorem} is provable in ${\tt ACA}_{0}$. However, note that in the next subsection, we will prove Corollary~\ref{HPPA:cor:thebomb}, which assures us that ${\tt \Delta^{1}_{1}-HP}_{0}<_{\mathrm{I}} {\tt ACA}_{0}$.
\end{rmk}

\begin{rmk}
In conjunction with Corollary~\ref{HPPA:prop:whatitporves}, the following corollary shows that there is a stark contrast between ${\tt \Delta^{1}_{1}-HP}_{0}$ and ${\tt \Delta^{1}_{1}-BL}_{0}$ on the score of whether they require the existence of injective non-surjective functions.
\end{rmk}

\begin{cor}\label{hppa:corcorcor3434123431}
There is a model $(M, S_{1}, S_{2}, \ldots, \#)$ of ${\tt \Delta^{1}_{1}-HP}_{0}$ such that there is no injective non-surjective function $s:M\rightarrow M$ such that $\mathrm{graph}(s)$ is in $S_{2}$.
\end{cor}
\begin{proof}
This follows immediately from the construction in Theorem~\ref{hppa:coolacfthm} and Ax's Theorem~\ref{HPPA:prop:axtheorem}.
\end{proof}

\begin{rmk}\label{hppa:rmk:linnebo}
Linnebo presented a description of properties that models of ${\tt AHP}_{0}$ and ${\tt \Delta^{1}_{1}-HP}_{0}$ must have {\it if} they fail to model a certain sort of successor axiom (\cite{Linnebo2004} pp.~164-165), and he additionally showed that there was a model of ${\tt AHP}_{0}$ which did not model this successor axiom (\cite{Linnebo2004} Theorem~2 p.~164). Linnebo then remarked that it was unknown whether there was a model of ${\tt \Delta^{1}_{1}-HP}_{0}$ that did not model the successor axiom (cf. \cite{Linnebo2004} Remark~6 p.~168). Subsequent to defining this successor axiom, we now show that the model from the previous theorem does not model this axiom. We also explain why certain properties  identified by Linnebo hold in this model.
\end{rmk}

\begin{defn}\label{hppa:index:defnsasssaaxiomslin}
The following are formulas in the language of ${\tt HP}^{2}$ (cf. Linnebo~\cite{Linnebo2004} pp.~158-160):
\begin{enumerate}
\item[(i)] $P(n,m) \Longleftrightarrow \exists \; X, Y \; \#X=n \; \& \; \#Y =m \; \& \; \exists \; y\in Y \; X=Y\setminus \{y\}$
\item[(ii)] $F$ is {\it hereditary} if $Fn$ and $P(n,m)$ implies $Fm$
\item[(iii)] $F$ is {\it closed} if $P(\#\emptyset, m)$ implies $Fm$
\item[(iv)] $n$ is a {\it pseudo-number} if $n=\#\emptyset$ or $n$ is contained in all hereditary, closed $F$.
\item[(v)] The successor axiom (${\tt SA}$) says that for any pseudo-number $n$, there is $m$ such that $P(n,m)$.\index{${\tt SA}$, the successor axiom, sentence in signature of ${\tt HP}^{2}$, Definition~\ref{hppa:index:defnsasssaaxiomslin}~(v)}
\end{enumerate}
\end{defn}

\begin{prop}\label{HPPA:lineeeddafdsfa}
Suppose that $k$ is a saturated algebraically closed field of characteristic zero. Suppose that $\#:D(k)\rightarrow k$ by $\#X=\left|X\right|$ if $X$ is finite and $\#X=-(\left|k\setminus X\right|+1)$ if $X$ is cofinite. Then $(k, D(k), D(k^{2}), \ldots, \#)\models {\tt \Delta^{1}_{1}-HP}_{0}+\neg {\tt SA}$.
\end{prop}
\begin{proof}
Before we begin, it is perhaps helpful to informally state the definition of $\#$ given above and describe how it interacts with the predicate $P(n,m)$. If $X$ is a finite set with $n$ elements, then $\#X=n$, and if $X$ is a cofinite set with $n$ elements in its complement, then $\#X=-(n+1)$. So, for example, the set $X=\{\sqrt{2}, -1\}$ has $\#X=2$, and the set $X=\{a\in k: k\models a^{2}+1\neq 0\}$ has $\#X=-(2+1)=-3$, and the set $X=k$ has $\#X=-1$, and the set $X=\emptyset$ has $\#X=0$. Further, if $X$ is finite, then by choosing an element $y\notin X$, we have $P(\#X, \#(X\cup \{y\}))$. For example, if $X$ is finite and has $n$ elements and $y\notin X$, we have that $\#X=n$ and $\#(X\cup \{y\})=n+1$, so that $P(n, n+1)$. Conversely, if $X$ is cofinite and has $n>0$ elements in its complement and $y\notin X$, then we have that $X\cup \{y\}$ has $n-1$ elements in its complement, so that $\#X=-(n+1)=-n-1$ and $\#(X\cup \{y\})=-((n-1)+1)=-n$ and hence $P(-n-1, -n)$. For example, we have $P(0,1), P(1,2), P(2,3),\ldots$ and $\ldots, P(-4,-3), P(-3,-2), P(-2, -1)$.

Now we begin the proof. In particular, we want to begin by describing what the hereditary, closed sets $F\in D(k)$ look like. So suppose that $F\in D(k)$ is hereditary and closed. First we claim that $\mathbb{N}\setminus \{0\}\subseteq F$. For, by the definition of $P(n,m)$ and $\#$, we have that $F$'s being closed implies that $P(0,1)$ and hence $1\in F$. So suppose that $n\in (\mathbb{N}\setminus \{0\})\cap F$. Then by the definition of $P(n,m)$ and $\#$, we have that $F$'s being hereditary implies that $P(n,n+1)$ and hence $n+1\in F$. By induction, we have that if $F\in D(k)$ is hereditary and closed then $\mathbb{N}\setminus \{0\}\subseteq F$.

We want to claim that $\{n\in \mathbb{Z}: n\neq 0\}\subseteq F$. Suppose not. That is, suppose that there are some negative integers that are not in $F$. Then, since  $F\in D(k)$ is infinite, strong minimality implies that $F$ is co-finite. So there are at most finitely many negative integers that are not in $F$. Suppose that we write these negative integers in increasing order as $a_{1}<a_{2}<\cdots<a_{n}$. (E.g. if $\mathbb{Z}\setminus F=\{-5,-10,-12\}$ then $a_{1}=-12, a_{2}=-10$ and $a_{3}=-5$). This implies that $a_{1}-1\in F$. But then by the definition of~$P(n,m)$ and~$\#$, we have that $F$'s being hereditary implies that $P(a_{1}-1,a_{1})$ and hence $Fa_{1}$, which is a contradiction. Hence, in fact we have that $\{n\in \mathbb{Z}: n\neq 0\}\subseteq F$. So, what we have shown in this paragraph is that if $F\in D(k)$ is hereditary and closed, then $\{n\in \mathbb{Z}: n\neq 0\}\subseteq F$.

This, of course, implies that every element of $\mathbb{Z}$ is a pseduo-number. Conversely, it is not difficult to see that all the pseudo-numbers are elements of $\mathbb{Z}$. Suppose that $a\in k$ is not an integer. Then the set $F=k\setminus \{a\}$ is a hereditary closed set that does not contain $a$. Hence, what we have shown in this paragraph is that the pseduo-numbers in the structure $(k, D(k), D(k^{2}), \ldots, \#)$ are precisely the integers.

Now we are in a position to show that $(k, D(k), D(k^{2}), \ldots, \#)\models \neg {\tt SA}$. For, consider the set $k\in D(k)$. By definition $\#k=-(\left|k\setminus k\right|+1)=-1$. Hence, by the results of the previous paragraph, we have that $\#k$ is a pseudo-number. So suppose that ${\tt SA}$ held on the structure $(k, D(k), D(k^{2}), \ldots, \#)$. Then there would be $m$ such that $P(\#k,m)$. Then by definition, there would be sets $X,Y\in D(k)$ such that $\#k=\#X$ and $m=\#Y$ and $\exists \; y\in Y \; X=Y\setminus \{y\}$. Since Hume's~Principle holds on the structure $(k, D(k), D(k^{2}), \ldots, \#)$, we have that $\#k=\#X$ implies that there is a bijection $f:X\rightarrow Y$ that is definable in the structure $k$. By Proposition~\ref{HPPA:prop:defnhpadsfasd}, we have that $k\setminus k$ and $k\setminus X$ are of the same cardinality, so that $X=k$. But then the condition that $y\in Y\setminus X$ implies that $y\in k\setminus k$, which is a contradiction. So, in fact, ${\tt SA}$ does not hold on the structure $(k, D(k), D(k^{2}), \ldots, \#)$.
\end{proof}

\begin{rmk}
In the course of his proof of the existence of a model of ${\tt AHP}_{0}+\neg {\tt SA}$, Linnebo noted several properties which must be had by such models  (cf. \cite{Linnebo2004} pp.~164-165). Since models of ${\tt \Delta^{1}_{1}-HP}_{0}+\neg {\tt SA}$ are automatically models of ${\tt AHP}_{0}+\neg {\tt SA}$, Linnebo's results predict several properties of the model from the previous proposition. In this remark, we briefly explain why the properties identified by Linnebo hold on this structure. First, Linnebo notes that the example of a pseduo-number $n$ witnessing that ${\tt SA}$ fails on the structure $(k, D(k), D(k^{2}), \ldots, \#)$ must be such that $n=\#k$. In the last paragraph of the previous proposition, we showed that $n=\#k$ was such a counterexample. Second, Linnebo notes that the example of a structure $(k, D(k), D(k^{2}), \ldots, \#)\models \neg \tt{SA}$ must be such that $k\setminus X\neq \emptyset$ implies $\#k\neq \#X$. In the context of the model constructed in the previous proposition, this is a consequence of Ax's Theorem (or Proposition~\ref{HPPA:prop:defnhpadsfasd}). Finally, Linnebo notes that the example of a structure $(k, D(k), D(k^{2}), \ldots, \#)\models \neg \tt{SA}$ must contain a copy of both $\omega$ and $\omega^{\ast}$ ordered by the $P$-relation, that is, this structure must contain a copy of the positive integers and the negative integers ordered by the $P$-relation. In the model constructed in the previous theorem, this is reflected in the fact that the pseduo-numbers are precisely the integers.
\end{rmk}

\subsection{Application to O-Minimal Expansions of Real-Closed Fields}

\begin{rmk}
In this section, we apply Theorem~\ref{HPPA:thm:metatheorem} to construct models of ${\tt \Sigma^{1}_{1}-PH}_{0}$ on top~of certain~o-minimal expansions of real-closed fields (cf. Theorem~\ref{HPPA:thm:thebomb}), and an  effectivization of this construction allows us to conclude that ${\tt \Sigma^{1}_{1}-PH}_{0}<_{\mathrm{I}} {\tt ACA}_{0}$ (cf. Corollary~\ref{HPPA:cor:thebomb}), thus filling in a key piece of the interpretability relation (cf. Figure~\ref{HPPA:figure2}). Prior to doing this, we recall some basic notions pertaining to the model theory of o-minimal expansions of real-closed fields, such as dimension and Euler characteristic; the reader who is already familiar with these notions may wish to proceed directly to   Theorem~\ref{HPPA:thm:thebomb}.
\end{rmk}

\begin{defn}\label{defn:hpapa:indexlobminimal}
Suppose that $L$ is a signature extending the signature of linear orders, and suppose that $M$ is an $L$-structure such that $(M, \leq)$ is a dense linear order. Then $M$ is {\it o-minimal} if every definable set is a finite union of points and intervals.\index{O-minimal structure, Definition~\ref{hppa:index:defnsasssaaxiomslin}}
\end{defn}

\begin{prop}
Every real-closed ordered field is o-minimal.
\end{prop}
\begin{proof}
See Marker~\cite{Marker2006ac} Corollary~2.5 p.~11.
\end{proof}

\begin{defn}\label{HPPA:defn:cellls}\index{Cells, in o-minimal structures, Definition~\ref{HPPA:defn:cellls}}
Suppose that $M$ is an o-minimal structure. If $X$ is a definable subset of $M^{n}$, then let $C(X)$ be the set of definable continuous functions $f:X\rightarrow M$, and let $C_{\infty}(X)$ be $C(X)$ plus the two constant functions $-\infty, \infty$. Further, if $f,g\in C_{\infty}(X)$ and $f<g$ on $X$, then let
\begin{equation}
(f,g)_{X} = \{ (x,r)\in X\times R: f(x)<r<g(x)\}
\end{equation}
Then inductively define the notion of a $\sigma$-cell, where $\sigma\in 2^{<\omega}$ is a finite sequence of zeros and ones. First, $0$-cells are points and $1$-cells are open intervals, including $(-\infty, a)$, $(a, \infty)$. Second, given a $\sigma$-cell $X$, the $\sigma 0$-cells are graphs of functions $f\in C(X)$, and the $\sigma 1$-cells are sets $(f,g)_{X}$ where $f,g\in C_{\infty}(X)$. 
\end{defn}

\begin{defn}\label{HPPA:defn:decomp}\index{Cell Decomposition, in o-minimal structures, Definition~\ref{HPPA:defn:decomp}}
Suppose that $M$ is an o-minimal structure. A {\it decomposition} of $M^{n}$ is defined inductively as follows. A decomposition of $M^{1}$ is a finite partition of $M$ with the following form:
\begin{equation}
 \{(-\infty, a_{1}), (a_{1}, a_{2}), \ldots, (a_{k}, +\infty), \{a_{1}\}, \ldots, \{a_{k}\}\}
 \end{equation}
 where $a_{1}<a_{2}<\cdots<a_{k}$.  A decomposition of $M^{m+1}=M^{m}\times M$ is a finite partition of $M^{m+1}$ into cells $\{A_{1}, \ldots, A_{n}\}$ such that the set of projections $\{\pi(A_{1}), \ldots, \pi(A_{n})\}$ is a decomposition of $M^{m}$, where $\pi: M^{m+1}\rightarrow M^{m}$ by $\pi(x_{1}, \ldots, x_{m+1})=(x_{1}, \ldots, x_{m})$. A decomposition of $M^{m}$ is said to {\it partition} a definable set $X\subseteq M^{m}$ if $X$ can be written as a finite union of cells in the decomposition.
\end{defn}

\begin{thm}\label{index:hppa:celldoeomcthemknight}
(Cell Decomposition Theorem)\index{Cell Decomposition Theorem~\ref{index:hppa:celldoeomcthemknight}} Suppose that $M$ is an o-minimal structure. For any finite sequence of $B$-definable sets $A_{1}, \ldots, A_{k}\subseteq M^{m}$, there is a decomposition of $M^{m}$ partitioning each of the $A_{i}$. Moreover, the cells in the decomposition are $B$-definable.
\end{thm}
\begin{proof}
See van den Dries \cite{Dries1998} Theorem~2.11 p.~52.
\end{proof}

\begin{defn}\label{index:hppa:dimeul}
Suppose that $M$ is an o-minimal structure and that $X\subseteq M^{n}$. Then define\index{Dimension, in o-minimal structures, Definition~\ref{index:hppa:dimeul}}
\begin{align}
& \dim(X)=\max\{i_{1}+\cdots+i_{n}: X \mbox{ contains a } (i_{1}, \ldots, i_{n})\mbox{-cell}\} \\
& E(X) =k_{0}-k_{1}+k_{2}-\cdots = \sum_{d=0}^{n} k_{d} (-1)^{d}
\end{align}
where $k_{d}$ is the number of $d$-dimensional cells contained in some cell decomposition of $X$.\index{Euler characteristic, in o-minimal structures, Definition~\ref{index:hppa:dimeul}}
\end{defn}

\begin{rmk}
Note that if $X\subseteq M$, then $\dim(X)>0$ if and only if $X$ contains an open interval. Note that the above definition of Euler dimension can be shown to be independent of the choice of the cell decomposition (cf. \cite{Dries1998} Proposition~2.2 p.~70).
\end{rmk}

\begin{prop}\label{HPPA:prop:defnbound}
Suppose that $M$ is an o-minimal structure and that $\theta(\overline{x}, \overline{y})$ is a $\emptyset$-formula. Then there is a positive integer $N_{\theta}>0$ such that for all $\overline{b}\in M$, it is the case that 
\begin{equation}
\left| \dim(\theta(\cdot, \overline{b})\right|, \left| E(\theta(\cdot, \overline{b}))\right|<N_{\theta}
\end{equation}
Further, for each integer $k$, it is the case that the sets 
\begin{equation}
\{\overline{b}\in M: \dim(\theta(\cdot, \overline{b}))=k\} \; \; \; \& \; \; \; \{\overline{b}\in M: E(\theta(\cdot, \overline{b}))=k\}
\end{equation}
are $\emptyset$-definable. Moreover, the formulas that define these sets and the positive integer $N_{\theta}$ can be uniformly computed from $\theta$.
\end{prop}
\begin{proof}
See van den Dries~\cite{Dries1998} Proposition~1.5 p.~65 and Proposition~2.10 p. 72.
\end{proof}

\begin{prop}\label{HPPA:prop:defnbiject}
Suppose that $M$ is an o-minimal expansion of a real-closed field, and suppose that $X\subseteq M^{n}$ and $Y\subseteq M^{m}$ are definable sets. Then there is a definable bijection $f:X\rightarrow Y$ if and only if $\dim(X)=\dim(Y)$ and $E(X)=E(Y)$.
\end{prop}
\begin{proof}
See van den Dries \cite{Dries1998} p.~132.
\end{proof}

\begin{rmk}
As a simple illustration of this fact, consider the example of the two sets
\begin{equation}
X=(-2,-1)\sqcup \{0\} \sqcup (1,2) \;\;\;\;\;\;\;\;\; Y=(-1,1)
\end{equation}
Both have dimension 1, since they both contain intervals, and their Euler characteristics are the same, namely, $E(X)=1-2=-1$ and $E(Y)=0-1=-1$. Hence, the above proposition predicts that there is a definable bijection $f:X\rightarrow Y$, and in fact this is the case: one simply sends $(-2,-1)$ to $(-1, 0)$ and one sends $0$ to $0$ and one sends $(1,2)$ to $(0,1)$.
\end{rmk}

\begin{prop}\label{HPPA:rmk:rcfdsf}
O-minimal expansions of real closed fields have definable skolem functions.
\end{prop}
\begin{proof}
See van den Dries~\cite{Dries1998} p.~94. 
\end{proof}

\begin{thm}\label{HPPA:thm:thebomb} Suppose that $k$ is a  recursively-saturated o-minimal expansion of a real-closed field. Then there is a computably uniformly definable function $\#:D(k)\rightarrow k$ such that $(k, D(k), D(k^{2}), \ldots, \#)$ is a model of ${\tt \Sigma^{1}_{1}-PH}_{0}+\neg {\tt \Pi^{1}_{1}-HP}_{0}$. Further, there is no function $\partial:D(k)\rightarrow k$ such that $(k, D(k), D(k^{2}), \ldots, \partial)$ is a model of ${\tt \Delta^{1}_{1}-BL}_{0}$.
\end{thm}
\begin{proof}
Since $k$ is a field of characteristic zero, the prime field of $k$ is $\mathbb{Q}$ and the integers $\mathbb{Z}$ are hence embedded into $k$ via $\mathbb{Q}$. Choose a recursive bijection  $\langle \cdot, \cdot\rangle: \mathbb{Z}^{2}\rightarrow \mathbb{Z}$. Using this embedding and this bijection, define $\#:D(k)\rightarrow k$ by $\#X=\langle \dim(X), E(X)\rangle$. Then by  Proposition~\ref{HPPA:prop:defnbiject}, the structure $(k, D(k), D(k^{2}), \ldots, \#)$ is a model of Hume's~Principle. To apply Theorem~\ref{HPPA:thm:metatheorem}~(iii)-(iv), we need to show that $\#: D(k)\rightarrow k$ is computably uniformly definable. So suppose that $\theta(x, \overline{y})$ is an $\emptyset$-formula with non-empty set $\overline{y}$ of parameter variables. Then by Proposition~\ref{HPPA:prop:defnbound}, from the formula $\theta(x,\overline{y})$ we can uniformly compute a positive integer $N_{\theta}>0$ such that
\begin{equation}
k\models \forall \; \overline{b} \; [\left| \dim(\theta(\cdot, \overline{b})\right|, \left| E(\theta(\cdot, \overline{b}))\right|<N_{\theta}]
\end{equation}
as well as $\emptyset$-formulas defining the sets $\{\overline{b}: \dim(\theta(\cdot, \overline{b})=n\}$ and $\{\overline{b}: E(\theta(\cdot, \overline{b}))=n\}$.
Then for each such formula $\theta(x,\overline{y})$ we define the following $\emptyset$-formula $\theta^{\prime}(x,\overline{y})$ as follows:
\begin{equation}
\theta^{\prime}(x,\overline{y})\equiv \bigvee_{i=0}^{N_{\theta}} \bigvee_{j=0}^{N_{\theta}} [\dim(\theta(\cdot, \overline{y}))=i \; \& \; E(\theta(\cdot, \overline{y}))= j]\rightarrow x = \langle i,j\rangle
\end{equation}
Hence, by definition, we have that for any $\overline{a}$
\begin{equation}
\{\#(\theta(\cdot, \overline{a}))\} = \{c: k\models \theta^{\prime}(c, \overline{a})\}
\end{equation}
Hence, by Theorem~\ref{HPPA:thm:metatheorem}~(iii)-(iv) and Proposition~\ref{HPPA:rmk:rcfdsf}, we have that $(k, D(k),D(k^{2}), \ldots, \#)$ is a model of ${\tt \Sigma^{1}_{1}-PH}_{0}$. Further, since the set $\mathrm{rng}(\#)=\mathbb{Z}$ is definable by a $\Sigma^{1}_{1}$-formula in the structure $(k, D(k), D(k^{2}), \ldots, \#)$ but is {\it not} definable in $k$ since $k$ is o-minimal, we have that $(k, D(k),D(k^{2}), \ldots, \#)$ is a model of $\neg {\tt \Pi^{1}_{1}-HP}_{0}$.

Now let us note why there is {\it no} function $\partial:D(k)\rightarrow k$ such that $(k, D(k), D(k^{2}), \ldots, \partial)$ is a model of ${\tt \Delta^{1}_{1}-BL}_{0}$. If there was such a function, then by Proposition~\ref{HPPA:prop:theinjection},  there would be a function $s:k^{2}\rightarrow k$ whose graph was in $D(k^{2})$ and which satisfied $s(x,y)=\partial(\{x,y\})$. Consider the definable set $X=\{(x,y)\in k^{2}: x<y\}$, and note that $\mathrm{dim}(X)=2$. Then $s\upharpoonright X: X\rightarrow k$ is an injection. For, suppose $s(x,y)=s(x^{\prime},y^{\prime})$ for  $(x,y),(x^{\prime},y^{\prime})\in X$. Then $\partial(\{x,y\})=\partial(\{x^{\prime},y^{\prime}\})$ and $x<y$ and $x^{\prime}<y^{\prime}$. Then by Basic Law~V, $\{x,y\}=\{x^{\prime},y^{\prime}\}$ and $x<y$ and $x^{\prime}<y^{\prime}$. Then $x=x^{\prime}$ and $y=y^{\prime}$. Hence, in fact, $s\upharpoonright X: X\rightarrow k$ is an injection. Then trivially $s\upharpoonright X: X\rightarrow \mathrm{rng}(s\upharpoonright X)$ is a bijection whose graph is in $D(k^{2})$. Then by the left-to-right direction of Proposition~\ref{HPPA:prop:defnbiject}, it would follow that 
\begin{equation}
2=\mathrm{dim}(X)=\mathrm{dim}( \mathrm{rng}(s\upharpoonright X))\leq \mathrm{dim}(k)=1
\end{equation}
which is a contradiction.
\end{proof}

\begin{rmk}
It is our claim that all of the results quoted and proved in this subsection can be proven in ${\tt ACA}_{0}$ for o-minimal structures $M$ with ${\tt ACA}_{0}$-provable quantifier-elimination, such as real-closed fields (cf. Marker \cite{Marker2006ac} Theorem~2.3 p.~10, Simpson~\cite{Simpson2009aa} Lemma~II.9.6 p.~98). The reason for this is that~(i) the proofs from van den Dries \cite{Dries1998} all concern properties of definable sets, as opposed to properties of the defining formula, and~(ii) the proofs from  van den Dries \cite{Dries1998} all proceed by induction on the cartesian power of the definable set. It is worthwhile to say a little bit more about each of these points.

In regard to~(i), the proofs in this section from van den Dries \cite{Dries1998} are all concerned with properties of a definable set $X$, so that the definable set $X$ has the property regardless of which particular formula is used to define $X$. For instance, the property of $X$'s being~a~cell has this feature, since a definable set $X\subseteq M$ is e.g. an interval or a point regardless of whether the formula $\varphi$ or the formula $\psi$ is being used to define it (where $\varphi$ and $\psi$ are two formulas that do in fact define $X$). By the same token, the proofs in this section from van den Dries \cite{Dries1998} are {\it not} concerned with the syntactic complexity of given formulas, for instance, whether or not they are $\Pi^{0}_{3}$-formulas or $\Pi^{0}_{4}$-formulas. Hence, if $M$ has quantifier-elimination, then for the purposes of the proofs in this section from van den Dries \cite{Dries1998}, we can take the quantifier-free formulas as representatives for the definable sets. For instance, in proving the Cell Decomposition Theorem in this manner, we would in fact prove that e.g. for every finite sequence of quantifier-free formulas $\varphi_{1}(\overline{x}), \ldots, \varphi_{k}(\overline{x})$ in $m$-free variables, there is a quantifier-free decomposition of $M^{m}$ partitioning each of the $\varphi_{i}(\overline{x})$.

In regard to~(ii), the proofs in this section from van den Dries \cite{Dries1998} all proceed by induction, where it is first shown that the definable subsets of $M$ have a given property, and then it is shown that if the definable subsets of $M^{n}$ have a given property, then the definable subsets of $M^{n+1}$ have this given property. Given our discussion in the previous paragraph, when proving these theorems in ${\tt ACA}_{0}$, we would in fact prove that the quantifier-free formulas $\varphi(x)$ have a given property, and that if the quantifier-free formulas $\varphi(x_{1}, \ldots, x_{n})$ have a given property, then the quantifier-free formulas $\varphi(x_{1}, \ldots, x_{n+1})$ have a given property. Since ${\tt ACA}_{0}$ has the mathematical induction axiom for all sets $X$, it suffices to note that ${\tt ACA}_{0}$ has enough comprehension to show that the sets $X$ on which it is doing induction exist. Here it suffices to note that the proofs in this section from van den Dries \cite{Dries1998} all concern properties of the definable sets that can be expressed by~(iii) finitely many quantifiers over quantifier-free definable sets and by~(iv) finitely many quantifiers over the structure $M$. For instance, to reiterate the point made in the last paragraph, in proving the Cell Decomposition Theorem in this fashion, we must show that for every $m$ and every finite sequence of quantifier-free formulas $\varphi_{1}(\overline{x}), \ldots, \varphi_{k}(\overline{x})$ in $m$-free variables, there is a quantifier-free decomposition of $M^{m}$ partitioning each of the $\varphi_{i}(\overline{x})$. In terms of~(iii), this involves a universal quantifier over quantifier-free definable sets followed by an existential quantifier over quantifier-free definable sets. In terms of~(iv), this involves a universal quantifier to say that e.g. the cells in the decomposition are disjoint and another universal quantifier to say that e.g. $\varphi_{i}(\overline{x})$ can be written as a finite union of pairwise disjoint cells in the decomposition. Since the number of quantifiers in~(iii) and~(iv) is fixed in advance, ${\tt ACA}_{0}$ can prove that the set on which one is doing induction exists. In this way, the proofs from van den Dries \cite{Dries1998} can be translated word-for-word into proofs in ${\tt ACA}_{0}$ for o-minimal structures $M$ which have ${\tt ACA}_{0}$-provable  quantifier-elimination, such as real-closed fields.
\end{rmk}

\begin{cor}\label{HPPA:cor:thebomb} ${\tt \Sigma^{1}_{1}-PH}_{0}<_{\mathrm{I}} {\tt ACA}_{0}$.
\end{cor}
\begin{proof}
This follows from Proposition~\ref{HPPA:conpropa}, the fact that ${\tt ACA}_{0}$ proves the existence of recursively saturated elementary extensions (cf. Simpson~\cite{Simpson2009aa} Lemma~IX.4.2 pp.~379), and the fact that the proof of the previous theorem can be formalized in ${\tt ACA}_{0}$ for o-minimal expansions of real-closed fields with ${\tt ACA}_{0}$-provable quantifier-elimination, such as real-closed fields.
\end{proof}

\subsection{Application to Separably Closed Fields of Finite Imperfection Degree}\label{hppa:lastsec}

\begin{rmk}
In the two previous subsections, we applied Theorem~\ref{HPPA:thm:metatheorem}  to construct models of ${\tt \Delta^{1}_{1}-HP}_{0}$ on top~of various fields, such as certain algebraically closed fields and o-minimal expansions of real-closed fields. We noted in both Theorem~\ref{hppa:coolacfthm} and Theorem~\ref{HPPA:thm:thebomb} that this construction  {\it cannot} result in models of  ${\tt \Delta^{1}_{1}-BL}_{0}$. Hence, this raises the question of whether there is some natural field such that one can  apply Theorem~\ref{HPPA:thm:metatheorem}  to it to obtain models of  ${\tt \Delta^{1}_{1}-BL}_{0}$. In this section, we isolate certain model-theoretic conditions on a field (such a uniform elimination of imaginaries) which suffice to ensure that such a construction can succeed (cf. Theorem~\ref{HPPA:thmasdfadsf-alt}). Then we note that separably closed fields of finite imperfection degree satisfy these model-theoretic conditions (cf. Theorem~\ref{HPPA:thmasdfadsf}).
\end{rmk}

\begin{defn}\label{hppa:index:uniformelimination}
Suppose that $M$ is an $L$-structure. Then $M$ has {\it uniform elimination of imaginaries} if for every  $\emptyset$-definable equivalence relation $E$ on $M^{n}$ there is an $\emptyset$-definable function $f: M^{n}\rightarrow M^{m}$ for some $m>0$ such that \index{Uniform Elimination of Imaginaries, Definition~\ref{hppa:index:uniformelimination}}
\begin{equation}
\overline{z}E\overline{y}\Longleftrightarrow f(\overline{z})=f(\overline{y})
\end{equation}
\end{defn}

\begin{defn}\label{hppa:index:uniformelimination3333}
Suppose that $M$ is an $L$-structure. Then $M$ has a {\it $\emptyset$-definable pairing function} if there is an $\emptyset$-definable injection $\iota: M^{2}\rightarrow M$.\index{Definable pairing function~\ref{hppa:index:uniformelimination3333}}
\end{defn}

\begin{prop}\label{HPPA:prop:thehehdaf}
Suppose that $M$ has uniform elimination of imaginaries and an $\emptyset$-definable pairing function. Then for every $\emptyset$-definable equivalence relation $E$ on $M^{n}$ there is an $\emptyset$-definable function $f: M^{n}\rightarrow M$ such that
\begin{equation}
\overline{z}E\overline{y}\Longleftrightarrow f(\overline{z})=f(\overline{y})
\end{equation}
\end{prop}
\begin{proof}
By hypothesis, $M$ has an $\emptyset$-definable pairing function $\iota:M^{2}\rightarrow M$. Then define injections $j_{n}:M^{n}\rightarrow M$ recursively as follows:
\begin{align}
& j_{1}(x_{1}) = x_{1} \\
& j_{2}(x_{1}, x_{2}) = \iota(x_{1}, x_{2}) \\
& j_{n+1}(x_{1}, \ldots x_{n}, x_{n+1})= \iota(j_{n}(x_{1}, \ldots, x_{n}), x_{n+1})
\end{align}
Finally, given a function $f: M^{n}\rightarrow M^{m}$ for some $m>0$ which witnesses the uniform elimination of imaginaries, simply define $f^{\ast}=j_{m}\circ f$.
\end{proof}

\begin{prop}\label{HPPA:eqn:disjointsss}
Suppose that $M$ has an $\emptyset$-definable pairing function and that $\mathrm{dcl}(\emptyset)$ has at least two elements. Then there is a uniformly computable sequence of injections $\iota_{n}:M\rightarrow M$ such that $n\neq m$ implies $\mathrm{rng}(\iota_{n})\cap \mathrm{rng}(\iota_{m})=\emptyset$.
\end{prop}
\begin{proof}
Suppose that $\iota:M^{2}\rightarrow M$ is the $\emptyset$-definable injection and that $b,c\in \mathrm{dcl}(\emptyset)$ are distinct. Then define injections $\iota_{n}: M\rightarrow M$ recursively as follows:
\begin{align}
& \iota_{0}(x)=\iota(c, \iota(c,x)) \\
& \iota_{2s+1}(x) = \iota(b, \iota_{2s}(x)) \\
& \iota_{2s+2}(x) = \iota(c, \iota_{2s+1}(x))
\end{align}
By construction, all the functions $\iota_{n}:M\rightarrow M$ are injections.  So it remains to show by induction on $m\leq n$ that $\mathrm{rng}(\iota_{n})\cap \mathrm{rng}(\iota_{m})=\emptyset$ when $m\neq n$. Clearly this holds for $n=0$. So suppose it holds for $n$. If $n$ is even then $n=2s$ and $n+1=2s+1$. Suppose that $m<n+1$ is such that $\mathrm{rng}(\iota_{n+1})\cap \mathrm{rng}(\iota_{m})\neq \emptyset$. Then there are $x,y$ such that $\iota_{n+1}(x)=\iota_{m}(y)$. Expanding this equation on the left, we have $\iota(b, \iota_{2s}(x)) = \iota_{2s+1}(x) = \iota_{n+1}(x) = \iota_{m}(y)$. Then by construction, $\iota_{m}(y)=\iota(b, \iota_{2t}(y))$ for some $2t+1=m$. Then $\iota_{m-1}(y)=\iota_{2t}(y)=\iota_{2s}(x)=\iota_{n}(x)$, which contradicts our induction hypothesis on $n$. On the other hand, if $n$ is odd then $n=2s+1$ and $n+1=2s+2$. Suppose that $m<n+1$ is such that $\mathrm{rng}(\iota_{n+1})\cap \mathrm{rng}(\iota_{m})\neq \emptyset$. Then there are $x,y$ such that $\iota_{n+1}(x)=\iota_{m}(y)$. Then expanding this equation on the left we have $\iota(c, \iota_{2s+1}(x)) = \iota_{2s+2}(x)=\iota_{n+1}(x)=\iota_{m}(y)$. There are then two cases. First suppose that $m=0$. Then by construction $\iota_{m}(y)=\iota(c,\iota(c,y))$. Then $\iota(b, \iota_{2s}(x))=\iota_{2s+1}(x)=\iota(c,y)$, and so $b=c$, which is a contradiction. Second, suppose that $m>0$. Then by construction, $\iota_{m}(y)=\iota(c, \iota_{2t+1}(y))$ for some $2t+2=m$. Then $\iota_{m-1}(y)=\iota_{2t+1}(y)=\iota_{2s+1}(x)=\iota_{n}(x)$, contradicting our induction hypothesis on $n$.
\end{proof}

\begin{rmk}
Prior to proving the following theorem, let us here attempt to sketch the intuitive proof~idea. Suppose that $M$ has uniform elimination of imaginaries and a $\emptyset$-definable pairing function. Then given a formula $\theta(x,\overline{y})$ with a set of parameter variables $\overline{y}$ of length $\ell>0$, these assumptions yield an $\emptyset$-definable function $\partial_{\theta}: M^{\ell}\rightarrow M$ such that 
\begin{equation}
M\models [\forall \; x \; \theta(x,\overline{a})\leftrightarrow \theta(x,\overline{b})] \Longleftrightarrow \partial_{\theta}(\overline{a})=\partial_{\theta}(\overline{b})
\end{equation}
Intuitively, the idea is to build a model $(M, D(M), D(M^{2}), \ldots, \partial)$ of Basic Law V by setting 
\begin{equation}
\partial(\theta(\cdot, \overline{a}))=\partial_{\theta}(\overline{a})
\end{equation}
However, there are two potential problems. First, such a function will not be well-defined, since a given set $X\in D(M)$ will be defined by many formulas $\theta_{1}(\cdot, \overline{a}), \theta_{2}(\cdot, \overline{b}), \ldots$. Second, it is not obvious that such a function will be injective, which is required by Basic Law~V. Overcoming these problems is the only thing that makes the below proof non-trivial. In particular, the first problem is overcome simply by fixing beforehand an enumeration of the all potential defining formulas $\theta_{1}(x,\overline{y}), \ldots, \theta_{n}(x,\overline{y}), \ldots$, and then defining $\partial(X)$ to be $\partial_{\theta_{n}}(\overline{a})$ for the first $\theta_{n}(x,\overline{a})$ in the enumeration that defines $X$ for some tuple~$\overline{a}$. The second problem is overcome by including additional hypotheses on $M$ which ensure that we can partition $M=\bigsqcup_{n} M_{n}$ and likewise ensure that $\partial_{\theta_{n}}(\overline{a})$ always takes values in $M_{n}$. The previous proposition was in effect devoted to explaining why the hypothesis of a  $\emptyset$-definable pairing function with $\left|\mathrm{dcl}(\emptyset)\right|>1$ ensure that we can construct such a partition.
\end{rmk}

\begin{thm}\label{HPPA:thmasdfadsf-alt}
Suppose that $M$ is a $Th(M)$-computably saturated structure such that (i)~$M$ has uniform elimination of imaginaries, (ii)~$M$ has an $\emptyset$-definable pairing function, and (iii) $\mathrm{dcl}(\emptyset)$ has at least two elements. Then there is a $Th(M)$-computably uniformly definable function $\partial:D(M)\rightarrow M$ such that $(M, D(M), D(M^{2}), \ldots, \partial)$ is a model of ${\tt \Delta^{1}_{1}-BL}_{0}$.
\end{thm}
\begin{proof}
To apply Theorem~\ref{HPPA:thm:metatheorem}~(iii)-(iv), we need to define an injection $\partial: D(M)\rightarrow M$ that is $Th(M)$-computably uniformly definable. Choose a fixed computable enumeration of the $\emptyset$-formulas $\theta(x,\overline{y})$ with non-empty set $\overline{y}$ of parameter variables of length $\ell_{n}$ as $\theta_{1}(x,\overline{y}), \ldots, \theta_{n}(x,\overline{y}), \ldots$. For each $n>0$ and $0<m\leq n$, consider the following $\emptyset$-definable sets $U_{n,m}\subseteq M^{\ell_{n}}$, where again $\ell_{n}$ is the length of the tuple $\overline{y}$ in $\theta_{n}(x, \overline{y})$:
\begin{align}
& U_{1,1}=M^{\ell_{1}} \\
& U_{2,1}=\{\overline{a}\in M^{\ell_{2}}: \exists \; \overline{b}\in M^{\ell_{1}} \; [\forall \; x \; \theta_{2}(x, \overline{a})\leftrightarrow \theta_{1}(x, \overline{b})]\} \\
& U_{2,2}=M^{\ell_{2}}\setminus U_{2,1} \\
& U_{3,1} =\{\overline{a}\in M^{\ell_{3}}: \exists \; \overline{b}\in M^{\ell_{1}} \; [\forall \; x \; \theta_{3}(x, \overline{a})\leftrightarrow \theta_{1}(x, \overline{b})]\} \\
& U_{3,2} =\{\overline{a}\in M^{\ell_{3}}: \exists \; \overline{b}\in M^{\ell_{2}} \; [\forall \; x \; \theta_{3}(x, \overline{a})\leftrightarrow \theta_{2}(x, \overline{b})]\}\setminus U_{3,1} \\
& U_{3,3} = M^{\ell_{3}}\setminus (U_{3,1}\cup U_{3,2})
\end{align}
Note that for a fixed $n>0$ that the sets $U_{n,1}, \ldots, U_{n,n}$ partition $M^{\ell^{n}}$ and that the formulas defining these sets are uniformly computable from $n$. Then define $\emptyset$-definable equivalence relations on $M^{\ell_{n}}$ as follows:
\begin{equation}\label{HPPA:eqn:rhsdefining4}
\overline{y} E_{n} \overline{z} \Longleftrightarrow M\models [\forall \; x \; \theta_{n}(x, \overline{y})\leftrightarrow \theta_{n}(x,\overline{z})]
\end{equation}
Note by definition that any two elements $\overline{y}$ and $\overline{z}$ which are $E_{n}$-equivalent are in the same member of the partition $U_{n,1}, \ldots, U_{n,n}$ of $M^{\ell_{n}}$. By Proposition~\ref{HPPA:prop:thehehdaf} from $\theta_{n}(x, \overline{y})$ we can uniformly $Th(M)$-compute a $\emptyset$-definable function $f_{n}:M^{\ell_{n}}\rightarrow M$ such that 
\begin{equation}\label{HPPA:eqn:rhsdefining3}
M\models [\forall \; x \; \theta_{n}(x, \overline{y})\leftrightarrow \theta_{n}(x,\overline{z})] \Longleftrightarrow \overline{y} E_{n} \overline{z} \Longleftrightarrow f_{n}(\overline{y})=f_{n}(\overline{z})
\end{equation}
By Proposition~\ref{HPPA:eqn:disjointsss}, we can uniformly compute a sequence of injections $\iota_{n}: M\rightarrow M$ with disjoint ranges, and we can define $g_{n}=\iota_{n}\circ f_{n}$. Finally, define $\partial:D(M)\rightarrow M$ by
\begin{equation}\label{HPPA:eqn:rhsdefining1}
\partial(\theta_{n}(\cdot, \overline{a}))= c\Longleftrightarrow \bigwedge_{m=1}^{n}[\overline{a}\in U_{n,m} \rightarrow (\exists \; \overline{b}\in M^{\ell_{m}} \; \& \; \forall \; x \; \theta_{n}(x, \overline{a})\leftrightarrow \theta_{m}(x, \overline{b}) \; \& \; c=g_{m}(\overline{b}))]
\end{equation}

First let us show that $\partial:D(M)\rightarrow M$ is a well-defined function. So suppose that $\theta_{n}(\cdot, \overline{a})$ and $c$ satisfy the right-hand side of equation~(\ref{HPPA:eqn:rhsdefining1}) and that $\theta_{n^{\prime}}(\cdot, \overline{a}^{\prime})$ and $c^{\prime}$ also satisfy the right-hand side of equation~(\ref{HPPA:eqn:rhsdefining1}), and suppose that $\theta_{n}(\cdot, \overline{a})$ and $\theta_{n^{\prime}}(\cdot, \overline{a}^{\prime})$ define the same set. Then we must show that $c=c^{\prime}$. Without loss of generality, $n^{\prime}\leq n$. If $n^{\prime}=n$, then since $\theta_{n}(\cdot, \overline{a})$ and $\theta_{n^{\prime}}(\cdot, \overline{a}^{\prime})$ define the same set, we have that $\overline{a}$ and $\overline{a}^{\prime}$ are $E_{n}$-equivalent and hence are in the same set $U_{n,m}$. Then by the right-hand side of equation~(\ref{HPPA:eqn:rhsdefining1}), we have that there are $\overline{b}, \overline{b}^{\prime}\in M^{\ell_{m}}$ such that 
\begin{align}
& M\models \forall \; x \;  \theta_{m}(x, \overline{b})\leftrightarrow \theta_{n}(x, \overline{a}) \leftrightarrow \theta_{n}(x, \overline{a}^{\prime})\leftrightarrow \theta_{m}(x, \overline{b}^{\prime}) \label{HPPA:eqn:rhsdefining2}\\
& c=g_{m}(\overline{b}) \label{HPPA:eqn:rhsdefining5}\\
& c^{\prime} = g_{m}(\overline{b}^{\prime}) \label{HPPA:eqn:rhsdefining6}
\end{align}
But by equation~(\ref{HPPA:eqn:rhsdefining2}), we have that $\overline{b}$ and $\overline{b}^{\prime}$ are $E_{m}$-equivalent, and hence by equation~(\ref{HPPA:eqn:rhsdefining3}), we have that $f_{m}(\overline{b})=f_{m}(\overline{b}^{\prime})$ and so by equations~(\ref{HPPA:eqn:rhsdefining5})-(\ref{HPPA:eqn:rhsdefining6}) we have that 
\begin{equation}
c=g_{m}(\overline{b})=\iota_{m}\circ f_{m}(\overline{b})=\iota_{m}\circ f_{m}(\overline{b}^{\prime}) =g_{m}(\overline{b}^{\prime})=c^{\prime}
\end{equation}
In the case where $n^{\prime}<n$, we have that $\overline{a}\in U_{n,m}$ and $\overline{a}^{\prime}\in U_{n^{\prime},m^{\prime}}$ and so by the right-hand side of equation~(\ref{HPPA:eqn:rhsdefining1}), we have that there is $\overline{b}\in M^{\ell_{m}}, \overline{b}^{\prime}\in M^{\ell_{m^{\prime}}}$ such that 
\begin{align}
& M\models \forall \; x \;  \theta_{m}(x, \overline{b})\leftrightarrow \theta_{n}(x, \overline{a}) \leftrightarrow \theta_{n}(x, \overline{a}^{\prime})\leftrightarrow \theta_{m^{\prime}}(x, \overline{b}^{\prime}) \label{HPPA:eqn:rhsdefining20}\\
& c=g_{m}(\overline{b}) \label{HPPA:eqn:rhsdefining21}\\
& c^{\prime} = g_{m^{\prime}}(\overline{b}^{\prime}) \label{HPPA:eqn:rhsdefining22}
\end{align}
Then by equation~(\ref{HPPA:eqn:rhsdefining20}) and the definition of the sets $U_{n,m}$, we must have that $m=m^{\prime}$.  
Then by equation~(\ref{HPPA:eqn:rhsdefining20}) again, we have that $\overline{b}$ and $\overline{b}^{\prime}$ are $E_{m}$-equivalent, and, hence, by equation~(\ref{HPPA:eqn:rhsdefining3}), we have that $f_{m}(\overline{b})=f_{m}(\overline{b}^{\prime})$, and so by equations~(\ref{HPPA:eqn:rhsdefining21})-(\ref{HPPA:eqn:rhsdefining22}) we have that 
\begin{equation}
c=g_{m}(\overline{b})=\iota_{m}\circ f_{m}(\overline{b})=\iota_{m}\circ f_{m}(\overline{b}^{\prime}) =g_{m}(\overline{b}^{\prime})=c^{\prime}
\end{equation}
Therefore, $\partial:D(M)\rightarrow M$ is a well-defined function. 

Now let us show that $\partial:D(M)\rightarrow M$ is an injection. Suppose that $\theta_{n}(\cdot, \overline{a})$ and $c$ satisfy the right-hand side of equation~(\ref{HPPA:eqn:rhsdefining1}) and that $\theta_{n^{\prime}}(\cdot, \overline{a}^{\prime})$ and $c^{\prime}$ satisfy the right-hand side of equation~(\ref{HPPA:eqn:rhsdefining1}) and suppose that $c=c^{\prime}$. Then we must show that $\theta_{n}(\cdot, \overline{a})$ and $\theta_{n^{\prime}}(\cdot, \overline{a}^{\prime})$ define the same set. We have that $\overline{a}\in U_{n,m}$ and $\overline{a}^{\prime}\in U_{n^{\prime},m^{\prime}}$, and by the right-hand side of equation~(\ref{HPPA:eqn:rhsdefining1}), we have that there is $\overline{b}\in M^{\ell_{m}}, \overline{b}^{\prime}\in M^{\ell_{m^{\prime}}}$ such that 
\begin{align}
& M\models \forall \; x \; \theta_{n}(x, \overline{a})\leftrightarrow \theta_{m}(x, \overline{b}) \label{HPPA:eqn:rhsdefining29}\\
& M\models \forall \; x \; \theta_{n^{\prime}}(x, \overline{a}^{\prime})\leftrightarrow \theta_{m^{\prime}}(x, \overline{b}^{\prime}) \label{HPPA:eqn:rhsdefining30}\\
& g_{m}(\overline{b})=c=c^{\prime}=g_{m^{\prime}}(\overline{b}^{\prime})  \label{HPPA:eqn:rhsdefining31}
\end{align}
Since $g_{m}=\iota_{m}\circ f_{m}$ and since the functions $\iota_{m}$ have distinct ranges, equation~(\ref{HPPA:eqn:rhsdefining31}) implies that $m=m^{\prime}$ and since $g_{m}=\iota_{m}\circ f_{m}$ and $\iota_{m}$ is an injection, we have that equation~(\ref{HPPA:eqn:rhsdefining31}) implies that $f_{m}(\overline{b})=f_{m^{\prime}}(\overline{b}^{\prime})$, which by equation~(\ref{HPPA:eqn:rhsdefining3}) implies that $\theta_{m}(\cdot, \overline{b})$ and $\theta_{m^{\prime}}(\cdot, \overline{b}^{\prime})$ define the same set. This in turn implies with equations~(\ref{HPPA:eqn:rhsdefining29})-(\ref{HPPA:eqn:rhsdefining30}) that $\theta_{n}(\cdot, \overline{a})$ and $\theta_{n^{\prime}}(\cdot, \overline{a}^{\prime})$ define the same set, which is what we wanted to show. Hence, in fact $\partial:D(M)\rightarrow M$ is an injection.

So, $\partial:D(M)\rightarrow M$ is well-defined and indeed an injection. Note that by its very definition in equation~(\ref{HPPA:eqn:rhsdefining1}), we have that $\partial:D(M)\rightarrow M$ is $Th(M)$-computably uniformly definable. Hence, by Theorem~\ref{HPPA:thm:metatheorem}~(iii)-(iv), we have that $(M, D(M), D(M^{2}), \ldots, \partial)$ is a model of ${\tt \Delta^{1}_{1}-BL}_{0}$.
\end{proof}

\begin{defn}\label{hppa:index:sepd33399267}
Suppose that $k$ is field of characteristic $p>0$. Then $k$ is a {\it separably closed field of finite imperfection degree} if~(i) there is a finite set $B\subseteq k$ such that the set of monomials $\{b_{1}^{m_{1}}\cdots b_{e}^{m_{e}}:  0\leq m_{i}<p \; \& \; b_{1}, \ldots, b_{e}\in B\}$ is a basis for $k$ over $k^{p}$, and if~(ii) every $f\in k[x]$ such that $f^{\prime}\neq 0$ has a root in $k$.\index{Separably closed field of finite imperfection degree, Definition~\ref{hppa:index:sepd33399267}}
\end{defn}

\begin{thm}\label{HPPA:thmasdfadsf}
Suppose that $k$ is a recursively saturated separably closed field of finite imperfection degree. Then there is a computably uniformly definable function $\partial:D(k)\rightarrow k$ such that $(k, D(k), D(k^{2}), \ldots, \partial)$ is a model of ${\tt \Delta^{1}_{1}-BL}_{0}$.
\end{thm}
\begin{proof}
This follows immediately from the fact that such fields satisfy the antecedents of the previous theorem and have a computable theory when names are added for the finite set $B$ from the previous definition (cf. Messmer~\cite{Messmer2006aa} Proposition~4.2 p.~140, p.~143, Remark~4.4 p.~141).
\end{proof}

\begin{rmk}\label{HPPA:seprmk}
If we knew that all the elements of the proof of the previous theorem were formalizable in ${\tt ACA}_{0}$, then we could infer from the proof of the above theorem and Proposition~\ref{HPPA:conpropa} that we have ${\tt \Delta^{1}_{1}-BL}_{0}<_{\mathrm{I}} {\tt ACA}_{0}$. It is clear from the proof  that this comes down to determining whether or not the uniform elimination of imaginaries for separably closed fields of finite imperfection degree is provable in ${\tt ACA}_{0}$.
\end{rmk}

\section{Further Questions}\label{HPPA:hppa:furtherquestions}

\begin{Q}
In Figure~\ref{HPPA:figure1}, we summarized what is known about the provability relation. Two questions which remain open are the following:  does ${\tt \Delta^{1}_{1}-BL}_{0}$ imply ${\tt \Sigma^{1}_{1}-LB}_{0}$ and does ${\tt \Pi^{1}_{1}-HP}_{0}$ imply ${\tt \Sigma^{1}_{1}-PH}_{0}$?
\end{Q}

\begin{Q}\label{HPPA:axq}
In Remark~\ref{HPPA:axrmk}, we noted that if Ax's Theorem~\ref{HPPA:prop:axtheorem} is provable in ${\tt ACA}_{0}$, then we would have another proof of ${\tt \Delta^{1}_{1}-HP}_{0}<_{\mathrm{I}} {\tt ACA}_{0}$ besides the proof from Corollary~\ref{HPPA:cor:thebomb}. Hence, is Ax's Theorem~\ref{HPPA:prop:axtheorem} provable in ${\tt ACA}_{0}$?
\end{Q}

\begin{Q}\label{HPPA:sepq}
In Remark~\ref{HPPA:seprmk}, we noted that if the uniform elimination of imaginaries for separably closed fields is provable in ${\tt ACA}_{0}$, then we would have ${\tt \Delta^{1}_{1}-BL}_{0}<_{\mathrm{I}} {\tt ACA}_{0}$. Hence, is the uniform elimination of imaginaries for separably closed fields provable in ${\tt ACA}_{0}$?
\end{Q}

\begin{Q} The results in Heck \cite{Heck1993}, Ganea \cite{Ganea2007}, and Visser \cite{Visser2009ab} imply that ${\tt ABL}_{0}$ is mutually interpretable with Robinson's~${\tt Q}$. Is ${\tt \Delta^{1}_{1}-BL}_{0}$ mutually interpretable  with Robinson's~$Q$?
\end{Q}

\begin{Q} What is the exact interpretability strength of ${\tt AHP_{0}}$ and ${\tt \Delta^{1}_{1}-HP}_{0}$? Are these theories interpretable in Robinson's~$Q$?
\end{Q}

\begin{Q}
In \S~2.2, and in particular around equation~(\ref{HPPA:eqn:counterthecounter}), we pointed out that there is no function symbol in our language for the mapping $(R, n)\mapsto \#(R_{n})$, where $R$ is a binary relation and $R_{n}=\{m: Rnm\}$. The inclusion of such a function symbol will not affect systems which contain the $\Delta^{1}_{1}$-comprehension schema, since the graph of this function is $\Delta^{1}_{1}$-definable (cf. equation~(\ref{HPPA:eqn:counterthecounter})). However, in Propositions~\ref{hppa:binaryiswierd1}-\ref{hppa:binaryiswierd2}, we pointed that ${\tt AHP}_{0}$ and ${\tt ABL}_{0}$ do not prove the existence of the graph of this function $(R, n)\mapsto \#(R_{n})$, in the sense that  ${\tt AHP}_{0}$ and ${\tt ABL}_{0}$ do not prove that the binary relation $\{(n,m): \#(R_{n})=m\}$ exists for every binary relation $R$. Does the addition of this function symbol affect the interpretability strength of ${\tt AHP}_{0}$ and ${\tt ABL}_{0}$? In particular, do the Heck-Visser-Ganea results about the mutual interpretability of ${\tt ABL}_{0}$ and Robinson's $Q$ mentioned in \S~\ref{HPPA:mainresultsasas} still hold if we add a function symbol for $(R, n)\mapsto \#(R_{n})$?
\end{Q}

\section{Acknowledgements}

\noindent This is material from my dissertation, which I wrote at the University of Notre~Dame under the supervision of Dr.~Peter~Cholak and Dr.~Michael~Detlefsen, and I would like to take this opportunity to thank them  for their support and guidance during my graduate work.

Further, this paper has been materially improved by my having had the opportunity to present it to several gracious audiences. In particular, I would like to thank Logan Axon, Joshua Cole, Stephen Flood, and Christopher Porter for listening to me speak on this material in Dr.~Cholak's seminar, and I would like to thank Antonio Montalb\'an and the other organizers and participants in the University of Chicago Logic Seminar for listening to me speak on this material in~March~2009. I would also like to thank {\O}ystein Linnebo, Richard Pettigrew, and Albert Visser, with whom I had some very helpful discussions of this material subsequent to my arrival in Paris in summer~2009.

During my graduate studies I was fortunate enough to be the beneficiary of generous financial support from many institutions and groups, whom I would like to thank here, including: the Mathematics Department at Notre Dame, the Philosophy Department at Notre Dame, the Ahtna Heritage Foundation, DAAD (Deutscher Akademischer Austausch Dienst), the George-August Universit\"at G\"ottingen, the National Science Foundation (under NSF Grants 02-45167, EMSW21-RTG-03-53748, EMSW21-RTG-0739007, and DMS-0800198), the Alexander von Humboldt Stiftung TransCoop Program, and the Ideals of Proof Project, which in turn was funded and supported by ANR (L'Agence nationale de la recherche), Universit\'e Paris~Diderot -- Paris~7, Universit\'e Nancy~2, Coll\`ege~de~France, and Notre~Dame.

%% The Appendices part is started with the command \appendix;
%% appendix sections are then done as normal sections
%% \appendix

%% \section{}
%% \label{}

%% References
%%
%% Following citation commands can be used in the body text:
%% Usage of \cite is as follows:
%%   \cite{key}          ==>>  [#]
%%   \cite[chap. 2]{key} ==>>  [#, chap. 2]
%%   \citet{key}         ==>>  Author [#]

%% References with bibTeX database:

\newpage

\bibliographystyle{plain}
\bibliography{walsh-sean.bib}

\begin{thebibliography}{10}

\bibitem{Ax1968}
James Ax.
\newblock The {E}lementary {T}heory of {F}inite {F}ields.
\newblock {\em Annals of Mathematics}, 88:239--271, 1968.

\bibitem{Barwise1975aa}
Jon Barwise and John Schlipf.
\newblock On {R}ecursively {S}aturated {M}odels of {A}rithmetic.
\newblock In A.~Dold and B.~Eckmann, editors, {\em Model {T}heory and
  {A}lgebra}, volume 498 of {\em Lecture Notes in Mathematics}, pages 42--55.
  Springer, Berlin, 1975.

\bibitem{Boolos1987}
George Boolos.
\newblock The {C}onsistency of {F}rege's {F}oundations of {A}rithmetic.
\newblock In Judith~Jarvis Thomson, editor, {\em On {B}eing and {S}aying:
  {E}ssays in {H}onor of {R}ichard {C}artwright}, pages 3--20. MIT Press,
  Cambridge, 1987.
\newblock Reprinted in \cite{Boolos1998}, \cite{Demopoulos1995}.

\bibitem{Boolos1995aa}
George Boolos.
\newblock Frege's {T}heorem and the {P}eano {P}ostulates.
\newblock {\em The Bulletin of Symbolic Logic}, 1(3):317--326, 1995.
\newblock Reprinted in \cite{Boolos1998}.

\bibitem{Boolos1996aa}
George Boolos.
\newblock On the {P}roof of {F}rege's {T}heorem.
\newblock In Adam Morton and Stephen~P. Stich, editors, {\em Benacerraf and His
  Critics}, pages 143--159. Blackwell, 1996.
\newblock Reprinted in \cite{Boolos1998}.

\bibitem{Boolos1998}
George Boolos.
\newblock {\em Logic, {L}ogic, and {L}ogic}.
\newblock Harvard University Press, Cambridge, MA, 1998.
\newblock Edited by Richard Jeffrey.

\bibitem{Boolos1998aa}
George Boolos and Richard~G. Heck~Jr.
\newblock Die {G}rundlagen der {A}rithmetik, {S}ections 82-83.
\newblock In Matthias Schirn, editor, {\em Philosophy of {M}athematics
  {T}oday}, pages 407--428. Clarendon Press, 1998.
\newblock Reprinted in \cite{Boolos1998}.

\bibitem{Burgess2005}
John~P. Burgess.
\newblock {\em Fixing {F}rege}.
\newblock Princeton Monographs in Philosophy. Princeton University Press,
  Princeton, 2005.

\bibitem{Burgess1998}
John~P. Burgess and A.~P. Hazen.
\newblock Predicative {L}ogic and {F}ormal {A}rithmetic.
\newblock {\em Notre Dame Journal of Formal Logic}, 39(1):1--17, 1998.

\bibitem{Dedekind1888}
Richard Dedekind.
\newblock {\em Was sind und was sollen die {Z}ahlen?}
\newblock Vieweg, Braunschweig, 1888.
\newblock Second edition 1893. Reprinted in
  \cite{Dedekind1930}~vol.~3~pp.~335-391.

\bibitem{Dedekind1930}
Richard Dedekind.
\newblock {\em Gesammelte mathematische {W}erke}.
\newblock Vieweg, Braunschweig, 1930-1932.
\newblock Three volumes. Edited by Robert Fricke, Emmy Noether, and $\O$ystein
  Ore.

\bibitem{Demopoulos1995}
William Demopoulos, editor.
\newblock {\em Frege's {P}hilosophy of {M}athematics}.
\newblock Harvard University Press, Cambridge, 1995.

\bibitem{Ferreira2002}
Fernando Ferreira and Kai~F. Wehmeier.
\newblock On the {C}onsistency of the {$\Delta\sp 1\sb 1$}-{CA} fragment of
  {F}rege's {G}rundgesetze.
\newblock {\em Journal of Philosophical Logic}, 31(4):301--311, 2002.

\bibitem{Frege1884aa}
Gottlob Frege.
\newblock {\em Die {G}rundlagen der {A}rithmetik}.
\newblock Koebner, Breslau, 1884.

\bibitem{Gandy1960}
R.~O. Gandy.
\newblock Proof of {M}ostowski's {C}onjecture.
\newblock {\em Bulletin de l'Acad\'emie Polonaise des Sciences. S\'erie des
  Sciences Math\'ematiques, Astronomiques et Physiques}, 8:571--575, 1960.

\bibitem{Ganea2007}
Mihai Ganea.
\newblock Burgess' {$PV$} is {R}obinson's {$Q$}.
\newblock {\em Journal of Symbolic Logic}, 72(2):618--624, 2007.

\bibitem{Hajek1998}
Petr H{\'a}jek and Pavel Pudl{\'a}k.
\newblock {\em Metamathematics of {F}irst-{O}rder {A}rithmetic}.
\newblock Perspectives in Mathematical Logic. Springer, Berlin, 1998.

\bibitem{Hale2001}
Bob Hale and Crispin Wright.
\newblock {\em The {R}eason's {P}roper {S}tudy}.
\newblock Oxford University Press, Oxford, 2001.

\bibitem{Heck1993}
Richard~G. Heck, Jr.
\newblock The {D}evelopment of {A}rithmetic in {F}rege's {G}rundgesetze der
  {A}rithmetik.
\newblock {\em The Journal of Symbolic Logic}, 58(2):579--601, 1993.

\bibitem{Heck1996}
Richard~G. Heck, Jr.
\newblock The {C}onsistency of {P}redicative {F}ragments of {F}rege's
  {G}rundgesetze der {A}rithmetik.
\newblock {\em History and Philosophy of Logic}, 17(4):209--220, 1996.

\bibitem{Heck2000ab}
Richard~G. Heck, Jr.
\newblock Cardinality, {C}ounting, and {E}quinumerosity.
\newblock {\em Notre Dame Journal of Formal Logic}, 41(3):187--209, 2000.

\bibitem{Kleene1959}
Stephen~Cole Kleene.
\newblock Quantification of {N}umber-{T}heoretic {F}unctions.
\newblock {\em Compositio Mathematica}, 14:23--40, 1959.

\bibitem{Lindstrom2003aa}
Per Lindstr\"om.
\newblock {\em Aspects of {I}ncompleteness}, volume~10 of {\em Lecture Notes in
  Logic}.
\newblock Association for Symbolic Logic, Urbana, IL, second edition, 2003.

\bibitem{Linnebo2004}
$\O$ystein Linnebo.
\newblock Predicative {F}ragments of {F}rege {A}rithmetic.
\newblock {\em The Bulletin of Symbolic Logic}, 10(2):153--174, 2004.

\bibitem{MacBride2003aa}
Fraser MacBride.
\newblock Speaking with the {S}hadows: {A} {S}tudy of {N}eo-{L}ogicism.
\newblock {\em The British Journal for the Philosophy of Science},
  54(1):103--163, 2003.

\bibitem{Marker2002}
David Marker.
\newblock {\em Model {T}heory}, volume 217 of {\em Graduate Texts in
  Mathematics}.
\newblock Springer, New York, 2002.

\bibitem{Marker2006ac}
David Marker.
\newblock Introduction to the {M}odel {T}heory of {F}ields.
\newblock In {\em Model {T}heory of {F}ields}, volume~5 of {\em Lecture Notes
  in Logic}, pages 1--37. Association for Symbolic Logic, La Jolla, CA, second
  edition, 2006.

\bibitem{Messmer2006aa}
Margit Messmer.
\newblock Some {M}odel {T}heory of {S}eparably {C}losed {F}ields.
\newblock In {\em Model {T}heory of {F}ields}, volume~5 of {\em Lecture Notes
  in Logic}, pages 135--152. Association for Symbolic Logic, La Jolla, CA,
  second edition, 2006.

\bibitem{Moschovakis1980}
Yiannis~N. Moschovakis.
\newblock {\em Descriptive {S}et {T}heory}, volume 100 of {\em Studies in Logic
  and the Foundations of Mathematics}.
\newblock North-Holland, Amsterdam, 1980.

\bibitem{Odifreddi1989aa}
Piergiorgio Odifreddi.
\newblock {\em Classical {R}ecursion {T}heory. {V}ol. {I}}, volume 125 of {\em
  Studies in Logic and the Foundations of Mathematics}.
\newblock North-Holland, Amsterdam, 1989.

\bibitem{Parsons2008}
Charles Parsons.
\newblock {\em Mathematical Thought and Its Objects}.
\newblock Harvard University Press, Cambridge, 2008.

\bibitem{Poizat2001aa}
Bruno Poizat.
\newblock {\em Stable {G}roups}, volume~87 of {\em Mathematical Surveys and
  Monographs}.
\newblock American Mathematical Society, Providence, RI, 2001.

\bibitem{Sacks1990}
Gerald~E. Sacks.
\newblock {\em Higher {R}ecursion {T}heory}.
\newblock Perspectives in Mathematical Logic. Springer, Berlin, 1990.

\bibitem{Shapiro1991}
Stewart Shapiro.
\newblock {\em Foundations without {F}oundationalism: {A} {C}ase for
  {S}econd-{O}rder {L}ogic}, volume~17 of {\em Oxford Logic Guides}.
\newblock The Clarendon Press, New York, 1991.

\bibitem{Simpson2009aa}
Stephen~G. Simpson.
\newblock {\em Subsystems of {S}econd {O}rder {A}rithmetic}.
\newblock Cambridge University Press, Cambridge, second edition, 2009.

\bibitem{Soare1987}
Robert~I. Soare.
\newblock {\em Recursively {E}numerable {S}ets and {D}egrees}.
\newblock Perspectives in Mathematical Logic. Springer, Berlin, 1987.

\bibitem{Spector1959}
Clifford Spector.
\newblock Hyperarithmetical {Q}uantifiers.
\newblock {\em Fundamenta Mathematicae}, 48:313--320, 1959/1960.

\bibitem{Steel1978}
John~R. Steel.
\newblock Forcing with {T}agged {T}rees.
\newblock {\em Annals of Mathematical Logic}, 15(1):55--74, 1978.

\bibitem{Dries1998}
Lou van~den Dries.
\newblock {\em Tame {T}opology and {O}-{M}inimal {S}tructures}, volume 248 of
  {\em London Mathematical Society Lecture Note Series}.
\newblock Cambridge University Press, Cambridge, 1998.

\bibitem{Visser2009ab}
Albert Visser.
\newblock The {P}redicative {F}rege {H}ierarchy.
\newblock {\em Annals of Pure and Applied Logic}, 160(2):129--153, 2009.

\bibitem{Wehmeier1999}
Kai~F. Wehmeier.
\newblock Consistent {F}ragments of {G}rundgesetze and the {E}xistence of
  {N}on-{L}ogical {O}bjects.
\newblock {\em Synthese}, 121(3):309--328, 1999.

\bibitem{Wehmeier2004aa}
Kai~F. Wehmeier.
\newblock Russell's {P}aradox in {C}onsistent {F}ragments of {F}rege's
  {G}rundgesetze der {A}rithmetik.
\newblock In {\em One {H}undred {Y}ears of {R}ussell's {P}aradox}, volume~6 of
  {\em de Gruyter Series in Logic and its Applications}, pages 247--257. de
  Gruyter, Berlin, 2004.

\bibitem{Wright1983}
Crispin Wright.
\newblock {\em Frege's {C}onception of {N}umbers as {O}bjects}, volume~2 of
  {\em Scots Philosophical Monographs}.
\newblock Aberdeen University Press, Aberdeen, 1983.

\bibitem{Wright1999}
Crispin Wright.
\newblock Is {H}ume's {P}rinciple {A}nalytic?
\newblock {\em Notre Dame Journal of Formal Logic}, 40(1):6--30, 1999.
\newblock Reprinted in \cite{Hale2001}.

\end{thebibliography}

%% Authors are advised to submit their bibtex database files. They are
%% requested to list a bibtex style file in the manuscript if they do
%% not want to use model1a-num-names.bst.

%% References without bibTeX database:

% \begin{thebibliography}{00}

%% \bibitem must have the following form:
%%   \bibitem{key}...
%%

% \bibitem{}

% \end{thebibliography}

\end{document}